%% file: Rahma_Heni.tex
\documentclass[11pt,twoside]{book}
\usepackage{amsfonts}
\usepackage{amsmath}
\usepackage{fancyhdr}
\usepackage{listings}
\usepackage{hyperref}
\hypersetup{bookmarksopenlevel=2}
\usepackage{titlesec}
\usepackage{ifthen}
\usepackage{amsthm} 
\usepackage[utf8]{inputenc} 
\usepackage{geometry}
\usepackage{lipsum}
\usepackage{tikz}
\usepackage{pgfplots}
\pgfplotsset{compat=1.18}
\usepackage{amsmath}
\usepackage{amssymb}
\usepackage{t1enc}
\usepackage[english]{babel}
\input{Commands_and_Macro}
\newtheorem{thm}{Theorem}[section]
\newtheorem{prop}{Proposition}[section]
\newtheorem{lem}{Lemma}[section]
\newtheorem{cor}{Corollary}[section]
\newtheorem{exm}{Example}[section]
\newtheorem{dfn}{Definition}[section]

\newtheorem{rem}{Remark}[section]

\renewcommand{\Re}{\operatorname{R}}
\newcommand{\sectiontitle}[1]{#1}
\newcommand{\subsectiontitle}[1]{#1}

\titleformat{\chapter}[display]
  {\normalfont\huge\bfseries}
  {Appendix \thechapter :}
  {20pt}
  {\Huge}
\titleformat{name=\chapter}[display]
  {\normalfont\huge\bfseries}{Chapter \thechapter :}{20pt}{\Huge}

\begin{document}

\frontmatter

\begin{flushright}
 \textit{ } \\
    \hfill  
 \vspace{0.5cm}
 \textit{ } \\
    \hfill  
 \vspace{0.5cm}
 \textit{ } \\
    \hfill  
 \vspace{0.5cm}
 \textit{ } \\
    \hfill  
 \vspace{0.5cm}
 \textit{ } \\
    \hfill  
 \vspace{0.5cm}
 \textit{ } \\
    \hfill  
 \vspace{0.5cm}
 \textit{ } \\
    \hfill  
 \vspace{0.5cm}
 \textit{ } \\
    \hfill  
 \vspace{0.5cm}
 \textit{ } \\
    \hfill  
 \vspace{0.5cm}
 \textit{ } \\
    \hfill  
 \vspace{0.5cm}

    \textit{``Mathematics is the music of reason.’’} \\
    \hfill  James Joseph Sylvester
    
     \vspace{0.5cm}
    
    \textit{``Mathematics is not only real, but it is the only reality.’’} \\
    \hfill  David Hilbert
  
    \vspace{0.5cm}
    
    \textit{``Pure mathematics is, in its way, the poetry of logical ideas.’’} \\
    \hfill  Albert Einstein
    
    \vspace{0.5cm}
    
   \textit{``It is impossible to be a mathematician without being a poet in soul.’’} \\
    \hfill  Sofia Kovalevskaya
    
    \vspace{0.5cm}
    
    \textit{``Mathematics is the most beautiful and powerful creation of the human spirit.’’} \\
    \hfill  Stefan Banach

\end{flushright}
\newpage

\begin{flushright}
 \textit{ } \\
    \hfill  
 \vspace{0.5cm}
 \textit{ } \\
    \hfill  
 \vspace{0.5cm}
 \textit{ } \\
    \hfill  
 \vspace{0.5cm}
 \textit{ } \\
    \hfill  
 \vspace{0.5cm}
 \textit{ } \\
    \hfill  
 \vspace{0.5cm}
 \textit{ } \\
    \hfill  
 \vspace{0.5cm}
 \textit{ } \\
    \hfill  
 \vspace{0.5cm}
 \textit{ } \\
    \hfill  
 \vspace{0.5cm}
 \textit{ } \\
    \hfill  
 \vspace{0.5cm}
 \textit{ } \\
    \hfill  
 \vspace{0.5cm}

\textit{To my heartbeat, to my parents...}\\

\textit{Rahma Heni}\\
\textit{henirahmahenirahma@gmail.com}

\end{flushright}

\title{Input-to-State Stability of time-varying infinite-dimensional systems.}
\author{Rahma Heni}

\chapter*{Preface}
\thispagestyle{empty}

The concept of input-to-state stability (ISS) proposed in the late 1980s by E. Sontag \cite{sontag1989smooth} is one of the central notions in robust nonlinear control. ISS has become indispensable for various branches of nonlinear systems theory, such as robust stabilization of nonlinear systems \cite{FrK96}, design of nonlinear observers \cite{ArK01}, analysis of large-scale networks \cite{DRW07,JTP94}, etc.
To study the robustness of systems with saturation and limitations in actuation and processing rate (which are typically not ISS), a notion of integral input-to-state stability (iISS) has been proposed in \cite{sontag1998comments}.
The equivalence of the ISS and the existence of an ISS Lyapunov function in a dissipative form was proved for time-invariant ordinary differential equations (ODEs) with Lipschitz's right-hand side in \cite{SoW95}. An analogous Lyapunov characterization was shown for iISS in \cite{ASW00}. An important implication of these results is that for smooth enough systems, ISS implies iISS.
For the overview of the ISS theory of ODEs, we refer to the classic survey \cite{Son08} as well as a recent monograph \cite{Mir23}.

Notable efforts have been undertaken to extend the ISS concept to time-varying systems governed by ODEs in the past years, see \cite{ELW00,KaT04,LWC05}.
In particular, in \cite{ELW00} the Lyapunov characterizations of ISS for time-varying nonlinear systems including
periodic time-varying systems have been studied.
As demonstrated in  \cite[p. 3502]{ELW00}, in contrast to time-invariant ODE systems, in the time-variant case, ISS does not necessarily imply iISS, even for systems with a sufficiently smooth right-hand side.

The success of the ISS theory of ODEs and the need for robust stability analysis of partial differential equations (PDEs) motivated the development of ISS theory in the infinite-dimensional setting, see \cite{DaM13,GGH21,jacob2018infinite,JSZ19,karafyllis2016iss,karafyllis2017iss,MiI15b,MiW18b,prieur2012iss,ZhZ18}.
For instance, the Lyapunov method was used in \cite{MiI15b} for analysis of iISS of nonlinear parabolic equations. In \cite{JMP20,MiW18b}, the notion of a non-coercive ISS Lyapunov function has been introduced and it was
demonstrated that the existence of such a function implies ISS for a broad class of nonlinear infinite-dimensional
systems.
In \cite{jacob2018infinite}, the relation between ISS and iISS has been studied for linear infinite-dimensional systems with a possibly unbounded control operator and inputs in general function spaces. 

In \cite{karafyllis2016iss}, the ISS property for 1-D parabolic PDEs with boundary disturbances has been considered. In \cite{ZhZ18}, ISS of a class of semi-linear parabolic PDEs with respect to boundary disturbances has been studied for the first time based on the Lyapunov method.
For an overview of the ISS theory for distributed parameter systems, we refer to \cite{CKP23,KaK19,MiP20}. ISS of control systems with application to robust global stabilization of the chemostat has been studied with the help of vector Lyapunov functions \cite{KKS08}. In \cite{pisano2017iss}, the variable structure
control approach was exploited to design discontinuous feedback control laws for the ISS stabilization of linear reaction-diffusion-advection equations w.r.t. actuator disturbances. \cite{ASW03} presented a new notion of ISS involving infinity norms of input derivatives. In \cite{JTP94}, the notion of ISS was generalized to input-to-state practical stability (ISpS) for finite-dimensional systems. This concept is extremely useful for sample-data control \cite{noroozi2015semiglobal}, control under quantized control \cite{sharon2011input} and the study of interconnections of nonlinear systems employing small-gain theorems \cite{jiang1996lyapunov,JTP94}. For instance, the authors in \cite{jiang1996lyapunov,JTP94} proved that the interconnection of two ISpS systems remains an
ISpS system where they established a Lyapunov-type nonlinear small-gain
theorem based on the construction of an appropriate Lyapunov function.

As demonstrated in \cite{JTP94} and other works, nonlinear small-gain has led to new solutions to several challenging problems in robust nonlinear control, such as stabilization by partial state and output feedback, nonlinear observers, and robust adaptive tracking. \cite{SoW96} presented characterizations of ISpS via stronger ISS property for time-varying finite-dimensional control systems. On the other hand, an equivalence between the ISpS property and the existence of an ISpS-Lyapunov function for finite-dimensional systems was shown in \cite{SoW95b}. \cite{mironchenko2018criteria} explored the characterizations of ISpS for a broad class of infinite-dimensional systems using the uniform limit property and in terms of input-to-state stability. 
We introduce in Chapter \hyperref[ch:chapter1]{1} the central
notion of input-to-state stability and show how to use Lyapunov functions to verify ISS
of several linear and nonlinear hyperbolic and parabolic systems with distributed and
boundary inputs. We introduce the central notion of input-to-state stability and show how to use Lyapunov functions to verify the (L)ISS/iISS of time-varying infinite-dimensional systems. \\
In Chapter \hyperref[ch:chapter2]{2}, We proceed with a well-posedness analysis of time-varying semi-linear evolution equations with locally Lipschitz continuous nonlinearities and piecewise-right continuous inputs. Lyapunov characterization for ISS of linear time-varying systems is presented in Hilbert spaces. We show also how a non-coercive LISS/iISS Lyapunov function can be constructed for time-varying semi-linear systems. Next, the iISS/ISS analysis of two parabolic PDEs is given to illustrate the proposed method. \\
In Chapter \hyperref[ch:chapter3]{3}, the sufficient conditions for ISpS and iISpS of time-varying infinite-dimensional systems via indefinite Lyapunov functions are proved. Also, we investigate the ISpS w.r.t. inputs from $L_{p}-$spaces of time-varying infinite-dimensional systems.\\
Basic facts about comparison
functions and ISS theory of ODEs are collected in Appendix \hyperref[appendix:A]{A}. For a detailed
treatise of finite-dimensional ISS theory. However, we assume that a
reader has basic knowledge of ordinary and partial differential equations as well as of linear
functional analysis and semigroup theory. Some basic definitions and results from the
theory of ordered Banach spaces and from PDE theory are summarized in Appendix \hyperref[appendix:B]{B}. \\ 
In collaboration with Andrii Mironchenko, Fabian Wirth, Hanen Damak, and Mohamed Ali Hammami, our work titled "Lyapunov methods for input-to-state stability of time-varying evolution equations," slated for submission to ESAIM: Control, Optimisation and Calculus of Variations, in 2024, \cite{damak2024inputandri} establishes that (local) input-to-state stability (L)ISS and integral input-to-state stability (iISS) of time-varying infinite-dimensional systems in abstract spaces follows from the existence of an (L)ISS/iISS Lyapunov function. 
Then, we discuss ISS for linear time-varying control systems in Hilbert spaces with bounded input operators. Methods for the construction of non-coercive LISS/iISS Lyapunov functions are presented for a certain class of time-varying semi-linear evolution equations with a family of unbounded linear operators $\{A(t)\}_{t\geq t_{0}}.$
In collaboration with Hanen Damak and Mohamed Ali Hammami \cite{damak2023input}, we have provided a small-gain theorem for time-varying nonlinear interconnected systems.
In collaboration with Hanen Damak and Mohamed Ali Hammami \cite{damak2023input2}, we have proposed the ISpS Lyapunov methodology to make it suitable for the analysis of ISpS w.r.t. inputs from $L_p$-spaces. The understanding of the nature of ISpS will be beneficial for the
development of quantized and sample-data controllers for infinite-dimensional systems and will give further insights into the ISS theory of infinite-dimensional systems in \cite{MiP20,MiW18b,mironchenko2018lyapunov,mironchenko2019non}. \\

\tableofcontents
\clearpage
\chapter*{Abbreviations and Symbols}
\markboth{Abbreviations and Symbols}{Abbreviations and Symbols}
\vspace{\baselineskip}
\vspace{\baselineskip}
\vspace{\baselineskip}
\vspace{\baselineskip}
\vspace{\baselineskip}
\vspace{\baselineskip}
\vspace{\baselineskip}
\vspace{\baselineskip}
\vspace{\baselineskip}
\vspace{\baselineskip}
\vspace{\baselineskip}
\section*{Abbreviations}
\begin{table}[htbp]
\begin{tabular}{@{}p{2.5cm}p{10cm}}
    \mbox{ODE, ODEs} & Ordinary differential equation(s) \\
    \mbox{PDE, PDEs} & Partial differential equation(s) \\
    FC & Forward complete (Definition \ref{FCsys})\\
    BIC & Boundedness-implies-continuation (Definition \ref{bicpr}) \\
    ISS & \mbox{Input-to-state stability/Input-to-state stable (Definition \ref{issdf})} \\
   LISS & Locally input-to-state stable (Definition \ref{lisssys}) \\
   0-UGAS & \mbox{Uniformly globally asymptotically stable at zero (Definition \ref{uniglobasysze})} \\ 
   eISS & Exponential input-to-state stability (Definition \ref{eissdf}) \\
  ISpS & Input-to-state practically stable (Definition \ref{ispssysy}) \\
   0-UGpAS & \mbox{Uniformly globally practically asymptotically stable at zero (Definition \ref{glasypraasystz})} \\  
   \end{tabular}
\end{table}

\newpage

\begin{table}[htbp]
   \begin{tabular}{@{}p{2.5cm}p{10cm}}
   
  CpUAG & \mbox{Completely practical uniform asymptotic gain property (Definition \ref{cpuagg})} \\  
 iISS & Integral input-to-state stability (Definition \ref{rhan23}) \\
 iISpS & Integral-input-to-state practically stable (Definition \ref{def1}) \\
\end{tabular}

\end{table}
\section*{Symbols}
\subsection*{Sets and Numbers}
\begin{table}[htbp]
   \begin{tabular}{@{}p{2cm}p{10cm}}
$\mathbb{N}$ & Set of natural numbers: $N := \{1, 2, . . .\}$ \\
$\mathbb{Z}$  & Set of integer numbers \\
$\mathbb{R}$  & Set of real numbers \\
$\mathbb{Z}_{+}$ & Set of nonnegative integer numbers $= \mathbb{N} \cup \{0\}$ \\
$\mathbb{R}_{+}$ & Set of nonnegative real numbers \\
$\mathbb{C}$ & Set of complex numbers \\
$S^{n}$ & \mbox{$:= \underbrace{S \times . . . \times S}_{\emph{n \  times}} $} \\
$x^{T}$ & \mbox{Transposition of a vector $x \in  \mathbb{R}^{n}$} \\
$I$ & Identity matrix \\
$|\cdot|$ & \mbox{Euclidean norm in the space $\mathbb{R}^{s} , s \in \mathbb{N}$} \\ 
$\|A\|$ & \mbox{For a matrix $A \in \mathbb{R}^{n\times n}$ we denote $\|A\| := \sup_{x\neq0} 
\frac{|Ax|}{|x|}$} \\
$\mathcal{T}$ & \mbox{Set of all the pairs $(t,s)\in \mathbb{R}_+^{2}$ with $t\geq s$} \\
\end{tabular}
\end{table}

\subsection*{Various Notation}

   \begin{tabular}{@{}p{2cm}p{10cm}}
$B_{r}$ & \mbox{Open ball of radius $r$ around $0 \in \mathbb{R}^{n},$ i.e., $\{x \in \mathbb{R}^{n} : 
|x| < r \}$} \\
$B_{r}(Z)$ & \mbox{Open ball of radius $r$ around $Z \subset \mathbb{R}^{n},$ i.e., $\{x \in 
\mathbb{R}^{n} : dist (x, Z) < r \}$} \\
$B_{r, \mathcal{U}}$ & \mbox{Open ball of radius $r$ around $0$ in a normed vector space $\mathcal{U}$ , 
i.e., $\{u \in \mathcal{U} : \|u\|_{\mathcal{U}} < r \}$}\\
$int A$ & \mbox{(Topological) interior of a set $A$ (in a given topology)}\\
$\bar{A}$ & \mbox{Closure of a set $A$ (in a given topology)} \\
$\mu$ & Lebesgue measure on $\mathbb{R}$ \\
\mbox{$t \rightarrow 0^{-}$} & t converges to $0$ from the left \\
\mbox{$t \rightarrow 0^{+}$} & t converges to $0$ from the right \\
\end{tabular}

\newpage
\subsection*{Sequence Spaces}
\begin{table}[htbp]
   \begin{tabular}{@{}p{2cm}p{10cm}}
$\|x\|_{l^{p}}$    & $:= \big{(}\sum_{k=1}^{\infty} |x_{k}|^{p}\big{)}^{\frac{1}{p}},$ for $p \in 
[1,+\infty)$ and $x = (x_{k})$\\
$\|x\|_{l^{\infty}}$  & $:= \sup_{k\in \mathbb{N}} |x_{k} |,$ for $x = (x_{k})$\\
$l^{p}$ & Space of the sequences $x = (x_{k})$ so that $\|x\|_{l^{p}} < \infty$\\
\end{tabular}
\end{table}

\subsection*{Function Spaces}
\begin{table}[htbp]
   \begin{tabular}{@{}p{2cm}p{10cm}}
   $C(X,U)$ & Linear space of continuous functions from $X$ to $U.$
\par \mbox{We associate with this space the map $u\mapsto\|u\|_{C(X,U)}$ $:=
\sup_{x\in X}\|u(x)\|_{U} ,$}
That may attain $+\infty$ as a value \\
$PC(\mathbb{R}_{+},U)$ & \mbox{The space of globally bounded piecewise right continuous functions from 
$\mathbb{R}_{+}$ to $U$} with the norm $\|u\|_{\mathcal{U}}:=\sup_{0\leq s\leq \infty} \|u(s)\|_{U}.$\\
$PC^{1}(\mathbb{R}_+, U)$ & \mbox{denotes the space of
continuous and piecewise right continuously differentiable functions} from $\mathbb{R}_+$ to $U.$\\
$AC(\mathbb{R}_+,U)$ & \mbox{Space of absolutely continuous functions from $\mathbb{R}_+$ to $U$ with 
$\|u\|_{C(\mathbb{R}_+,U)} <\infty$}\\
$C(X)$ & $:= C(X, X)$\\
$C_{0}(\mathbb{R})$ &\mbox{$:= \{ f \in  C(\mathbb{R}) : \forall\varepsilon > 0 $ there is a compact set 
$K_{\varepsilon} \subset \mathbb{R} : | f (s)| <
\varepsilon \forall s \in \mathbb{R}\setminus K_{\varepsilon}\}$} \\
$C^{k}(\mathbb{R}^{n},\mathbb{R}_{+})$ & \mbox{(Where $k\in\mathbb{N} \cup \{ \infty \}, n\in\mathbb{N})$ 
linear space of $k$ times continuously
differentiable functions} from $\mathbb{R}^{n}$ to $\mathbb{R}_{+}.$\\
$L(X,U)$ & Space of bounded linear operators from $X$ to $U$ \\
$L(X)$ & $:= L(X, X)$ \\
$\|f \|_{\infty}$ & $:= ess \sup_{x\geq0}|f(x)|= \inf_{D\subset\mathbb{R}_{+},\mu(D)=0} 
\sup_{x\in\mathbb{R}_{+}\setminus D}|f(x)|$ \\
$L^{\infty}(\mathbb{R}_{+},\mathbb{R}^{m})$ & Set of Lebesgue measurable functions with $\|f\|_{\infty} < 
\infty$ \\

\end{tabular}
\end{table}
\begin{table}[htbp]
\begin{tabular}{@{}p{2cm}p{10cm}}

$L^{\infty}_{loc}(\mathbb{R}_{+},\mathbb{R}^{m})$ & $:= \{u : \mathbb{R}_{+} \rightarrow \mathbb{R}^{m} : 
u_{t} \in L^{\infty}(\mathbb{R}_{+},\mathbb{R}^{m}) \forall t \in \mathbb{R}_{+}\},$ where
$u_{t} (s) := \left\{
               \begin{array}{ll}
                 u(s), & \text{if } s \in [0,t], \\
                 0, & \text{if } s > t.
               \end{array}
             \right.$

\par \mbox{The set of Lebesgue measurable and locally essentially bounded
functions} \\
$\mathcal{U}$ & Space of input functions. \\
$\|f\|_{L^{p}(0,d)}$ & $:=
\big{(} \int_{0}^{d}
|f (x)|^{p}dx\big{)}^{\frac{1}{p}}, \ p \in [1,+\infty)$ \\
$L^{p}(0, d)$ & \mbox{ Space of $p-$th power integrable functions $f : (0, d) \rightarrow \mathbb{R}$ 
with $\|f\|_{L^{p}(0,d)}<\infty, \ p \in [1,+\infty) $ } \\
\end{tabular}
\end{table}
\newpage
\subsection*{Comparison Functions}
\begin{table}[htbp]
   \begin{tabular}{@{}p{2cm}p{10cm}}
$\mathcal{P}$ & \mbox{$:=\{\gamma:\mathbb{R}_+\to \mathbb{R}_+:\gamma \;\;\hbox{is continuous,} 
\gamma(0)=0\; \hbox{and}\;\gamma(r)>0 \;\hbox{for} \;r>0 \}.$}\\
$\mathcal{K}$  & \mbox{$:=\{\gamma\in \mathcal{P}:\gamma \;\;\hbox{is strictly increasing}\}.$}\\
$\mathcal{K}_{\infty}$ &  \mbox{$:=\{\gamma\in \mathcal{K}:\gamma \;\;\hbox{is unbounded}\}.$}\\
$\mathcal{L}$ & \mbox{$:=\{\gamma:\mathbb{R}_+\to \mathbb{R}_+:\gamma \;\;\hbox{is continuous and 
strictly decreasing with} \displaystyle\lim_{t\to \infty}\gamma(t)=0\}.$}\\
$\mathcal{KL}$ &  \mbox{$:=\{\beta\in C(\mathbb{R}_+\times\mathbb{R}_+,\mathbb{R}_+):\beta(\cdot,t)\in 
\mathcal{K}, \forall t\geq0, \beta(r,\cdot)\in \mathcal{L},\forall r> 0\}.$} \\
\end{tabular}
\end{table}

\mainmatter

\chapter{Introduction}
\label{ch:chapter1}
\markboth{INTRODUCTION}{ }
This chapter defines the notion of input-to-state stability (ISS). Furthermore, we introduce ISS Lyapunov 
functions that constitute the primary tool for the ISS analysis, at least in
the context of nonlinear systems. We show that the existence of an ISS Lyapunov function implies ISS. 
Afterward, we show how the ISS can be verified using the Lyapunov method for several basic partial 
differential equations (PDE), such as transport equation, reaction-diffusion equation with distributed 
and boundary controls, Burgers' equation, etc. In this introductory chapter, we use predominantly 
quadratic ISS Lyapunov functions and utilize integral inequalities in Sobolev spaces to verify 
dissipative inequalities for these functions. Deeper methods for ISS analysis will be developed in the 
subsequent chapters.
\section{\sectiontitle{Time-varying control systems}}

We start with abstract axiomatically defined time-varying systems on the state space $X$ in the fashion of \cite{KaJ11}, and recall that $\mathcal{T}$ is the set of all the pairs $(t,s)\in \mathbb{R}_+^{2}$ with $t\geq s.$								
\begin{dfn}\label{csyol} (\cite{mancilla2023characterization})
Consider a triple $\Sigma=(X,\mathcal{U},\phi),$ consisting of the following
components:
\begin{enumerate}
\item[$(i)$] A normed linear space $X,$ called \emph{the state space}, endowed with the
norm $\|\cdot\|_{X}.$
\item[$(ii)$] A set of admissible input values $U,$ which is a nonempty subset of a certain normed linear space.
\item[$(iii)$] A normed linear space of inputs $\mathcal{U} \subset \{f: \mathbb{R}_+ \to U\}$ endowed with a norm $\|\cdot\|_{\mathcal{U}}.$ We assume that the following axiom hold:

The axiom of shift invariance: for all $u\in\mathcal{U} $ and all $\tau\geq0,$ the time shift $u(\cdot+\tau)$ belongs to $\mathcal{U}$ with $\|u\|_{\mathcal{U}}\geq\|u(\cdot+\tau)\|_{\mathcal{U}}.$

\item[$(iv)$] A set $D_{\phi}\subseteq \mathcal{T}\times X\times
  \mathcal{U}$, and map $\phi:D_{\phi}\to X$, called the transition map.
   For all $(t_0,x_0,u) \in \mathbb{R}_+\times X \times
  \mathcal{U},$ there exists a
  $t_m=t_m(t_0,x_0,u) \in (t_0,\infty) \cup\{\infty\}$ with
 $D_{\phi} \cap \bigl( [t_0,\infty)\times (t_0,x_0, u)\bigr)= [t_0,t_m)\times
 (t_0,x_0, u)$. 
\par The interval $[t_0,t_m)$ is called \emph{the maximal interval of
  existence of $t \mapsto \phi(t,t_0,x_0,u)$}, i.e., of the solution corresponding to the initial condition $x(t_0)=x_0$ and the input $u.$
\end{enumerate}
The triple $\Sigma$ is called \emph{a (control) system} if the following properties hold:
\begin{enumerate}
\item[$(\Sigma_{1})$] \emph{Identity property}: for every $(t_0,x_0,u)\in \mathbb{R}_+\times X\times \mathcal{U}$ it holds that $\phi(t_0,t_0,x_0,u)=x_0.$

\item[$(\Sigma_{2})$] \emph{Causality}: for every $(t, t_0, x_0,u)\in D_{\phi},$ and for every $\tilde{u}\in \mathcal{U}$ such that
$\tilde{u}(\tau)=u(\tau)$ for all $\tau\in [t_0,t),$ it holds that $(t,t_0,x_0,\tilde{u})\in D_{\phi}$ and $\phi(t,t_0,x_0,u)=\phi(t,t_0,x_0,\tilde{u}).$
\item[$(\Sigma_{3})$] \emph{Continuity}: for every $(t_0,x_0,u)\in
  \mathbb{R}_+\times X\times \mathcal{U}$ the map $t \mapsto
  \phi(t,t_0,x_0,u)$ is continuous on its maximal interval of existence $[t_0,t_m(t_0,x_0,u))$.
\item[$(\Sigma_{4})$] \emph{The cocycle property}: for all $t_0\geq0,$ all $x_0\in X,$ all $u\in \mathcal{U}$, and all $t\in [t_0,t_m(t_0,x_0,u)),$ we have
$\phi(t,\tau ,\phi(\tau,t_0,x_0,u),u)=\phi(t,t_0,x_0,u),$ for all $\tau\in [t_0,t].$
\end{enumerate}
\end{dfn}
To relate this definition to other concepts available in the literature, consider the
following concept:
\begin{dfn}\label{strfami}
Let $t_0\geq0.$ Let $W(t,s):X\to X, \ t\geq s\geq t_0$ be a family of bounded linear operators. The family 
$\{W(t,s)\}_{t\geq s\geq t_0}$ is called a \emph{strongly continuous evolution family} if the following 
holds:
\begin{enumerate}
    \item[$(i)$] \emph{Identity at Diagonal}: For all $t\geq t_0$ it holds that $W(t,t) = I.$
    \item[$(ii)$] \emph{Semigroup Property}: For all $t\geq r\geq s\geq t_0$ it holds that 
        $W(t,s)=W(t,r)W(r,s).$
    \item[$(iii)$] \emph{Parameter Continuity}: For all $t\geq s\geq t_0$ and for each $x\in X,$ the 
        map $(t,s)\mapsto W(t,s)x$ is continuous.
    \item[$(iv)$] \emph{Uniform Boundedness}: For all $t\geq s\geq t_0,$ there exists a constant 
        $M\geq0$ such that $$\|W(t, s)\|\leq M.$$ 
\end{enumerate}
\end{dfn}
Take the family $\{W(t,s)\}_{t\geq s\geq t_0}$ as in Definition \ref{strfami}. Setting $\mathcal{U} := 
\{0\},$ and defining $\phi(t,s,x,0) := W(t,s)(x),$ one can see that $\Sigma:= (X,\{0\},\phi)$ is a 
control system according
to Definition \ref{csyol}. Indeed, the axioms $(i)-(iv)$ of Definition \ref{csyol}, as well as causality, 
trivially hold. The identity property, continuity of $\phi$ w.r.t. time, as well as the cocycle property, 
correspond directly to the axioms $(i)-(iii)$ of Definition \ref{strfami}.
Abstract linear control systems, considered in Section 2.1.1, are also a special case
of the control systems that we consider here. In a certain sense, Definition \ref{csyol} is a
direct generalization and a unification of the concepts of strongly continuous nonlinear
semigroups with abstract linear control systems.
This class of systems encompasses control systems generated by ordinary differential
equations (ODEs), switched systems, time-delay systems, evolution PDEs, abstract
differential equations in Banach spaces and many others \cite[Chapter 1]{KaJ11b}.
\section{\sectiontitle{Forward completeness}}
\begin{dfn}\label{FCsys}
We say that a control system (as introduced in Definition \ref{csyol}) is \emph{forward complete}(FC) if 
$D_{\phi}=\mathcal{T}\times X\times\mathcal{U}$, that is, for every $(t_0,x_0,u) \in \mathbb{R}_+\times 
X\times \mathcal{U}$ and for all $t\geq t_0,$ the value $\phi(t,t_0,x_0,u)\in X$ is well-defined.
\end{dfn}
Verification of forward completeness for time-varying nonlinear systems is
often a complex task. In this context, a useful property that is satisfied by many time-varying infinite-dimensional control systems is the possibility to prolong bounded solutions to a larger interval.

For a wide class of control systems boundedness of a solution implies the possibility of
prolonging it to a larger interval, see \cite[Chapter 1]{KaJ11b}. Next, we formulate this property
for abstract systems:
\begin{dfn}(\cite{KaJ11})\label{bicpr}
We say that a control system $\Sigma=(X,\mathcal{U},\phi)$ satisfies the
\emph{boundedness-implies-continuation (BIC)} property if for each
$(t_0,x_0,u)\in \mathbb{R}_+\times X \times \mathcal{U},$ with
$t_m=t_m(t_0,x_0,u) < \infty$, and for every $M>0,$ there exists $t\in [t_0,t_m)$ with $\|\phi(t,t_0,x_0,u)\|_{X}> M.$
 \end{dfn}

\section{\sectiontitle{Input-to-state stability}}
For the formulation of stability properties, we use the standard classes of comparison functions, for which we refer to Appendix \hyperref[appendix:A]{A}. We start with the main concept for this work:
\begin{dfn}\label{issdf}
System $\Sigma=(X,\mathcal{U},\phi)$ is called \emph{input-to-state stable (ISS)}, if it is forward 
complete and there exist $\beta\in\mathcal{KL},$ and $\gamma\in \mathcal{K},$ such that for all $x_0\in 
X,$ all $u\in\mathcal{U},$ all $t_0\geq 0,$ and all $t\geq t_0,$ holds that
\begin{equation}\label{vR1}
\|\phi(t,t_0,x_0,u)\|_{X}\leq \beta(\|x_0\|_{X},t-t_0)+\gamma(\|u\|_{\mathcal{U}}).
\end{equation}
The function $\gamma$ is called gain and describes the influence of input on the system. The
map $\beta$ describes the transient behavior of the system.
\end{dfn}
The definition of ISS can be refined in various directions. For example, in the
definition of ISS, we have assumed in advance that the system is forward complete,
and thus the expression (\ref{vR1}) makes sense.

\begin{prop}
Let $\mathcal{U}=PC(\mathbb{R}_+,U).$ System $\Sigma=(X,\mathcal{U},\phi)$ is ISS, if it is forward 
complete and there exist $\beta\in\mathcal{KL},$ and $\gamma\in \mathcal{K},$ such that for all $x_0\in 
X,$ all $u\in\mathcal{U},$ all $t_0\geq 0,$ and all $t\geq t_0,$ holds that
\begin{equation}\label{ISSPR}
\|\phi(t,t_0,x_0,u)\|_{X}\leq \beta(\|x_0\|_{X},t-t_0)+\gamma(\sup_{t_0\leq s\leq\infty}\|u(s)\|_{U}).
\end{equation}
\end{prop}
\begin{proof}
Sufficiency is clear, since $\sup_{t_0\leq s\leq t}\|u(s)\|_{U}=\sup_{t_0\leq 
s\leq\infty}\|u(s)\|_{U}=\|u\|_{\mathcal{U}}.$ 
\par Now, let $\Sigma$ is ISS. Due to causality property of $\Sigma,$ the state $\phi(\tau,x_0,u),\tau\in 
[t_0,t]$ of $\Sigma$ does not depend on the values of $u(s), s > t.$ For arbitrary $t\geq t_0, x_0\in X$ 
and $u\in\mathcal{U},$ consider another input $\hat{u}\in\mathcal{U},$ defined by
$$\hat{u}(\tau) =
\left\{
  \begin{array}{ll}
    u(\tau), \ \ \ \ \ \ \ if \ \tau\in[t_0,t]; \\
    u(t),  \ \ \ \ \ \ \ if \ \tau>t.
  \end{array}
\right.$$
\par The inequality (\ref{vR1}) holds for all admissible inputs, and then il holds also for $\hat{u}.$
Substituting $\hat{u}$ into (\ref{vR1}) and using that $\|\hat{u}\|_{\mathcal{U}}=\sup_{t_0\leq 
s\leq\infty}\|u(s)\|_{U},$ we get (\ref{ISSPR}). The proof is completed.
\end{proof}


To study the stability properties of time-varying control systems with respect to external inputs, we introduce the following notion.
\begin{dfn}\label{lisssys}
System $\Sigma=(X,\mathcal{U},\phi)$ is called \emph{locally input-to-state stable} \emph{(LISS)} if 
there exist $\rho_{x},\rho_u>0$ and $\beta \in \mathcal{KL}, \gamma \in \mathcal{K},$ such that for all 
$x_0: \|x_0\|_{X}\leq\rho_{x},$ all $t_0\geq 0$, and all $u \in \mathcal{U}:\|u\|_{\mathcal{U}}\leq 
\rho_u$ the trajectory $\phi(\cdot,t_0,x_0,u)$ exists globally and for all $t\geq t_0$ it holds that
\begin{equation}\label{rhm21}
 \|\phi(t,t_0,x_0,u)\|_{X}\leq \beta(\|x_0\|_{X},t-t_0)+\gamma(\|u\|_{\mathcal{U}}).
\end{equation}
The control system is ISS, if in the above definition $\rho_x,\rho_u$ can be chosen to be equal to $\infty.$\\ 
\end{dfn}
ISS trivially implies the following property, which plays a key role in the stability
analysis of the systems without disturbances.
\begin{dfn}\label{uniglobasysze}
System $\Sigma=(X,\mathcal{U},\phi)$ is called \emph{uniformly globally asymptotically stable at zero 
(\emph{0-UGAS})} if there exists $\beta \in \mathcal{KL}$
such that for all $x_0\in X$ and all $t_0\geq 0$ the corresponding trajectory $\phi(\cdot,t_0,x_0,0)$ 
exists globally, and for all $t\geq t_0$ it holds that
\begin{equation}\label{rhm1}
\|\phi(t,t_0,x_0,0)\|_{X}\leq \beta(\|x_0\|_{X},t-t_0).
\end{equation}
\end{dfn}
Substituting $u:=0$ into (\ref{rhm21}), we obtain that ISS systems are always 0-UGAS.\\
Finally, a stronger than ISS property of exponential ISS is of interest:
\begin{dfn}\label{eissdf}
System $\Sigma=(X,\mathcal{U},\phi)$ is called \emph{exponentially input-to-state stable (eISS)} if it is 
forward complete and there exist $M,a>0$ and $\gamma\in \mathcal{K},$ such that for all $x_0\in X,$ all 
$u\in\mathcal{U},$ and all $t\geq t_0,$ it holds that
\begin{equation}\label{eISS_condition}
\|\phi(t,t_0,x_0,u)\|_{X}\leq Me^{-a(t-t_0)}\|x_0\|_{X}+\gamma(\|u\|_{\mathcal{U}}).
\end{equation}
Furthermore, if in addition $\gamma$ can be chosen as a linear function, then we call $\Sigma$ 
exponentially ISS with a linear gain.
\end{dfn}

\section{\sectiontitle{Input-to-state practical stability}}
In this section, we delve into the notion of input-to-state practical stability (ISpS) for control systems, examining conditions under which a system remains stable in the presence of inputs.
\begin{dfn}\label{ispssysy}
The control system $\Sigma=(X, \mathcal{U},\phi)$ is called input-to-state practically stable (ISpS) 
w.r.t. inputs in the space $\mathcal{U},$ if it is forward complete and there exist 
$\beta\in\mathcal{KL},$ $\gamma\in \mathcal{K}$ and $r>0,$ such that for all $x_0\in X,$ all
$u\in\mathcal{U}$ all $t_0\geq0,$ and all $t\geq t_0,$ the following holds
\begin{equation}\label{vR1}
\|\phi(t,t_0,x_0,u)\|_{X}\leq \beta(\|x_0\|_{X},t-t_0)+\gamma(\|u\|_{\mathcal{U}})+r.
\end{equation}
If $\Sigma$ is ISpS and $\beta$ can be chosen as $\beta(s,t)=Me^{-at}s$ for all $s,t\in \mathbb{R}_+,$ 
for some $a,M>0,$ then $\Sigma$ is called exponentially ISpS (eISpS) w.r.t. inputs in the space 
$\mathcal{U}.$\\
If ISpS property w.r.t. inputs in the space $\mathcal{U}$ holds with $r=0,$ then $\Sigma$ is called ISS 
w.r.t. inputs in the space $\mathcal{U}.$
\end{dfn}
\begin{prop} Let $\Sigma=(X, \mathcal{U},\phi)$ be a forward complete control system with 
$\mathcal{U}=PC(\mathbb{R}_+,U).$ System $\Sigma$ is ISpS if and only if there exist 
$\beta\in\mathcal{KL},$ $\gamma\in \mathcal{K}$ and $r>0,$ such that for all $x_0\in X,$ all $u\in 
\mathcal{U},$ all $t_0\geq0$ and all $t\geq t_0$ the following holds
\begin{equation}\label{vR1hh}
\|\phi(t,t_0,x_0,u)\|_{X}\leq \beta(\|x_0\|_{X},t-t_0)+\gamma(\sup _{t_0\leq s\leq t}\|u(s)\|_{U})+r.
\end{equation}
\end{prop}
\begin{proof}
Sufficiency is clear, since $\displaystyle\sup_{t_0\leq s\leq t}\|u(s)\|_{U}\leq \sup _{0\leq s\leq \infty}\|u(s)\|_{U}=\|u\|_{\mathcal{U}}.$ Now, let $\Sigma$ is ISpS. Due to causality property of $\Sigma,$ the state $\phi(\tau,x_0,u),$ $\tau\in [t_0,t]$ of the system $\Sigma$ does not depend on the values of $u(s),s>t.$ For arbitrary $t_0\geq 0,$ $t\geq t_0,$ $x_0\in X$ and $u\in \mathcal{U},$ consider another input $\widehat{u}\in\mathcal{U},$ defined by
$$\widehat{u}(\tau)=\left\{
                       \begin{array}{ll}
                         u(\tau),\quad \tau\in [t_0,t]; \\
                         u(t),\quad \tau>t.
                       \end{array}
                     \right.$$
The inequality (\ref{vR1}) holds for all admissible inputs, and then it holds also for $\widehat{u}.$ Substituting $\widehat{u}$ into (\ref{vR1}) and using that $\|\widehat{u}\|_{\mathcal{U}}=\displaystyle\sup _{t_0\leq s\leq t}\|u(s)\|_{U},$ we get (\ref{vR1hh}).
\end{proof}
We are also interested in the stability properties of system $\Sigma$ in the absence of inputs.
\begin{dfn}\label{glasypraasystz}
 A system $\Sigma=(X, \mathcal{U},\phi)$ is called uniformly globally practically asymptotically stable at zero (0-UGpAS) if it is forward complete and there exist $\beta\in \mathcal{KL}$ and $r>0,$ such
that for all $x_0\in X,$ all $t_0\geq 0$ and all $t\geq t_0,$ it holds that
\begin{equation}
\|\phi(t,t_0,x_0,0)\|_{X}\leq \beta(\|x_0\|_{X},t-t_0)+r.
\label{vRnnn1}
\end{equation}
The estimate (\ref{vRnnn1}) has been derived under an assumption that $\Sigma$ is uniformly globally asymptotically stable (0-UGAS) if the input $u$ is set to zero and $r=0.$
\end{dfn}
\begin{dfn}\label{cpuagg}
We say that the control system $\Sigma$ satisfies the completely practical uniform asymptotic gain property (CpUAG) if there are
$\beta\in\mathcal{KL},$ $\gamma\in \mathcal{K},$ and $c,\varsigma>0,$ such that for all $x_0\in X,$ all $u\in \mathcal{U},$ all $t_0\geq 0$ and all $t\geq t_0,$ it holds that
$$\|\phi(t,t_0,x_0,u)\|_{X}\leq \beta(\|x_0\|_{X}+c,t-t_0)+\gamma(\|u\|_{\mathcal{U}})+\varsigma.$$
\end{dfn}

\section{\sectiontitle{iISS and iISpS}}
\begin{dfn}\label{rhan23}
Let $\Sigma=(X, \mathcal{U},\phi)$ be a forward complete control system with $\mathcal{U}=PC(\mathbb{R}_+,U).$ System $\Sigma$ is called \emph{integral input-to-state stable (iISS)} if there exist $\alpha \in \mathcal{K}_{\infty},$ $\mu\in \mathcal{K}$ and $\beta\in\mathcal{KL}$, such that for all $x_0\in X,$ all $u\in \mathcal{U},$ all $t_0\geq 0,$ and all $t\geq t_0$ the following holds
\begin{equation}\label{EQua1in}
\|\phi(t,t_0,x_0,u)\|_{X}\leq \beta(\|x_0\|_{X},t-t_0)+\alpha \bigg{(}\int_{t_0} ^{t} \mu(\|u(s)\|_{U})ds\bigg{)}.
\end{equation}
\end{dfn}
The ISS estimate provides an upper bound of the system’s response to the maximal
magnitude of the applied input. In contrast to that, integral ISS gives an upper bound
of the response of the system with respect to a kind of energy fed into the system,
described by the integral in the right-hand side of (\ref{EQua1in}).
\begin{dfn}\label{def1}
A system $\Sigma=(X, \mathcal{U},\phi)$ is called integral-input-to-state practically stable (iISpS) if it is forward complete and there exist $\alpha \in \mathcal{K}_{\infty},$ $\mu\in \mathcal{K},$ $\beta\in\mathcal{KL}$ and $r>0,$ such that for all $x_0\in X,$ all $u\in C(\mathbb{R}_+,U),$ all $t_0\geq 0$ and all $t\geq t_0,$ the following holds
$$\|\phi(t,t_0,x_0,u)\|_{X}\leq \beta(\|x_0\|_{X},t-t_0)+\alpha\left(\int_{ t_0} ^{t} \mu(\|u(s)\|_{U})ds\right)+r.$$
If iISpS property holds with $r=0,$ then  $\Sigma$ is called integral input-to-state stable (iISS).
\end{dfn}



\section{\sectiontitle{ISS Lyapunov functions}}
In this section, we recall the concept of ISS Lyapunov functions. These are crucial for the verification of input-to-state stability of time-varying control systems.

The Lie derivative of $V \in C(\mathbb{R}_+\times X, \mathbb{R}_+)$ at $(t,x)\in \mathbb{R}_+ \times X$ corresponding to the input $u\in \mathcal{U}$ along the corresponding trajectory of $\Sigma=(X,\mathcal{U},\phi)$ is defined by
$$
\dot{V}_{u}(t,x)=D^{+}V(s,\phi(s,t,x,u))|_{s=t}:=\limsup\limits_{h\to0^+}\frac{1}{h}\Big(V(t+h,\phi(t+h,t,x,u))-V(t,x)\Big),
$$
where $D^{+}$ means the right upper Dini derivative.

Recall the notation $B_{r}:=\{x\in X: \|x\|_{X}\leq r\}$.
\begin{dfn}\label{DEFINITION2}
Let $\Sigma=(X,\mathcal{U},\phi)$ be a time-varying control system,
$D\subset X$ with $0\in \intt(D)$. A continuous function $V:\mathbb{R}_+\times D \to \mathbb{R}_+$ 
is called a non-coercive \emph{LISS Lyapunov function in implication form}
for the system $\Sigma$, if there exist $\alpha_2\in\mathcal{K}_{\infty}$, $\kappa\in\mathcal{K},$ constants $r_{1},r_{2}>0$ and a positive definite function $\mu,$ such that $B_{r_1}\subset D,$ and for all $t\geq 0,$ $V(t,0)=0,$ and

\begin{equation}\label{R1EN}
0<V(t,x)\leq \alpha_2(\|x\|_{X}) \quad \forall x\in D\backslash\{0\} \quad \forall t\geq0,
\end{equation}
and for all $t\geq0,$ all $x\in X:$ $\|x\|_{X}\leq r_{1},$ all $u\in \mathcal{U}$: $\|u\|_{\mathcal{U}}\leq r_{2}$ it holds that
\begin{equation}\label{STAvNNv}
\|x\|_{X}\geq\kappa(\|u\|_{\mathcal{U}}) \quad\Longrightarrow\quad \dot{V}_{u}(t,x)\leq -\mu(V(t,x)).
\end{equation}
If additionally, there is $\alpha_1\in\mathcal{K}_{\infty},$ so that
\begin{equation}\label{R1E}
 \alpha_1(\|x\|_{X}) \leq V(t,x)\leq \alpha_2(\|x\|_{X}) \quad \forall x\in D \quad \forall t\geq0,
\end{equation}
then $V$ is called a (coercive) LISS Lyapunov function in implication form
for $\Sigma.$ The function $\kappa$ is called ISS Lyapunov gain for $\Sigma.$

If in the previous definition $D=X,$ $r_1=\infty,$ and $r_2=\infty,$ then
$V$ is called a  non-coercive (coercive, if \eqref{R1E} holds) ISS Lyapunov function in implication form for $\Sigma.$
\end{dfn}

Since Definition \ref{DEFINITION2} presumes that $V$ is merely continuous, the map $t\mapsto V(t,\phi(t,t_0,x_0,u))$ is only continuous in $t$.

The importance of ISS Lyapunov functions is due to the following
result. The proof follows ideas already developed in \cite[p. 441]{Son89} and
extended to an infinite-dimensional, time-invariant setting in
\cite[Theorem 1.5.4]{Mir23b}.
\begin{thm}\label{hjMMMMMljg}
Let $\Sigma=(X,\mathcal{U},\phi)$ be a time-varying control system
satisfying the BIC property. If there exists a
coercive (L)ISS Lyapunov function $V$ for $\Sigma$, then $\Sigma$ is (L)ISS.
\end{thm}

\begin{proof}
Let $V$ be a coercive LISS Lyapunov function for
$\Sigma=(X,\mathcal{U},\phi)$ and let $\alpha_1,\alpha_2,\kappa,\mu,r_1,r_2$
be as in Definition \ref{DEFINITION2}.
Define $\rho_1:=\min\{\alpha_2^{-1}\circ\alpha_1(r_1),\alpha_1(r_1)\}$ and $\rho_2:=\min\{r_2,\kappa^{-1}\circ\alpha_2^{-1}(\rho_1)\}.$
Take an arbitrary control $u\in \mathcal{U},$ with
$\|u\|_{\mathcal{U}}\leq\rho_2,$ and define for $t\geq 0$
$$\Lambda_{u,t}:=\{x\in D: V(t,x)\leq \alpha_2\circ\kappa(\|u\|_{\mathcal{U}})\}.$$

Since $\|u\|_{\mathcal{U}}\leq\rho_2$,  for any $t\geq0,$ and any $x\in
\Lambda_{u,t},$ it holds using \eqref{R1E} that
$$\alpha_1(\|x\|_{X}) \leq V(t,x)\leq \alpha_2\circ\kappa(\|u\|_{\mathcal{U}})\leq\alpha_2\circ\kappa(\rho_2)\leq \rho_1,$$
and so $\|x\|_{X}\leq \alpha_1^{-1}(\rho_1) \leq r_1.$ On the other hand,
if $\|x\|_X\leq\kappa(\|u\|_{\mathcal{U}})$ then for all $t\geq 0$ we
have $V(t,x) \leq \alpha_2(\|x\|_X) \leq
\alpha_2\circ\kappa(\|u\|_{\mathcal{U}})$ and so $x\in \Lambda_{u,t}$.
This implies that $ B_{\kappa(\|u\|_{\mathcal{U}})}\subset \Lambda_{u,t}\subset B_{r_1}\subset D$ for all $t\geq0$.

\noindent {\bf{Step 1:}} We are going to prove that the family $\{\Lambda_{u,t}\}_{t\geq0}$ is forward invariant in the following sense:
\begin{equation}\label{impl1}
 t_0\in\mathbb{R}_+ \wedge x_0\in\Lambda_{u,t_0} \wedge \;
 \vartheta\in\mathcal{U}\wedge\;
 \|\vartheta\|_{\mathcal{U}}\leq\|u\|_{\mathcal{U}}\qrq \forall t\geq t_0
 : \phi(t,t_0,x_0,\vartheta)\in\Lambda_{u,t}.
\end{equation}
Suppose that this implication does not hold for certain $t_0\geq0,$ $x_0\in\Lambda_{u,t_0}$ and $ \vartheta\in\mathcal{U}$ satisfying
$\|\vartheta\|_{\mathcal{U}}\leq\|u\|_{\mathcal{U}}.$ There are two alternatives.


First, it can happen that the maximal existence time $t_m(t_0,x_0,\vartheta)$ of the map $\phi(\cdot,t_0,x_0,\vartheta)$ is finite and $\phi(t,t_0,x_0,\vartheta)\in\Lambda_{u,t}$ for $t\in[t_0,t_m(t_0,x_0,\vartheta)).$ 

However, $\Lambda_{u,t}\subset B_{r_1}$ for all $t\geq0$, and thus the trajectory is uniformly bounded on its whole domain of existence.
Then $t_m(t_0,x_0,\vartheta)$ cannot be the maximal existence time, as the trajectory can be prolonged to a larger interval thanks to the BIC property of $\Sigma.$

 
The second alternative is that there exist $t\in(t_0,t_m(t_0,x_0,\vartheta))$ and $\epsilon>0$ so that 
\[
V(t,\phi(t,t_0,x_0,\vartheta))>\alpha_2\circ\kappa(\|u\|_{\mathcal{U}})+\epsilon.
\]
Let
$$\hat{t}:=\inf\{t\in[t_0,t_m(t_0,x_0,\vartheta)):\; V(t,\phi(t,t_0,x_0,\vartheta))>\alpha_2\circ\kappa(\|u\|_{\mathcal{U}})+\epsilon\}.$$
As $x_0\in\Lambda_{u,t_0}$ and $t\mapsto\phi(t,t_0,x_0,\vartheta)$ is
continuous, we obtain with \eqref{R1E} that
$$\alpha_2(\|\phi(\hat{t},t_0,x_0,\vartheta)\|_{X})\geq V(\hat{t},\phi(\hat{t},t_0,x_0,\vartheta))=\alpha_2\circ\kappa(\|u\|_{\mathcal{U}})+\epsilon\geq\alpha_2\circ\kappa(\|\vartheta\|_{\mathcal{U}}).$$


Due to (\ref{STAvNNv}), it holds that
$$D^{+}V(\hat{t},\phi(\hat{t},t_0,x_0,\vartheta))\leq -\mu(V(\hat{t},\phi(\hat{t},t_0,x_0,\vartheta))).$$


Denote $\omega(t)=V(t,\phi(t,t_0,x_0,\vartheta)).$ Suppose that there is
$(t_k)\subset(\hat{t},\infty),$ such that $t_k\longrightarrow \hat{t}$ as
$k\longrightarrow \infty$ and $\omega(t_k)\geq \omega(\hat{t})$ for all
$k\in \N$. It follows that,
$$0 \leq \limsup\limits_{k\to\infty}\frac{\omega(t_k)-\omega(\hat{t})}{t_k-\hat{t}}\leq D^{+}\omega(\hat{t})\leq-\mu(\omega(\hat{t}))<0,$$
a contradiction. Thus, there is some $\delta>0$ such that $\omega(t)<\omega(\hat{t})=\alpha_2\circ\kappa(\|u\|_{\mathcal{U}})+\epsilon$ for all $t\in(\hat{t},\hat{t}+\delta).$ This contradicts the definition of $\hat{t}.$ The statement (\ref{impl1}) is shown.
We note for further reference that for any initial time $t_0\geq0$ and any initial condition $x_0\in\Lambda_{u,t_0}$ we have
$$\alpha_1(\|\phi(t,t_0,x_0,u)\|_{X})\leq V(t,\phi(t,t_0,x_0,u))\leq\alpha_2\circ\kappa(\|u\|_{\mathcal{U}}) \quad \forall t\geq t_0,$$
which implies that for all $t_0\geq0$ and all $t\geq t_0$
\begin{equation}\label{inequa1}
\|\phi(t,t_0,x_0,u)\|_{X}\leq\gamma(\|u\|_{\mathcal{U}}),
\end{equation}
where $\gamma=\alpha_1^{-1}\circ\alpha_2\circ\kappa.$

\noindent {\bf{Step 2:}} 
Recall that $u\in \mathcal{U},$ with $\|u\|_{\mathcal{U}}\leq\rho_2.$ Pick
an arbitrary initial time $t_0\geq0$ and any initial condition $x_0$ with
$\|x_0\|_{X}\leq\rho_{1}$. The case $x_0\in \Lambda_{u,t_0}$ has been treated in Step 1 of this proof. In Step 2, we consider the case $x_0\notin\Lambda_{u,t_0}$.  Due to the continuity of the flow map $t\mapsto\phi(t,t_0,x_0,u)$ and of the Lyapunov function $V$, there exists $\tau>0$ such that $\phi(t,t_0,x_0,u)\notin\Lambda_{u,t}$ for all $t\in [t_0,t_0+\tau].$

Using the axiom of shift invariance $(iii)$ of $\Sigma$ and employing the
sandwich bounds \eqref{R1E}, we get 
\begin{equation*}
\alpha_2(\|\phi(t,t_0,x_0,u)\|_{X})\geq
V(t,\phi(t,t_0,x_0,u))\geq\alpha_2\circ\kappa(\|u\|_{\mathcal{U}}),\quad t
\in [t_0,t_0+\tau],
\end{equation*}
and using that $\alpha_2\in\mathcal{K}_{\infty},$ we obtain that $\|\phi(t,t_0,x_0,u)\|_{X}\geq\kappa(\|u\|_{\mathcal{U}})$ for $t\in[t_0,t_0+\tau].$

As long as $\|\phi(t,t_0,x_0,u)\|_{X}\leq r_1,$ and $t \in
[t_0,t_0+\tau]$,  we obtain from \eqref{STAvNNv} that
\begin{equation}\label{INE11}
D^{+}V(t,\phi(t,t_0,x_0,u))=\dot{V}_{u}(t,\phi(t,t_0,x_0,u))\leq -\mu(V(t,\phi(t,t_0,x_0,u))).
\end{equation}
Then, for all such $t,$ it holds that $V(t,\phi(t,t_0,x_0,u))\leq V(t_0,x_0),$ which implies that
$$\|\phi(t,t_0,x_0,u)\|_{X}\leq\alpha_1^{-1}\circ\alpha_2(\|x_0\|_{X})\leq\alpha_1^{-1}\circ\alpha_2(\rho_1)\leq r_1.$$
Thus, (\ref{INE11}) is valid for all $t_0\geq 0$ and all $t\in [t_0,t_0+\tau].$\\
Since the map $t \mapsto V(t,\phi(t,t_0,x_0,u))$ is continuous, we can use Lemma \ref{lemmm1} to infer that
there exists
$\hat{\beta}\in \mathcal{KL},$ such that $$V(t,\phi(t,t_0,x_0,u))\leq
\hat{\beta}(V(t_0,x_0),t-t_0),\quad t\in[t_0,t_0+\tau] $$
and hence
\begin{equation}\label{EE11}
\|\phi(t,t_0,x_0,u)\|_{X}\leq \beta(\|x_0\|_{X},t-t_0)\quad  \quad  t\in[t_0,t_0+\tau],
\end{equation}

where $\beta(s,t)=\alpha_1^{-1}\circ \hat{\beta}(\alpha_2^{-1}(s),t)$, $s,t\geq0.$
The estimate (\ref{EE11}) together with the BIC property ensure that the solution can be
prolonged onto an interval that is larger than $[t_0,t_0+\tau],$ and if on this interval $\phi(t,t_0,x_0,u) \notin \Lambda_{u,t},$ then the estimate (\ref{EE11}) holds at this larger interval. This indicates
that $\phi(t,t_0,x_0,u)$ exists at least until the time when it intersects $\Lambda_{u,t}.$
As $ B_{\kappa(\|u\|_{\mathcal{U}})}\subset \Lambda_{u,t}$ for all
$t\geq0$, the estimate \eqref{EE11} implies that there exists $\bar{t},$ such that
$$\bar{t}=\inf\{t\geq t_0:\phi(t,t_0,x_0,u)\in \Lambda_{u,t}\}.$$
By the above reasoning, (\ref{EE11}) holds on $[t_0,\bar{t}].$

From the forward invariance of the family  $\{\Lambda_{u,t}\}_{t\geq0}$, the estimate (\ref{inequa1}) is valid. Hence, from (\ref{inequa1}) and (\ref{EE11}) we conclude that
$$\|\phi(t,t_0,x_0,u)\|_{X}\leq\beta(\|x_0\|_{X},t-t_0)+\gamma(\|u\|_{\mathcal{U}})\quad \forall t_0\geq0\quad  \forall t\geq t_0.$$
Our estimates hold for an arbitrary control $u:$ $\|u\|_{\mathcal{U}}\leq\rho_2.$ Then, combining (\ref{inequa1}) and (\ref{EE11}), we obtain the LISS estimate for $\Sigma$ in $(t,t_0,x_0,u)\in[t_0,\infty)\times\mathbb{R}_+\times B_{\rho_1}\times B_{\rho_2,\mathcal{U}}.$

To prove that the existence of an ISS Lyapunov function ensures that
$\Sigma$ is ISS, one has to argue as above but with $r_1=r_2=\infty.$
\end{proof}

We will now define the notion of (non)-coercive iISS Lyapunov function and prove that the existence of a coercive iISS (resp. coercive ISS) Lyapunov function implies iISS (resp. ISS).
\begin{dfn}
Consider a time-varying control system $\Sigma=(X,\mathcal{U},\phi)$ with the input space $\mathcal{U}=PC(\mathbb{R}_+, U).$
A continuous function $V:\mathbb{R}_+\times X \to \mathbb{R}_+$ is called a non-coercive iISS Lyapunov function for $\Sigma$, if there exist $\alpha_2\in \mathcal{K}_{\infty},$ $\eta\in \mathcal{P}$ and $\chi\in \mathcal{K},$ such that for all $t\geq 0,$ $V(t,0)=0,$ (\ref{R1EN}) holds
and the Lie derivative along the trajectories of the system $\Sigma,$ satisfies
\begin{equation}\label{R2E}
\dot{V}_{u}(t,x)\leq -\eta(\|x\|_{X})+\chi(\|u(t)\|_{U}),
\end{equation}
for all $x\in X,$ all $u\in \mathcal{U},$ and all $t\geq 0.$

If additionally, there is $\alpha_1\in\mathcal{K}_{\infty},$ so that (\ref{R1E}) holds,
then $V$ is called a (coercive) iISS Lyapunov function for $\Sigma.$
Furthermore, if
\begin{equation}\label{R1Ev}
\displaystyle\lim _{\tau \to \infty} \eta(\tau)=\infty \; \;or \; \;\liminf _{\tau \to \infty}\eta(\tau)\geq \lim _{\tau \to \infty}\chi(\tau)
\end{equation}
holds, then $V$ is called an ISS Lyapunov function for $\Sigma$ in dissipative form.
\end{dfn}

The next theorem underlines the importance of iISS and ISS Lyapunov functions.
\begin{thm}\label{hjljg}
Let $\mathcal{U}=PC(\mathbb{R}_+, U)$.
Consider a time-varying control system $\Sigma=(X,\mathcal{U},\phi),$ satisfying the BIC property. If there is a coercive \emph{iISS} (resp. coercive ISS) Lyapunov function in dissipative form for $\Sigma,$ then $\Sigma$ is iISS (resp. ISS).
\end{thm}

\begin{proof}
Let $V$ be a coercive iISS Lyapunov function with associated functions $\alpha_1,\alpha_2,\eta,\chi$.
As $\eta\in \mathcal{P},$ it follows from \cite[Proposition
B.1.14]{mironchenko2023input} that there are $\varrho\in\mathcal{K}_{\infty}$ and $\xi\in\mathcal{L}$ such that $\eta(\tau)\geq\varrho(\tau)\xi(\tau)$ for all $\tau\geq0.$ Pick any $\tilde{\eta}\in\mathcal{P}$ satisfying
$$\varrho(\alpha_2^{-1}(\tau))\xi(\alpha_1^{-1}(\tau))\geq \tilde{\eta}(\tau), \quad \tau\geq 0.$$
From (\ref{R1E}) and (\ref{R2E}), we obtain for any $x\in X,$ any $t\geq0$ and any $u\in\mathcal{U}$ that
\begin{align*}
\dot{V}_{u}(t,x)&\leq -\eta(\|x\|_{X})+\chi(\|u(t)\|_{U})\\
&\leq-\varrho(\|x\|_{X})\xi(\|x\|_{X})+\chi(\|u(t)\|_{U}).\\
\intertext{Using $\alpha_2^{-1}(V(t,x)) \leq \|x\|_{X} \leq
  \alpha_1^{-1}(V(t,x))$ we can continue } 
&\leq-\varrho\circ\alpha_2^{-1}(V(t,x))\xi\circ\alpha_1^{-1}(V(t,x))+\chi(\|u(t)\|_{U})\\
&\leq-\tilde{\eta}(V(t,x))+\chi(\|u(t)\|_{U}).
\end{align*}
Now take any input $u\in\mathcal{U},$ any initial time $t\geq0,$ any initial condition $x_0\in X,$ and consider the corresponding solution $\phi(\cdot,t_0,x_0,u)$ defined on $[t_0,t_m(t_0,x_0,u)).$

As $V$ is continuous and $\phi$ is continuous in $t,$ the map $t\mapsto V(t,\phi(t,t_0,x_0,u))$ is continuous, and for all $t_0\geq0$ and all $t\in[t_0,t_m(t_0,x_0,u)),$ it holds that
$$D^{+}V(t,\phi(t,t_0,x_0,u))\leq-\tilde{\eta}(V(t,x))+\chi(\|u(t)\|_{U}).$$
By Corollary \ref{cor1}, there exists a $\hat{\beta}\in \mathcal{KL},$ so that
$$V(t,\phi(t,t_0,x_0,u))\leq \hat{\beta}(V(t_0,x_0),t-t_0)+2\int _{t_0} ^{t} \chi(\|u(s)\|_{U})ds, \quad t_0\geq0, \quad t\in[t_0,t_m(t_0,x_0,u)).$$
With (\ref{R1E}), it holds that
\begin{equation}\label{wxcvs}
\alpha_1(\|x(t)\|_{X})\leq \tilde{\beta}(\|x(t_0)\|_{X},t-t_0)+2\int _{t_0} ^{t} \chi(\|u(s)\|_{U})ds, \quad t_0\geq0, \quad t\in[t_0,t_m(t_0,x_0,u)),
\end{equation}
where
$\tilde{\beta}(s,t)=\hat{\beta}(\alpha_2(s),t)\in \mathcal{KL}\quad \forall s,t\geq0 .$
Since $\alpha_1$ is of class $\mathcal{K}_{\infty}$, it satisfies the triangle inequality $\alpha_1^{-1}(a+b)\leq\alpha_1^{-1}(2a)+\alpha_1^{-1}(2b)$ for any $a,b\in\mathbb{R}_{+},$ we have iISS as in Definition \ref{rhan23} with $\alpha \in \mathcal{K}_{\infty},$ $\mu\in \mathcal{K}$ and $\beta\in\mathcal{KL},$ for all $t_0\geq 0$ and all $t\in [t_0,t_m(t_0,x_0,u))$ with $\alpha(s)=\alpha_1^{-1}(2s),$ $\mu(s)=2 \chi(s)$ and $\beta(s,r)=\alpha_1^{-1}(2\tilde{\beta}(s,r)).$
Estimate (\ref{wxcvs}) together with the BIC property, implies that $t_m(t_0,x_0,u)= \infty,$ and hence
system $\Sigma$ is iISS.

Suppose that $V$ is an ISS Lyapunov function in dissipative form, i.e., (\ref{R2E}) holds. Arguing as in the proof of \cite[Theorem $3.2$]{damak2021input}, we obtain that
$V$ is an ISS Lyapunov function of $\Sigma$ in implication form. Consequently, $\Sigma$ is ISS.
\end{proof}

\chapter{Lyapunov Methods for ISS of Time-Varying Evolution Equations}
\label{ch:chapter2}
\markboth{LYAPUNOV METHODS FOR ISS OF TIME-VARYING EVOLUTION EQUATIONS}{ }

\textbf{Time-varying infinite-dimensional systems.}
 In \cite{ZhZ18}, ISS of a class of semi-linear parabolic PDEs with respect to boundary disturbances has been studied for the first time based on the Lyapunov method.
For an overview of the ISS theory for distributed parameter systems, we refer to \cite{CKP23,KaK19,MiP20}. Much less attention has been devoted to the ISS
of time-varying infinite-dimensional systems \cite{damak2021input,hanen2022input}.
For instance, in \cite{damak2021input}, a direct Lyapunov theorem was shown for a class of time-varying semi-linear evolution equations in Banach spaces with Lipschitz continuous nonlinearities.

The analysis of ISS for time-varying systems is more involved than that of time-invariant systems, even for linear systems. Consider an abstract Cauchy problem of the form

$$\dot{x}(t)=A(t)x(t),\quad x(t_0)=x_0, \quad t\geq t_0. $$

If $A$ does not depend on time, Hille-Yosida Theorem \cite[Theorem 3.1]{Paz83} gives a characterization of the infinitesimal
generators of strongly continuous semigroups, and hence, it characterizes all
well-posed autonomous abstract Cauchy problems.
At the same time, there is no extension of the Hille-Yosida theorem to the time-variant case.
The conditions that ensure that $\{A(t)\}_{t\geq0}$ generates an evolution family are well-understood if $\{A(t)\}_{t\geq0}$ is a family of bounded operators. However, when the operators $A(t)$ are unbounded, it is a delicate matter to prove the well-posedness of an abstract Cauchy problem. We refer the reader to
\cite{carmen,Paz83} and the references therein for more information. There are almost no results on ISS and iISS of time-varying systems with unbounded $A(t)$.

\textbf{Contribution.} Motivated by the foregoing discussion, this chapter investigates the (L)ISS/iISS property of time-varying infinite-dimensional systems in terms of (L)ISS/iISS Lyapunov functions. We develop conditions for the well-posedness of time-varying semi-linear evolution equations in Banach spaces. We show the stability results for evolution operators in Banach spaces. On the other hand, we indicate that for each input-to-state stable linear time-varying control system with a bounded linear part, there is a coercive Lyapunov function. Also, we develop a method for the construction of a non-coercive ISS Lyapunov function for a class of linear time-varying control systems where the operators in the family $\{A(t)\}_{t\geq0}$ are unbounded. Next, we derive a novel construction of non-coercive LISS/iISS Lyapunov functions for a certain class of time-varying semi-linear systems with the associated nominal system being linear and unbounded that depends on the time.
\section{\sectiontitle{Time-varying semi-linear evolution equations}}
In the previous section a general ISS formalism for abstract control
systems has been  introduced. We now proceed to the analysis of time-varying semi-linear systems in Banach spaces.

Let $X$ and $U$ be Banach spaces, and the input space be $\mathcal{U}=PC(\mathbb{R}_{+}, U).$ We consider time-varying semi-linear evolution equations of the form:
\begin{equation}
\label{R1}
\left\lbrace
\begin{array}{l}
\dot{x}(t)=A(t)x(t)+\Psi(t,x(t),u(t)),\qquad t\ge t_0\ge0,\\
x(t_0)=x_0,
\end{array}\right.
\end{equation}
where $x\in X$ is the state of the system, $u\in \mathcal{U}$ is the
control input, $t_0 \ge 0$ is the initial time, $x_0$ is the initial
state, $\Psi:\mathbb{R}_{+}\times X\times U \to X$ is a nonlinear map for
which we will specify the assumptions below. In addition,
$\{A(t):D(A(t)) \subset X \rightarrow X\}_{t \geq t_0\geq 0}$ is a family
of linear operators such that for all $t\geq 0,$ the domain $D(A(t))$ of
$A(t)$ is
 dense in $X.$ We assume that
$\{A(t)\}_{t\geq0}$ generates a strongly continuous evolution family
$\{W(t,s)\}_{t\geq s\geq0},$ that is, for all $t\geq s\geq t_0\ge0$ there
exists a bounded linear operator $W(t,s):X\to X$ and the following
properties hold:
\begin{enumerate}
\item[$(i)$] $W(s,s)=I,$ \quad $W(t,s)=W(t,r)W(r,s)$ for all $t\geq r\geq s\geq t_0.$
\item[$(ii)$] $(t,s)\mapsto W(t,s)$ is strongly continuous for $t\geq s\geq t_0,$ that is for each $x \in X$ the mapping $(t,s)\mapsto W(t, s)x$ is continuous.
\item[$(iii)$] For all $t\geq s\geq t_0$ and all $\upsilon\in D(A(s)),$ we
  have $W(t,s)v \in D(A(t))$ and 
$$\frac{\partial}{\partial t}\big(W(t,s)\upsilon\big)=A(t)W(t,s)\upsilon,$$
$$\frac{\partial}{\partial s}\big(W(t,s)\upsilon\big)=-W(t,s)A(s)\upsilon.$$
\end{enumerate}
\begin{rem}
Provided it exists, the evolution family $\{W(t,s)\}_{t\geq s\geq0}$ is uniquely determined by the family $\{A(t)\}_{t\geq0},$ see e.g., \cite[p. 58]{carmen}.
\end{rem}
In the next definition, we extend the concept of a classical solution, as defined in \cite[Definition 2.1 p. 105]{Paz83}, to encompass piecewise classical solutions for (\ref{R1}).
\begin{dfn}\label{rrnnann23}
Let $x_0\in X,$ $u\in\mathcal{U}$ and $t_0\ge 0$ be given. A function $x:[t_0,\infty)\rightarrow X$ is called a \emph{piecewise classical solution} of (\ref{R1}) on $[t_0,\infty)$ if $x(\cdot)$ is continuous on $[t_0,\infty),$ piecewise continuously differentiable on $(t_0,\infty)$ and $x(t)$ satisfies (\ref{R1}) 
at all points of differentiability.
\end{dfn}

\subsection{\subsectiontitle{Semilinear evolution equations as control systems}}
In this section, we show under appropriate Lipschitz continuity
assumptions that the semi-linear system (\ref{R1}) generates a well-posed control system.
\begin{dfn}
We call $\Psi:\mathbb{R}_+\times X\times U\longrightarrow X$ \emph{locally Lipschitz continuous in $x,$ uniformly in $t$ and $u$ on bounded subsets} if for every $\tilde{t}\geq0$ and $c\geq 0,$ there is a constant $K(c,\tilde{t}),$ such that for all $x,y\in X:\|x\|_{X},\|y\|_{X}\leq c$, all $t\in [0,\tilde{t}]$, and all $u\in \mathcal{U}$: $\|u\|_{\mathcal{U}}\leq c,$ it holds that
\begin{equation}\label{bvbnbpp}
\|\Psi(t,x,u)-\Psi(t,y,u)\|_{X}\leq K(c,\tilde{t})\|x-y\|_{X}.
\end{equation}
\end{dfn}

The following assumption will be needed in the remainder of the paper
\begin{description}
\item[$(\mathcal{H}_1)$] The nonlinearity $\Psi$ is continuous in $t$ and $u$ and locally Lipschitz continuous in $x$, uniformly in $t$ and $u$ on bounded sets.
\end{description}
Now, we introduce the concept of mild solutions for (\ref{R1}) and explore some associated properties.
\begin{dfn}
Assume $(\mathcal{H}_1)$. Let $0 \leq t_0 < t_1 \leq \infty$. A continuous function
$x:[t_0,t_1)\rightarrow X$ is called a \emph{(mild) solution} of
(\ref{R1}) on the interval $[t_0,t_1)$ for the initial condition
$x(t_0)=x_0\in X,$ and the input $u\in \mathcal{U}$, if $x$ satisfies the integral equation
\begin{equation}\label{bvbbLL4}
x(t) = W(t,t_0)x_0 + \int _{t_0} ^{t} W(t,s)\Psi(s,x(s),u(s))ds, \quad t
\in [ t_0,t_1 ).
\end{equation}
\end{dfn}
Let $x_1, x_2$ be solutions of (\ref{R1}) defined on the intervals $[t_0, \tau_1)$
and $[t_0, \tau_2),$ respectively, $\tau_1, \tau_2 > t_0.$ We call $x_2$ an extension of $x_1$ if $\tau_2 > \tau_1,$ and
$x_2(t) = x_1(t)$ for all $t \in [t_0, \tau_1).$
A mild solution $x$ of (\ref{R1}) is called \emph{maximal} if there is no
solution of (\ref{R1}) that extends $x$. The solution is called \emph{global} if $x$ is defined on $[t_0,\infty).$

A central property of the system (\ref{R1}) is
\begin{dfn}\label{phi}
The system (\ref{R1}) is called \emph{well-posed} if for every initial time $t_0\geq0$, every initial condition $x_0\in X$ and every control input $u \in \mathcal{U},$ there exists a unique maximal solution that we denote by $$\phi(\cdot,t_0,x_0, u):[t_0, t_m(t_0,x_0,u))\to X.$$ 
In this case $t_m=t_m(t_0,x_0,u)$ is called the \emph{maximal existence time} of the solution corresponding to $(t_0,x_0,u).$
\end{dfn}
The map $\phi$ defined in Definition \ref{phi}, describing the evolution of the system (\ref{R1}), is called \emph{the flow map}, or just \emph{flow}. The domain of definition of the flow $\phi$ is
$$D_{\phi}\subseteq\mathcal{T}\times X\times\mathcal{U}.$$
The following result gives sufficient conditions for the well-posedness of
the system (\ref{R1}). It is a minor extension of a result by
Pazy, \cite[Theorem 6.1.4, p.185]{Paz83}, to the situation when the inputs are piecewise continuous.
\begin{prop}\label{propooo1}
If Assumption $(\mathcal{H}_1)$ holds, then system (\ref{R1}) is a well-posed control system with the BIC property.
\end{prop}
\begin{proof}
Let $t_0\geq 0,$ $x_0\in X,$ and $u\in PC(\mathbb{R}_{+}, U).$ Since $u$ is piecewise right continuous, there is an increasing sequence of discontinuities $(t_k)_{k\in \mathbb{N}}$ of $u$ without accumulation points.
We will construct the solution in an iterative fashion. Extend $u_{\vert [t_0,t_1)}$ in an arbitrary manner to a continuous function $\tilde{u}_0$ defined on
$[t_0,t_1+1]$ with $u_{\vert [t_0,t_1)} = \tilde{u}_{0,\vert [t_0,t_1)}$.  By an application
of the existence result \cite[Theorem 6.1.4, p.185]{Paz83}, there is $\tilde{t}\in (t_0,t_1+1),$ such that there is a unique
(maximal) mild solution of (\ref{R1}) on $[t_0,\tilde{t}),$ corresponding
to the initial condition $x_0$ and the input $\tilde{u}_0$.
If $\tilde{t}\leq t_1$, then $\tilde{t}$ is the maximal existence time
corresponding to $(t_0,x_0,u)$, as $u$ and $\tilde{u}$ coincide on
$[t_0,\tilde{t})$. Thus
$t_m=t_m(t_0,x_0,u)=\tilde{t}$. Otherwise, $x_1 :=
\phi(t_1,t_0,x_0,\tilde{u}) = \phi(t_1,t_0,x_0,u)$ exists. We may then
repeat the argument on the interval $[t_1, t_2 +1 ]$, again using a
continuous extension $\tilde{u}_1$ of $u_{\vert [t_1,t_2]} $ to that interval. 

Let $x: [t_0, \min \{ t_2 , t_m(t_1,x_1,\tilde{u}_1))\to X$ be defined by 
\begin{eqnarray*}
x(t):=
\begin{cases}
\phi(t,t_0,x_0,u), \quad & t \in [ t_0,t_1 ],\\
\phi(t,t_1,x_1,\tilde{u}_1) = \phi(t,t_1,x_1,u),\quad & t \in (t_1,\min \{ t_2 , t_m(t_1,x_1,\tilde{u}_1)).
\end{cases}
\end{eqnarray*}
In the equality in the definition of $x$, we have used that $u$ and $\tilde{u}_1$ coincide on $[t_1,t_2]$.

Then $x$ is continuous, clearly a solution of the integral equation on
$[t_0,t_1 ]$ and we have for $t_1 \leq t < \min \{ t_2 , t_m
\}$ that
\begin{align*}
    x(t) &= W(t,t_1) x_1 + \int_{t_1}^t
    W(t,s)\Psi(s,\phi(s,t_1,x_1,u),u(s)) ds \\
		&=  W(t,t_1) \left(
      W(t_1,t_0)x_0 + \int_{t_0}^{t_1}W(t_1,s)\Psi(s,\phi(s,t_0,x_0,u),u(s))ds \right)  + \int_{t_1}^t
    W(t,s)\Psi(s,x(s),u(s)) ds \\
		&=  W(t,t_0) x_0 + \int_{t_0}^{t}W(t,s)\Psi(s,x(s),u(s))ds.
\end{align*}
This shows that indeed $x$ is a mild solution on $[t_0, \min \{ t_2 , t_m(t_1,x_1,\tilde{u}_1))$.
By applying this reasoning iteratively, we obtain a unique maximal mild solution
$\phi(\cdot,t_0,x_0,u)$ on an interval $[t_0,t_m)$. If $u$ only has finitely many
discontinuities the same arguments can be applied on the final infinite interval of
continuity. By construction and \cite[Theorem 6.1.4, p.185]{Paz83}, each solution can be continued as long it
remains bounded so that the BIC property is guaranteed.
\end{proof}
Now we show that well-posed systems of the form (\ref{R1}) are a special case of general control systems introduced in Definition \ref{csyol}.
\begin{thm}
Let (\ref{R1}) be well-posed. Then the triple $(X,\mathcal{U},\phi)$, where $\phi$ is a flow map of (\ref{R1}), constitutes a control system in the sense of Definition \ref{csyol}.
\end{thm}
\begin{proof}
By construction, all the axioms in the definition of a control system are
fulfilled. We only need to check the cocycle property. Let $\phi$ be the
flow map and consider an initial time $t_0\geq 0$, an initial condition
$x_0\in X$, and an input $u \in \mathcal{U}$.
Let $t \in (t_0,t_m(t_0,x_0,u))$, $\tau\in [t_0,t]$ and consider the solution corresponding to the
initial condition $x(\tau) = \phi(\tau,t_0,x_0,u)$ and the input $u$. By
well-posedness of the system, there exists a unique solution corresponding
to these data on a maximal time interval $[\tau, \tau')$. Define a function $x
: [t_0, \min \{\tau',t \})\to X$, by $x(s) = \phi(s,t_0,x_0,u)$, $s\in [t_0,\tau ]$,
and $x(s) =  \phi(s,\tau,\phi(\tau,t_0,x_0,u),u)$, $s\in [\tau, \min \{\tau',t \})$. Then $x$ is continuous, clearly solves the integral equation
(\ref{bvbbLL4}) on the interval $[t_0,\tau ]$ and for $s \in [\tau, \min
\{\tau',t \})$ we have
\begin{multline*}
    x(s) = \phi(s,\tau,\phi(\tau,t_0,x_0,u),u) =
    W(s,\tau)\phi(\tau,t_0,x_0,u)+\int_{\tau}^{s}W(s,r)\Psi(r,\phi(r,\tau,\phi(\tau,t_0,x_0,u),u),u(r))dr
    \\ 
    = W(s,\tau)\bigg{[}W(\tau,t_0)x_0+\int_{t_0}^{\tau}W(\tau,r)\Psi(r,\phi(r,t_0,x_0,u),u(r))dr \bigg{]}
+\int _{\tau}^{s} W(s,r)\Psi(r,x(r),u(r))dr \\
= W(s,t_0)x_0+\int _{t_0}^{s} W(s,r)\Psi(r,x(r),u(r))dr.
\end{multline*}

This shows that $x$ is a solution for the initial values $(t_0, x_0)$ and
the input $u$ on the interval  $[t, \min \{\tau',t \})$. As this solution
is unique, $x$ coincides with $\phi(\cdot, t_0,x_0,u)$ on $[t, \min
\{\tau',t \})$. The BIC property then implies that $\tau' > t$
and the cocycle property holds, because for $
\phi(t,\tau,\phi(\tau,t_0,x_0,u),u)= x(t) = \phi(t,t_0,x_0,u)$.
\end{proof}

\subsection{\subsectiontitle{Stability of evolution families}}
In this subsection, we recall different types of stability for strongly continuous evolution families $\{W(t,t_0)\}_{t\geq t_0\geq0}$. These stability concepts include uniform stability, uniform attractiveness, uniform asymptotic stability, and uniform exponential stability. We also provide a proposition and an example that help to clarify the relationships between these different forms of stability.

\begin{dfn}
\label{def:StabAttr}
\cite[p. 112]{DaK74} \cite[Def 36.9 p. 174]{Hah67} \cite[Def 3.4 p. 60]{carmen}
A strongly continuous evolution family $\{W(t,t_0)\}_{t\geq t_0\geq0}$ is said to be:
\begin{enumerate}
\item[$(i)$] \emph{Uniformly stable} if there is $0<N<\infty$ such that
  for all $t\geq t_0 \geq0$ we have $\|W(t,t_0)\|\leq N.$
\item[$(ii)$] \emph{Uniformly attractive} if for all $\epsilon>0$, there exist $T=T(\epsilon)$ such that, for all $t_0\geq0, \ \|W(t,t_0)\| \leq \epsilon,$ for all $t\geq t_0+T$.
\item[$(iii)$] \emph{Uniformly asymptotically stable} if $\{W(t,t_0)\}_{t\geq t_0\geq0}$ is uniformly stable and uniformly attractive.
\item[$(iv)$] \emph{Uniformly exponentially stable} if there exist $k,\omega>0,$ such that $\|W(t,t_0)\|\leq ke^{-\omega(t-t_0)} $ holds for each $t_0\geq0$ and all $t\geq t_0$.
\end{enumerate}
\end{dfn}
The next proposition and Example~\ref{ex:Uniform-attractivity-and-uniform-stability} clarify the
relationships between these notions. These results are presumably
well-known in the theory of linear
time-varying systems. Nevertheless, we could not find the statements of
these results in the literature and the result is presented for completeness.
\begin{prop}
\label{prop:Uniform-attractivity-and-uniform-stability}
Assume that
\begin{eqnarray}
\sup_{t_0\geq 0,\ t\in[t_0,t_0+1]}\|W(t,t_0)\|=K<\infty.
\label{eq:Finite-Bohl-exponent}
\end{eqnarray}
If $\{W(t,t_0)\}_{t\geq t_0\geq0}$ is uniformly attractive, then it is uniformly stable and hence uniformly asymptotically stable.
\end{prop}
\begin{proof}
Let $\kappa>0.$ By the uniform attractivity of $\{W(t,t_0)\}_{t\geq t_0\geq0}$, there exist $T=T(\kappa)>0,$ such that $\|W(t,t_0)\|\leq \kappa$, for all $t_0\geq0$ and all $t\geq t_0+T.$ By taking a larger $T$ if needed, we can assume that $T \in \mathbb{N}.$
Using \eqref{eq:Finite-Bohl-exponent} and the property $(i)$ of $\{W(t,t_0)\}_{t\geq t_0\geq0}$, we have for any $r\in\{1,\ldots,T\}$, 
and any $t\in [t_0+r-1,t_0+r]$ that
$$ W(t,t_0)=W(t,t_0+r-1)W(t_0+r-1,t_0+r-2)...W(t_0+2,t_0+1)W(t_0+1,t_0),$$
and then for any $t\in [t_0+r-1,t_0+r]$
\begin{align*}
\|W(t,t_0)\| &\leq  \|W(t,t_0+r-1)\|\|W(t_0+r-1,t_0+r-2)\|\\
& \qquad \ldots \|W(t_0+2,t_0+1)\|\|W(t_0+1,t_0)\|< K^r<\infty.
\end{align*}
As $K\geq 1$ (consider $t=t_0$ in \eqref{eq:Finite-Bohl-exponent}), this ensures that
\begin{eqnarray}
\sup_{t_0\geq 0,\ t\in[t_0,t_0+T]}\|W(t,t_0)\|\leq K^T<\infty.
\label{eq:Finite-Bohl-exponentT}
\end{eqnarray}
Choosing $ N=\max\left(K^T,\kappa\right) $, we obtain $\|W(t,t_0)\|\leq
N$, for all $t\geq t_0\geq0$. Hence, $\{W(t,t_0)\}_{t\geq t_0\geq0}$ is uniformly stable, and overall it is uniformly asymptotically stable.
\end{proof}
In view of \cite[Theorem 4.2]{DaK74}, the condition \eqref{eq:Finite-Bohl-exponent} is equivalent to the finiteness of the upper Bohl exponent of \eqref{linear}. This condition holds, e.g., if $A(\cdot)$ is continuous and uniformly bounded on $[0,\infty)$.
Without finiteness of the upper Bohl exponent,
Proposition~\ref{prop:Uniform-attractivity-and-uniform-stability} does not
necessarily hold, in contrast to the time-invariant linear systems, where
strong stability of a $C_0$-semigroup (i.e. non-uniform global
attractivity of a dynamical system) always implies uniform stability. We
illustrate this fact by means of the Example
\ref{ex:Uniform-attractivity-and-uniform-stability}:
\begin{exm}\label{ex:Uniform-attractivity-and-uniform-stability}
In view of \cite[Theorem 4.2]{DaK74}, the condition \eqref{eq:Finite-Bohl-exponent} is equivalent to the finiteness of the upper Bohl exponent of \eqref{linear}. This condition holds, e.g., if $A(\cdot)$ is continuous and uniformly bounded on $[0,\infty)$.
Without finiteness of the upper Bohl exponent, Proposition~\ref{prop:Uniform-attractivity-and-uniform-stability} does not necessarily hold, in contrast to the time-invariant linear systems, where strong stability of a $C_0$-semigroup (i.e. non-uniform global attractivity of a dynamical system) always implies uniform stability. We illustrate this fact by means of an example.
We construct a uniformly attractive system for which the evolution operator $W(t,t_0)$ is not uniformly bounded. It suffices to consider a one-dimensional system. Define $A:\mathbb{R}_{+}\to\mathbb{R}$ by setting for $k=0,1,2,\ldots$
\begin{equation*}
A(t):=
\begin{cases}
-2\ln \left(2(k+1)^2\right), \quad & t\in [k, k+\frac{1}{2}),\\
\phantom {-} 2\ln(k+1),\quad & t\in [k+\frac{1}{2}, k+1)\,.
\end{cases}
\end{equation*}
With this, we have for all $k\geq 0$ the following observations:
\begin{equation*}
\sup_{s\in[0,1]}W\Big(k+1,k+s\Big)=W\Big(k+1,k+\frac{1}{2}\Big)=e^{2\ln(k+1)\frac{1}{2}}=k+1\,.
\end{equation*}
In particular, the associated system is not uniformly stable. On the other hand,
\begin{equation*}
W(k+1,k)=e^{2\ln (k+1)\frac{1}{2}}e^{-2\ln(2(k+1)^2)\frac{1}{2}}=(k+1)\cdot\frac{1}{2(k+1)^2}=\frac{1}{2(k+1)}\,,
\end{equation*}
and
\begin{equation*}
\sup_{s\in[0,1]}W(k+s,k)=W(k,k)=1.
\end{equation*}
We claim that the system is uniformly attractive. To this end, fix $\varepsilon >0$ and choose $m\in \mathbb{N}$ such that
\begin{equation*}
\frac{1}{2m}<\varepsilon.
\end{equation*}
We claim that $T(\varepsilon)=m+1$ is a suitable choice to satisfy the item 2 of Definition~\ref{def:StabAttr}.
Fix an arbitrary $t_0=k_0+\delta_0$ with $k_0\in\mathbb{N}\cup\{0\}$, and $0\leq\delta_0<1.$ Now pick any $t\ge t_0+m+1.$ Then, $t=k_0+M+s,$ where $M\ge m+1,$ and $s\in[0,1].$ Then,
\begin{align*}
W(t,t_0)= & W(t, k_0+M)W(k_0+M, k_0+M-1)
           \cdots
          W(k_0+2,k_0+1) W(k_0+1, t_0)
\end{align*}
and we obtain
\begin{equation*}
\left|W(t,t_0)\right|\leq 1 \cdot\frac{1}{2(k_0+M)}\cdots\frac{1}{2(k_0+2)}(k_0+1)<\frac{1}{2^m}<\varepsilon\,.
\end{equation*}
This proves the uniform attractivity.
\end{exm}
The following lemma is essential for establishing a fundamental connection between the stability of evolution families and the behavior of time-varying linear systems.
\begin{lem}
\label{lemmmcc1}\cite[Lemma $4.1$]{damak2021input}
Let $\{W(t,s)\}_{t\geq s\geq0}$ be a strongly continuous evolution family
with generating family $\{ A(t) \}_{t\geq 0}$. 
The following statements are equivalent:
\begin{enumerate}
\item[$(i)$] The family $\{W(t,s)\}_{t\geq s\geq0}$ is uniformly asymptotically stable.
\item[$(ii)$] The time-varying linear system \begin{equation}\label{linear}
\dot{x}=A(t)x,\quad x(t_0)=x_0,
\end{equation}
 is uniformly exponentially stable, that is, there exist $k,\omega > 0,$ such that
$$\|x(t)\| \leq k\|x_0\|e^{-\omega(t-t_0)}$$ holds for all $t_0\geq0,$ all $x_0 \in X,$ and all $t \geq t_0.$

\item[$(iii)$]The family $\{W(t,s)\}_{t\geq s\geq0}$ is uniformly exponentially stable.
\end{enumerate}
\end{lem}



\section{\sectiontitle{Lyapunov criteria for ISS of time-varying linear systems}}
\

We now consider a special case of (\ref{R1}), namely linear systems on a Banach space $X$ of the form: 
\begin{equation}\label{unb}
\left\lbrace
\begin{array}{l}
\dot{x}(t)=A(t)x(t)+B(t)u(t), \qquad t\ge t_0\ge0,\\
x(t_0)=x_0,
\end{array}\right.
\end{equation}
where $B\in C(\mathbb{R}_{+},L(U,X)),$ with $\displaystyle
\sup_{t\geq0}\|B(t)\|<\infty.$ We assume that the family $\{ A(t)
\}_{t\geq 0}$ generates a strongly continuous evolution family $\{W(t,s)\}_{t\geq s\geq0}$. As in the general case, we assume that inputs belong to the space $\mathcal{U}=PC(\mathbb{R}_{+}, U).$
Since $B$ is continuous and $u$ is piecewise right continuous, $\Psi:(t,x,u)\mapsto B(t)u(t)$ satisfies the assumption $(\mathcal{H}_1)$ and, according to Proposition \ref{propooo1}, there is a unique global mild solution $\phi(\cdot,t_0,x_0,u)$ of system (\ref{unb}) for any data $(t_0,x_0,u).$
By definition, it has the form
\begin{equation}\label{b5484}
\phi(t,t_0,x_0,u)= W(t,t_0)x_0+\int_{t_0}^{t}W(t,\tau)B(\tau)u(\tau)d\tau.
\end{equation}

\subsection{\subsectiontitle{Uniformly bounded and continuous $A$}}
In this part, we concentrate on the case of Hilbert spaces $X$ endowed with the scalar product $\langle \cdot,\cdot\rangle.$
We start with a Lyapunov characterization of the ISS property for the
system \eqref{unb} for the case that $A: \R_+ \to L(X)$ is continuous in
the uniform operator topology. Under these assumptions, the mild solutions
of \eqref{unb} exist and are piecewise classical solutions in the sense of Definition \ref{rrnnann23}.

\begin{prop}\label{prop:mild-is-classical}
Let $A: \R_+ \to L(X)$ be continuous and uniformly bounded in the uniform operator topology. Then, for every $u\in \mathcal{U},$ every initial time $t_0$ and every initial condition $x_0,$ the 
mild solution $\phi(\cdot,t_0,x_0,u)$ given by (\ref{b5484}) belongs to $PC^{1}([t_0,\infty),X)$ and is the unique piecewise classical solution of (\ref{unb}).
\end{prop}

\begin{proof}
By \cite[Theorem 5.1, p. 127]{Paz83}, for each initial time $t_0$ and each initial condition $x_0$, there is a 
unique classical solution $x(\cdot)$ for the (undisturbed) system (\ref{linear}). Defining
$$x(t):=W(t,t_0)x_0,\quad t\geq t_0,$$
we construct an evolution family $\{W(t,t_0)\}_{t\geq t_0\geq0}: X \to  X$ of bounded operators, generated by $\{A(t)\}_{t\geq0}.$ Furthermore, according to \cite[Theorem 5.2, p. 128]{Paz83}, we have the following formulas for the partial derivatives of $W$ taken in $L(X):$
$$\frac{\partial}{\partial t}W(t,s)=A(t)W(t,s),\quad \frac{\partial}{\partial s}W(t,s)=-W(t,s)A(s).$$
The result follows using arguments as in \cite[Equation (1.18), p. 129]{Paz83}.
\end{proof}
For the Lyapunov analysis of linear systems, we recall the following notions.
A self-adjoint operator $S \in L(X)$ is said to be positive, if $\langle Sx,x\rangle >0$ for all $x\in X\backslash\{0\}.$ We call an operator-valued function $P:\mathbb{R}_{+}\to L(X)$ positive if $P(t)$ is self-adjoint and positive for any $t\geq0$.
An operator-valued function $P:\mathbb{R}_{+}\to L(X)$ is called coercive,
if it is positive and there exists $\mu>0,$ such that
$$\langle P(t)x,x\rangle\geq \mu\|x\|^{2}_{X} \quad \forall x\in X.$$
The following result guaranteeing the existence of a coercive solution for operator Lyapunov equations will be helpful in the sequel. It is related to analogous characterizations of exponential dichotomy given in \cite[Corollary 4.48]{carmen}.
\begin{prop}\label{hmfdrr}\cite[Theorem $5.2$]{damak2021input}.
Let $A : \R_+ \to L(X)$ be continuous and uniformly bounded in the uniform
operator topology.
Then, the evolution operator generated by $\{A(t)\}_{t\geq0}$ is uniformly exponentially stable if and only if there exists a continuously differentiable, bounded, coercive positive operator-valued function $P:\mathbb{R}_{+}\to L(X)$, satisfying the Lyapunov equality
\begin{equation}\label{R2}
A(t)^*P(t)+ P(t)A(t)+\dot{P}(t)=-I,\quad t\geq 0.
\end{equation}
\end{prop}
\begin{proof}
According to our assumption and arguments at the beginning of the
subsection, the family $\{A(t)\}_{t\geq0}$ generates a strongly continuous
evolution family of bounded operators $\{W(t,s)\}_{t\geq s\geq0}$. Assume
that $\{W(t,t_0)\}_{t\geq t_0\geq0}$ is uniformly exponentially stable. Define the operator-valued function $P$ by
$$P(t)=\int_{t}^{\infty}W(\tau,t)^{*}W(\tau,t)d\tau.$$
The integral converges by the uniform exponential stability.
 Thanks to the continuity of $t\mapsto A(t)$, the map $(\tau,t)\mapsto
 W(\tau,t)$ is uniformly continuous by \cite[Theorem 5.2, p.128]{Paz83},
 and thus $P$ is continuously differentiable. $P$ is coercive due to \cite[Lemma 5.1]{damak2021input}.
 Thus, we can compute
 $\dot{P}$ as in \cite[Theorem $5.2$]{damak2021input} and obtain
$$\dot{P}(t)= \int _{t} ^{\infty} W(\tau,t)\frac{\partial}{\partial
  t}W(\tau,t)d\tau+\int_{t}^{\infty }\frac{\partial}{\partial
  t}W(\tau,t)^*W(\tau,t)d\tau-I = -P(t)A(t) - A(t)^*P(t)+   -I. $$

For the converse direction see \cite[Theorem 5.2]{damak2021input}.
\end{proof}
Next, we give exhaustive criteria for ISS of time-varying linear systems (\ref{unb}) with a family of the linear uniformly bounded operators 
$\{A(t)\}_{t\geq0}.$
\begin{thm} Let $A: \R_+ \to L(X)$ be continuous and uniformly bounded in the uniform
operator topology with associated evolution operator $\{W(t,t_0)\}_{t\geq t_0 \geq 0}$. Then we have 
\begin{itemize}
\item[$(i)$] (\ref{unb}) is ISS $\quad\Leftrightarrow\quad$ (\ref{unb}) is
  0-UGAS $\quad\Leftrightarrow\quad$ (\ref{unb}) is iISS.
\end{itemize}
And each of the  following statements is equivalent to (\ref{unb}) being ISS:
\begin{itemize}
\item[$(ii)$] The evolution operator $\{W(t,t_0)\}_{t\geq t_0 \geq 0}$ is uniformly asymptotically stable.
\item[$(iii)$] The evolution operator $\{W(t,t_0)\}_{t\geq t_0 \geq 0}$ is uniformly exponentially stable.
\item[$(iv)$] If $P(t)$ satisfies (\ref{R2}), then the function $V: \mathbb{R}_{+}\times X \to \mathbb{R}_{+}$ defined by
\begin{equation}\label{R1Ev}
V(t,x)=\langle P(t)x,x\rangle, \quad t\geq0, \quad x\in X,
\end{equation}
is an ISS Lyapunov function for (\ref{unb}).
\end{itemize}
\end{thm}
\begin{proof}
The equivalences in $(i)$ and the equivalence of each of the statements in
(i) to (ii) and (v) are covered by \cite[Theorem 4.1]{damak2021input},
\cite[Corollary 4.1]{damak2021input}, \cite[Lemma 4.1]{damak2021input}. It
follows from Theorem \ref{hjMMMMMljg} that (iv) implies that (\ref{unb}) is ISS.

$(iii)\Longrightarrow (iv).$ As the evolution operator $\{W(t,t_0)\}_{t\geq t_0 \geq 0}$ is uniformly exponentially stable, according to Proposition \ref{hmfdrr}, there is an operator-valued function $P$ satisfying (\ref{R2}). Consider $V:\mathbb{R}_{+}\times X \to \mathbb{R}_{+}$ as defined in (\ref{R1Ev}). As $P$ is continuous, $V$ is continuous as well. Since $P$ is bounded and coercive, for some $\mu_1,p> 0$, it holds that
$$\mu_1\|x\|^{2}_{X}\leq V(t,x)\leq p\|x\|^{2}_{X}\quad \forall x \in X\quad \forall t\geq 0,$$
where $p=\displaystyle \sup_{t\in \mathbb{R}_{+}}\|P(t)\|$. Thus, (\ref{R1E}) holds.

For any given $x_0\in X,$ for any $u\in\mathcal{U},$ any $t_0\geq0,$ and
since the corresponding mild solution $x(t)=\phi(t,t_0,x_0,u)$ is a
piecewise classical solution by Proposition~\ref{prop:mild-is-classical},
it is piecewise continuously differentiable. Hence, with the exception of
countably many points the Lie derivative of $V$ along the trajectories of
system (\ref{unb}) satisfies (omitting the argument $t$ for legibility)
\begin{align*}
\dot{V}_{u}(t,x)&=\langle\dot{P}(t)x,x\rangle+\langle P(t)\dot{x},x\rangle +\langle P(t)x,\dot{x}\rangle\\
&=\langle \dot{P}(t)x,x\rangle+ \langle P(t)[A(t)x+B(t)u],x(t)\rangle+\langle P(t)x, A(t)x+B(t)u\rangle\\
&=\langle \dot{P}(t)x,x\rangle+\langle P(t)A(t)x,x\rangle+\langle
P(t)B(t)u,x\rangle +\langle P(t)x,A(t)x\rangle+\langle P(t)x,B(t)u\rangle.
\end{align*}
Since $\langle P(t)x,A(t)x\rangle=\langle A(t)^*P(t)x,x\rangle$
by applying the Lyapunov equation (\ref{R2}) 
and using the Cauchy-Schwarz inequality, we obtain
$$\dot{V}_{u}(t,x)\leq -\|x\|^{2}_{X}+2 p\|x\|_{X}\|B(t)\|\|u\|_{\mathcal{U}}.$$
Utilizing Young's inequality we have for any $\eta>0$
$$\dot{V}_{u}(t,x)\leq -\|x\|^{2}_{X}+ p\tilde{b}\eta\|x\|^{2}_{X}+\frac{p\tilde{b}}{\eta}\|u\|^{2}_{\mathcal{U}},$$
where $\tilde{b}=\sup_{t\in \mathbb{R_{+}}}\|B(t)\|.$ Choosing
$\eta<(p\tilde{b})^{-1}$ this shows that $V$ is an ISS Lyapunov function for (\ref{unb}).
\end{proof}


\subsection{\subsectiontitle{Family of unbounded generators}}

In the remainder, all vector spaces will assumed to be Banach spaces.
In this subsection, we extend the analysis of ISS Lyapunov functions to
systems where the operators in the family $\{A(t): D(A(t))\subset X\to
X\}_{t\geq 0}$ are generally unbounded. 
As before $\mathcal{U}=PC(\mathbb{R}_+, U)$ will be used throughout.
Again we obtain that uniform exponential stability of the evolution family
$\{W(t,\tau)\}_{t\geq \tau\geq0}$, characterizes ISS 
and allows for the construction of noncoercive Lyapunov functions for systems with unbounded operators $A(\cdot)$.
The main contribution of this subsection is Theorem \ref{hdkmnmlo}, which establishes the equivalence between various stability concepts for the system (\ref{unb}) including ISS, 0-UGAS, iISS, and the uniform exponential stability of the evolution operator. Theorem \ref{hdkmnmlo} also proposes a non-coercive ISS Lyapunov function for the system.
\

We start with a simple lemma. 
\begin{lem}\label{hnlkfrankm}
Let $B\in C(\mathbb{R}_{+},L(X))$ with $\sup_{t\in \mathbb{R_{+}}}\|B(t)\| < \infty,$ and $\{W(t,s)\}_{t\geq s\geq0}$ be a strongly continuous evolution family. Then, for any $u\in\mathcal{U}$ and any $t\geq0$ it holds that 
$$ \displaystyle\lim _{h \to 0^{+}}\frac{1}{h}\int_{t}^{t+h}W(t,s)B(s)u(s)ds=B(t)u(t). $$
\end{lem}

%
%
%
%

Next, we generalize our results about ISS Lyapunov functions to the case of the unbounded operator $A(\cdot)$ and derive a constructive converse Lyapunov theorem for (\ref{unb}) with a bounded input operator $\{B(t)\}_{t\geq0}$. 
Our construction \eqref{eq:non-coercive ISS-LF-linsys} is motivated by a corresponding construction for linear time-invariant systems, see \cite{mironchenko2018lyapunov}.

\begin{thm}\label{hdkmnmlo}
Let the family $\{A(t)\}_{t\geq0}$ be the generator of a strongly
continuous evolution family $\{W(t,s)\}_{t\geq s\geq0}.$ Assume
$B\in C(\mathbb{R}_{+},L(U,X)),$ with $\|B\|_{\infty}=\sup_{t\in \mathbb{R_{+}}}\|B(t)\|<\infty.$ The following statements are equivalent for the system (\ref{unb}):
\begin{enumerate}
\item[$(i)$] (\ref{unb}) is ISS.
\item[$(ii)$] (\ref{unb}) is 0-UGAS.
\item[$(iii)$] (\ref{unb}) is iISS.
\item[$(iv)$] The evolution operator $\{W(t,t_0)\}_{t\geq t_0 \geq 0}$ is uniformly asymptotically stable.
\item[$(v)$] The evolution operator $\{W(t,t_0)\}_{t\geq t_0 \geq 0}$ is uniformly exponentially stable.
\item[$(vi)$] The function $V:\mathbb{R}_{+} \times X \rightarrow \mathbb{R}_{+},$ defined by
\begin{eqnarray}\label{rrrrhnlkfn}
V(t,x)=\int_{t}^{\infty}\|W(\tau,t)x\|_{X}^{2}d\tau,
\label{eq:non-coercive ISS-LF-linsys}
\end{eqnarray}
is \emph{a non-coercive ISS Lyapunov function} for (\ref{unb}) which is locally Lipschitz continuous. Moreover, for all $t\geq0,$ all $x\in X,$ all $u\in\mathcal{U},$ and all $\eta>0$ it holds that 
\begin{equation}\label{rhmnhlh}
\dot{V}_{u}(t,x)\leq-\|x\|_{X}^{2}+\frac{\eta k^{2}}{2w}\|x\|_{X}^{2}+\frac{k^{2}}{2\eta w}\|B\|_{\infty}\|u(t)\|_{U}^{2},
\end{equation}
where $k,w>0$ are so that
\begin{equation}\label{rhklet}
\|W(t,t_0)\|\leq k e^{-w(t-t_0)} \quad \forall t_0\geq0 \quad t\geq t_0.
\end{equation}
\end{enumerate}
\end{thm}
\begin{proof}
$(i)\Longleftrightarrow(ii)$. This follows from \cite[Theorem 4.1]{damak2021input}. 

$(iii)\Longrightarrow (ii)$. Evident.

$(i)\Longrightarrow(iii)$. This follows as ISS implies so-called $L^1$-ISS (see again \cite[Theorem 4.1]{damak2021input}).

$(iv)\Longleftrightarrow (v)$ is contained in Lemma~\ref{lemmmcc1}. 

$(iii)\Longrightarrow (iv)$ follows directly from the definition of iISS.

$(v) \Longrightarrow (iv)$ is easy to establish using the exponential
bounds for the evolution operator and the variation of the constants formula.

$(vi) \Longrightarrow (v)$. The proof of this implication is analogous to the proof of \cite[Theorem 1]{datko1972uniform}.

$(v)\Longrightarrow(vi)$. 
 The evolution operator $\{W(t,t_0)\}_{t\geq t_0 \geq 0}$ is uniformly exponentially stable, i.e, there exist $k,w>0$ such that (\ref{rhklet}) holds.
Consider $V:\mathbb{R}_{+} \times X \rightarrow \mathbb{R}_{+},$ as defined in (\ref{eq:non-coercive ISS-LF-linsys}). 
We have for all $t\geq0$ and all $x\in X$ 
\begin{align*}
V(t,x)&\leq\int_{t}^{\infty}\|W(\tau,t)\|^{2}\|x\|^{2}_{X}d\tau \leq\frac{k^{2}}{2w}\|x\|^{2}_{X}.
\end{align*}
Let $V(t,x)=0.$ Then, $\|W(\tau,t)x\|_{X}=0$ for all $\tau\geq t$ and all $t\geq0.$ By the strong continuity of $\{W(t,s)\}_{t\geq s},$ we get that $x=0$ and hence (\ref{R1EN}) holds.  
Next we pick any $x\in X,$ $t\geq0$, $u\in\mathcal{U}$, and estimate the
Lie derivative of $V$ as  
\begin{align*}
\dot{V}_{u}(t,x)&=\displaystyle\limsup _{h \to 0^{+}}\frac{1}{h}\big{(}V(t+h,\phi(t+h,t,x,u))-V(t,x)\big{)}\\
&=\displaystyle\limsup _{h \to 0^{+}}\frac{1}{h}\bigg{(}\int_{t+h}^{\infty}\|W(\tau,t+h)\phi(t+h,t,x,u)\|_{X}^{2}d\tau-\int_{t}^{\infty}\|W(\tau,t)x\|_{X}^{2}d\tau\bigg{)}.
\end{align*}
Since 
\begin{align*}
W(\tau,t+h)\phi(t+h,t,x,u) 
&= W(\tau,t+h)\Big(W(t+h,t)x_0+\int_{t}^{t+h}W(t+h,s)B(s)u(s)ds\Big)\\
&= W(\tau,t)x_0+\int_{t}^{t+h}W(\tau,s)B(s)u(s)ds,
\end{align*}
we obtain
\begin{align*}
\dot{V}_{u}(t,x)&\leq\displaystyle\limsup _{h \to 0^{+}}\frac{1}{h}\bigg{(}\int_{t+h}^{\infty}\Big\|W(\tau,t)x+\int_{t}^{t+h}W(\tau,s)B(s)u(s)ds\Big\|_{X}^{2}d\tau- \int_{t}^{\infty}\|W(\tau,t)x\|_{X}^{2}d\tau\bigg{)}\\
&\leq\displaystyle\limsup _{h \to 0^{+}}\frac{1}{h}\bigg{(}\Big{(}\int_{t+h}^{\infty}\|W(\tau,t)x\|_{X}^{2}-\int_{t}^{\infty}\|W(\tau,t)x\|_{X}^{2}d\tau\Big{)} \\
&\ \ \ \ \ \ \ \ +\int_{t+h}^{\infty}\Big{(}\Big\|\int_{t}^{t+h}W(\tau,s)B(s)u(s)ds\Big\|_{X}^{2}+2\|W(\tau,t)x\|_{X}\Big\|\int_{t}^{t+h}W(\tau,s)B(s)u(s)ds\Big\|_X \Big{)}d\tau\bigg{)}\\
&\leq J_{1}+J_{2},
\end{align*}
where
 $$J_1=\displaystyle\limsup _{h \to 0^{+}}\frac{1}{h}\Big{(}\int_{t+h}^{\infty}\|W(\tau,t)x\|_{X}^{2}-\int_{t}^{\infty}\|W(\tau,t)x\|_{X}^{2}d\tau\Big{)},$$
 and
$$J_2=\displaystyle\limsup _{h \to 0^{+}}\frac{1}{h}\int_{t+h}^{\infty}\bigg{(}\Big\|\int_{t}^{t+h}W(\tau,s)B(s)u(s)ds\Big\|_{X}^{2} + 2\|W(\tau,t)x\|_{X}\Big\|\int_{t}^{t+h}W(\tau,s)B(s)u(s)ds\Big\|_{X} \bigg{)}d\tau.$$
We have $$J_1=\displaystyle\limsup _{h \to 0^{+}}\frac{1}{h}\bigg{(}-\int_{t}^{t+h}\|W(\tau,t)x\|_{X}^{2}d\tau\bigg{)}=-\|x\|_{X}^{2}.$$

Now we proceed with $J_2:$
\begin{align*}
J_2&\leq\displaystyle\limsup _{h \to 0^{+}}\int_{t}^{\infty}\frac{1}{h}\Big\|\int_{t}^{t+h}W(\tau,s)B(s)u(s)ds\Big\|_{X}^{2}d\tau+\displaystyle\limsup _{h \to 0^{+}}\int_{t}^{\infty}2\|W(\tau,t)x\|_{X}\Big\|\frac{1}{h}\int_{t}^{t+h}W(\tau,s)B(s)u(s)ds\Big\|_{X}d\tau.
\end{align*}

The limit of the first term equals zero since
\begin{multline*}
\limsup _{h \to
  0^{+}}\frac{1}{h}\int_{t+h}^{\infty}\Big\|\int_{t}^{t+h}W(\tau,s)B(s)u(s)ds\Big\|_{X}^{2}d\tau
\leq k^{2} \|u\|^2_{\mathcal{U}}\|B\|^2_{\infty} \limsup _{h \to 0^{+}} \int_{t+h}^{\infty}\frac{1}{h} 
\left(\int_{t}^{t+h} e^{-w (\tau-s)} ds \right)^2
d\tau\\
= k^{2} \|u\|^2_{\mathcal{U}}\|B\|^2_{\infty} \limsup _{h \to 0^{+}}
 \frac{\left( e^{wh}-1\right)^2 }{hw^2} \int_{t+h}^{\infty} e^{-2w(\tau-t)}
   d\tau = 0.  
\end{multline*}
To bound the limit of the second term, note that for each fixed $\tau$
\begin{align*}
2\|W(\tau,t)x\|_{X}\left\|\frac{1}{h}\int_{t}^{t+h}W(\tau,s)B(s)u(s)ds\right\|_{X}
&\leq 2k^{2}e^{-2w(\tau-t)}\|x\|_{X}\|B\|_{\infty}\|u\|_{\mathcal{U}}
\frac{1}{wh}\left(e^{wh -1 }\right) .
\end{align*}
Applying the dominated convergence theorem and Lemma \ref{hnlkfrankm}, we obtain 
\begin{align}\label{jhgfyjk}
J_2&=\displaystyle\limsup _{h \to 0^{+}}\int_{t}^{\infty}2\|W(\tau,t)x\|_{X}\Big\|W(\tau,t)\int_{t}^{t+h}\frac{1}{h}W(t,s)B(s)u(s)ds\Big\|_{X}d\tau \notag\\
&=\int_{t}^{\infty}2\|W(\tau,t)x\|_{X}\|W(\tau,t)B(t)u(t)\|_{X}d\tau.
\end{align}
By using Young's inequality we obtain for any $\eta> 0$ that
\begin{equation}\label{nai23mr}
J_2\leq\int_{t}^{\infty}\eta\|W(\tau,t)x\|_{X}^{2}+\frac{1}{\eta}\|W(\tau,t)B(t)u(t)\|_{X}^{2}d\tau.
\end{equation}
Then by (\ref{rhklet}), we have
\begin{align*}
J_2&\leq\frac{\eta k^{2}}{2w}\|x\|^{2}_{X}+\frac{k^{2}}{2\eta w}\|B(t)\|^{2}\|u(t)\|^{2}_{U}\\
&\leq\frac{\eta k^{2}}{2w}\|x\|^{2}_{X}+\frac{k^{2}}{2\eta w}\|B\|^{2}_{\infty}\|u(t)\|^{2}_{U}.
\end{align*}
Therefore, 
$$\dot{V}_{u}(t,x)\leq -\|x\|_{X}^{2} + \frac{\eta k^{2}}{2w}\|x\|^{2}_{X} + \frac{k^{2}}{2\eta w}\|B\|^{2}_{\infty}\|u(t)\|^{2}_{U}.$$
Overall, for all $x\in X,$ all $u\in\mathcal{U},$ all $t\geq0,$ and all $\eta>0$
we obtain that the inequality (\ref{rhmnhlh}) holds. Considering $\eta <\frac{2w}{k^{2}}$ this shows that V is a non-coercive ISS Lyapunov function in dissipative form for (\ref{unb}). 

For the local Lipschitz continuity of $V,$ pick any $r>0,$ any $x,y\in X$ with $\|x\|_{X},\|y\|_{X}<r$ and any $t\geq0.$ We have: 
\begin{align*}
|V(t,x)-V(t,y)|&=\bigg{|}\int_{t}^{\infty}\|W(\tau,t)x\|_{X}^{2} - \|W(\tau,t)y\|_{X}^{2}d\tau\bigg{|}\\
&\leq\int_{t}^{\infty}\big{|}\|W(\tau,t)x\|_{X}^{2}-\|W(\tau,t)y\|_{X}^{2}\big{|}d\tau\\
&=\int_{t}^{\infty}\bigg{|}\big{(}\|W(\tau,t)x\|_{X}-\|W(\tau,t)y\|_{X}\big{)}\big{(}\|W(\tau,t)x\|_{X}+\|W(\tau,t)y\|_{X}\big{)}\bigg{|}d\tau.
\end{align*}
Since $\big{|}\|a\|_{X}-\|b\|_{X}\big{|}\leq\|a-b\|_{X},$ 
for all $a,b\in X$ and since $\|x\|_{X},\|y\|_{X}<r,$ one has by (\ref{rhklet}) that 
\begin{align*}
|V(t,x)-V(t,y)|&\leq\int_{t}^{\infty}2kr\|W(\tau,t)(x-y)\|_{X}d\tau\\
&\leq2kr\int_{t}^{\infty}ke^{-w(\tau-t)}\|x-y\|_{X}d\tau\\
&\leq\frac{2k^2r}{w}\|x-y\|_{X},
\end{align*}
which shows that $V$ is locally Lipschitz.

\end{proof}

\section{\sectiontitle{Lyapunov methods for LISS and iISS of semi-linear systems}}

We reformulate the system (\ref{R1}) in the following form on a Banach space $X$:
\begin{equation}\label{psiunb}
\left\lbrace
\begin{array}{l}
\dot{x}(t)=A(t)x(t)+B(t)u(t)+\psi(t,x(t),u(t)),\qquad t\geq t_0 \geq 0,\\
x(t_0)=x_0,
\end{array}\right.
\end{equation}
where $\psi:\mathbb{R}_{+} \times X\times U\rightarrow X$ satisfies the
Assumption $(\mathcal{H}_1)$ and $B\in C(\mathbb{R}_{+},L(U,X)),$ with
$\displaystyle \sup_{t\geq0}\|B(t)\|<\infty.$ As in the case of system (\ref{R1}), we assume that inputs belong to the space $\mathcal{U}=PC(\mathbb{R}_+,U).$
In this section, we prove two Lyapunov results. In Theorem \ref{hndrhn23},
we construct a (non-)coercive LISS Lyapunov function for the system
(\ref{psiunb}). In Theorem \ref{hndkrnm}, under a certain condition on $\psi,$ we construct a (non-)coercive iISS Lyapunov function for the system (\ref{psiunb}). 
Since $B$ is continuous and $u$ is piecewise continuous, $\Psi(t,x,u)=B(t)u+\psi(t,x,u)$ satisfies the Assumption $(\mathcal{H}_1)$ and according to Proposition \ref{propooo1}, there is a unique maximal mild solution $\phi(\cdot,t_0,x_0,u)\in C([t_0,t_m],X)$ of system (\ref{unb}) for any data $(t_0,x_0,u),$
where $0<t_m=t_m(t_0,x_0,u)\leq\infty.$

\subsection{\subsectiontitle{Constructions of LISS Lyapunov functions}}

In this subsection, we provide a method for the construction of
non-coercive LISS Lyapunov functions for the system (\ref{psiunb}) under a
local linear boundedness assumption.

\begin{description}
\item[$(\mathcal{H}_2)$]  For each $a>0$, there exist $b>0$ and (sufficiently small) $\rho>0$ such that for all $x \in X,$ all $t\geq 0,$ all $u\in \mathcal{U}$ satisfying $\|x\|_{X},\|u\|_{\mathcal{U}}\leq\rho$ it holds that
\begin{equation}
\label{psiiii}
\|\psi(t,x,u)\|_{X} \leq a\|x\|_{X}+b\|u\|_{\Uc}.
\end{equation}
\end{description}

In the following theorem, we outline the conditions for constructing a non-coercive LISS Lyapunov function for the system (\ref{psiunb}).

\begin{thm}\label{hndrhn23}
Let the family $\{A(t)\}_{t\geq0}$ be the generator of a strongly
continuous evolution family $\{W(t,s)\}_{t\geq s\geq0}.$ Assume
$B\in C(\mathbb{R}_{+},L(U,X)),$ with $\|B\|_{\infty}=\sup_{t\in \mathbb{R_{+}}}\|B(t)\|<\infty.$
Consider system
(\ref{psiunb}) and assume that Assumptions $(\mathcal{H}_1)$ and $(\mathcal{H}_2)$ hold. 
If (\ref{unb}) is 0-UGAS, then the function $V:\mathbb{R}_{+} \times X \rightarrow \mathbb{R}_{+},$ defined as in (\ref{rrrrhnlkfn}) by
$$V(t,x)=\int_{t}^{\infty}\|W(\tau,t)x\|_{X}^{2}d\tau,$$
is a non-coercive LISS Lyapunov function for (\ref{psiunb}).
\end{thm}
\begin{proof}
%
%
%
%
Let (\ref{unb}) is 0-UGAS and pick $u\equiv0.$ According to Theorem
\ref{hdkmnmlo}, the evolution family $\{W(t,s)\}_{t\geq s\geq0}$ generated
by $\{A(t)\}_{t\geq0}$ is uniformly exponentially stable, that is there
exist $k,w>0$ such that $\|W(t,t_0)\|\leq ke^{-\omega(t-t_0)} $ holds for
all $t\geq t_0 \geq 0$.

Pick $a,b,\rho>0$ as in Assumption $(\mathcal{H}_2)$. For all $x \in X,$
all $u\in \Uc$ with $\|x\|_{X}, \|u\|_{\mathcal{U}}\leq\rho$, we apply a
similar analysis as in the proof of Theorem \ref{hdkmnmlo}, with the distinction that instead of $B(t)u(t)$, we consider $B(t)u(t) + \psi(t,x(t),u(t))$ in (\ref{jhgfyjk}).
Then, we obtain that 
\begin{align*}
\dot{V}_{u}(t,x)&\leq -\|x\|_{X}^{2}+\int_{t}^{\infty}2\|W(\tau,t)x\|\Big\|W(\tau,t)\Big(B(t)u(t)+\psi(t,x(t),u(t))\Big)\Big\|_{X}d\tau\\
&\leq -\|x\|_{X}^{2}+2\int_{t}^{\infty}\|W(\tau,t)\|^{2}\|x\|_{X}\|B(t)u(t)+\psi(t,x(t),u(t))\|_{X}d\tau\\
&\leq -\|x\|_{X}^{2}+\frac{k^{2}}{w}\|x\|_{X}\big{(}\|B(t)\|\|u(t)\|_{U}+\|\psi(t,x(t),u(t))\|_{X}\big{)}.
\end{align*}
 We continue the above estimates: 
\begin{align*}
\dot{V}_{u}(t,x)
&\leq -\|x\|_{X}^{2}+\frac{k^{2}}{w}\|x\|_{X}\bigg{(}\|B\|_{\infty}\|u(t)\|_{U}+a\|x\|_{X}+b\|u(t)\|_{U}\bigg{)}\\
&\leq\big{(} -1+\frac{k^{2}}{w}a\big{)}\|x\|_{X}^{2}+\frac{k^{2}}{w}\|x\|_{X}\big{(}\|B\|_{\infty}+b\big{)}\|u\|_{\mathcal{U}},
\end{align*}
where $\|B\|_{\infty}=\sup_{t\geq0}\|B(t)\|.$
Define $\kappa\in\mathcal{K} $ by $\kappa(r)=\sqrt{r},$ $r\geq0$. 
Whenever $\|u\|_{\mathcal{U}}\leq \kappa^{-1}(\|x\|_{X})=\|x\|^{2}_{X},$ we yield
\begin{align*}
\dot{V}_{u}(t,x)
&\leq\big{(} -1+\frac{k^{2}}{w}a\big{)}\|x\|_{X}^{2}+\frac{k^{2}}{w}\big{(}\|B\|_{\infty}+b\big{)}\|x\|_{X}^{3}\\
&\leq\big{(} -1+\frac{k^{2}}{w}a\big{)}\|x\|_{X}^{2}+\frac{k^{2}}{w}\rho\big{(}\|B\|_{\infty}+b\big{)}\|x\|_{X}^{2}.
\end{align*}
Since $a$ and $\rho$ can be chosen arbitrarily small, the right-hand side can be estimated from above by some negative quadratic function of $\|x\|_{X}$.
According to Theorem \ref{hjMMMMMljg}, $V$ is a non-coercive LISS Lyapunov function for (\ref{psiunb}).


\end{proof}
\subsection{\subsectiontitle{Constructions of iISS Lyapunov functions}}

Now, we turn our attention to time-varying bilinear systems of the form (\ref{psiunb}). 
In \cite{MiI16}, the equivalence between uniform global asymptotic stability and iISS was shown for bilinear time-invariant infinite-dimensional control systems. Additionally, in \cite{mironchenko2015note} a method for the construction of non-coercive iISS Lyapunov functions was presented when the state is a Banach space. It was further extended in \cite{damak2021input} to bilinear time-varying infinite-dimensional control systems, where the operators in the family $\{A(t)\}_{t\geq0}$ are bounded. We will generalize these results to time-varying bilinear infinite-dimensional control systems where the operators in the family $\{A(t)\}_{t\geq0}$ are unbounded. Together with the results from \cite[Proposition 5]{mironchenko2015note}, we establish Lyapunov characterization of iISS for (\ref{psiunb}). For this purpose, we impose the following assumption
\begin{description}
\item[$(\mathcal{H}_3)$]  There exist $\gamma>0$ and $\delta\in \mathcal{K},$ so that for all $x\in X$, all $u\in \mathcal{U},$ and all $t\geq 0,$ we have
\begin{equation}\label{ibiISS}
\|\psi(t,x,u)\|_{X}\leq\gamma\|x\|_{X}\delta(\|u\|_{\mathcal{U}}).
\end{equation}
\end{description}

\begin{thm}\label{hndkrnm}
Let the family $\{A(t)\}_{t\geq0}$ be the generator of a strongly
continuous evolution family $\{W(t,s)\}_{t\geq s\geq0}.$ Assume
$B\in C(\mathbb{R}_{+},L(U,X)),$ with $\|B\|_{\infty}=\sup_{t\in \mathbb{R_{+}}}\|B(t)\|<\infty.$
Consider system
(\ref{psiunb}) and
let Assumptions $(\mathcal{H}_1)$ and $(\mathcal{H}_3)$ hold.
Assume that the evolution operator $\{W(t,s)\}_{t\geq s\geq0}$ is uniformly exponentially stable and 
let $V$ be defined as in (\ref{rrrrhnlkfn}). 
If $\{W(t,s)\}_{t\geq s\geq0}$ satisfies
\begin{equation}
\label{nmin23}
\|W(t,s)x\|_{X}\geq Me^{-\lambda(t-s)}\|x\|_{X},\qquad x \in X,\quad t \ge s\ge 0,
\end{equation}
then
\begin{equation}\label{rhnkdrnhp}
Z(t,x):=\ln(1+V(t,x)),\qquad x \in X,\quad t\geq0,
\end{equation}
is a coercive iISS Lyapunov function for (\ref{psiunb}). 
In particular,  (\ref{psiunb}) is iISS.
\end{thm}%
\begin{proof}
In view of \eqref{nmin23}, we have the following sandwich bounds:
\begin{align}
\label{eq:Sandwich-bounds-V}
\frac{2\gamma^{2}\lambda}{ M^{2}} \|x\|^{2}_{X} \leq V(t,x)&\leq\int_{t}^{\infty}\|W(\tau,t)\|^{2}\|x\|^{2}_{X}d\tau \leq\frac{k^{2}}{2w}\|x\|^{2}_{X},\qquad x \in X,\quad t\geq 0.
\end{align}

Defining $Z$ as in (\ref{rhnkdrnhp}) yields for any $u\in\mathcal{U}$ and any $t\geq 0$ the estimate:
\begin{eqnarray}\label{hkjkklhjj}
  \dot{Z}_{u}(t,x) &=& \frac{1}{1+V(t,x)} \dot{V}_{u}(t,x).
\end{eqnarray}
Using a similar analysis as in the proof of Theorem \ref{hdkmnmlo}, with the difference that instead of $B(t)u(t)$, we consider $B(t)u(t) + \psi(t,x(t),u(t))$ in (\ref{nai23mr}), we obtain the following estimate for any $\eta > 0$
$$
\dot{V}_{u}(t,x)\leq-\|x\|_{X}^{2}+\int_{t}^{\infty}\eta\|W(\tau,t)x\|_{X}^{2}d\tau+\frac{1}{\eta}\int_{t}^{\infty}\|W(\tau,t)\|_{X}^{2}\|B(t)u(t)+\psi(t,x(t),u(t)\|_{X}^{2}d\tau.
$$
Since the evolution family $\{W(t,s)\}_{t\geq s\geq0}$ generated by $\{A(t)\}_{t\geq0}$ is uniformly exponentially stable, there exist $k,w>0$ such that $\|W(t,t_0)\|\leq ke^{-\omega(t-t_0)} $ holds. It follows that
$$
\dot{V}_{u}(t,x)\leq\bigg{(}\frac{nk^{2}}{2w}-1\bigg{)}\|x\|_{X}^{2}+\frac{k^{2}}{2\eta w}\|B(t)u(t)+\psi(t,x(t),u(t)\|_{X}^{2}.
$$
Using the inequality $\|a+b\|_{X}^{2}\leq2\|a\|_{X}^{2}+2\|b\|_{X}^{2},$
for all $a, b\in X,$ we obtain
\begin{align*}
\dot{V}_{u}(t,x)&\leq\bigg{(}\frac{nk^{2}}{2w}-1\bigg{)}\|x\|_{X}^{2}+\frac{k^{2}}{\eta w} \big{(}\|B(t)\|^{2}\|u(t)\|^{2}_{U} + \|\psi(t,x(t),u(t)\|_{X}^{2} \big{)} \\
&\leq\bigg{(}\frac{nk^{2}}{2w}-1\bigg{)}\|x\|_{X}^{2}+\frac{k^{2}}{\eta w} \Big{(}\|B\|^{2}_{\infty}\|u(t)\|^{2}_{U}+{\gamma}^{2} \|x\|_{X}^{2}\delta^{2}(\|u(t)\|_{U})\Big{)}.
\end{align*}
Now (\ref{hkjkklhjj}) yields
\begin{equation}\label{hkdrharh32}
 \dot{Z}_{u}(t,x)= \frac{1}{1+V(t,x)}\bigg{(}\bigg{(}\frac{nk^{2}}{2w}-1\bigg{)}\|x\|_{X}^{2}+\frac{k^{2}}{\eta w} \Big{(}\|B\|^{2}_{\infty}\|u(t)\|^{2}_{U}+{\gamma}^{2} \|x\|_{X}^{2}\delta^{2}(\|u(t)\|_{U})\Big{)}\bigg{)}.
\end{equation}
Pick $\eta\in(0,\frac{2w}{k^{2}}),$ such that $\frac{nk^{2}}{2w}-1<0$ and due to (\ref{hkdrharh32}), we get
\begin{align*}
\dot{Z}_{u}(t,x)&\leq\bigg{(}\frac{nk^{2}}{2w}-1\bigg{)} \bigg{(} \frac{\|x\|_{X}^{2}}{1+\frac{k^{2}}{2 w}\|x\|_{X}^{2}}\bigg{)}+\frac{k^{2}}{\eta w}\|B\|^{2}_{\infty}\|u(t)\|^{2}_{U}+\gamma^{2}\bigg{(}\frac{\|x\|_{X}^{2}}{1+V(t,x)}\bigg{)}\delta^{2}(\|u(t)\|_{U})
\end{align*}
Using sandwich bounds \eqref{eq:Sandwich-bounds-V}, we have
\begin{align*}
\dot{Z}_{u}(t,x)&\leq\bigg{(}\frac{nk^{2}}{2w}-1\bigg{)} \bigg{(} \frac{\|x\|_{X}^{2}}{1+\frac{k^{2}}{2 w}\|x\|_{X}^{2}}\bigg{)}+\frac{k^{2}}{\eta w}\|B\|^{2}_{\infty}\|u(t)\|^{2}_{U}+\gamma^{2}\frac{\|x\|_{X}^{2}}{1+\frac{M^{2}}{2\lambda}\|x\|_{X}^{2}}\delta^{2}(\|u(t)\|_{U})\\
&\leq\bigg{(}\frac{nk^{2}}{2w}-1\bigg{)} \bigg{(} \frac{\|x\|_{X}^{2}}{1+\frac{k^{2}}{2 w}\|x\|_{X}^{2}}\bigg{)}+\frac{k^{2}}{\eta w}\|B\|^{2}_{\infty}\|u(t)\|^{2}_{U}+\frac{2\gamma^{2}\lambda}{ M^{2}}\delta^{2}(\|u(t)\|_{U}).
\end{align*}
This establishes the dissipativity inequality for $V$. In view of \eqref{eq:Sandwich-bounds-V}, the map $Z$ satisfies the sandwich bounds. 
Hence, $Z$ is a coercive iISS Lyapunov function for (\ref{psiunb}).
According to Theorem \ref{hjljg}, (\ref{psiunb}) is iISS.
\end{proof}

\section{\sectiontitle{Examples}}


In this subsection, we will analyze two examples showing the applicability of our methods.
\begin{exm}
Consider
the controlled time-varying Kuramoto-Sivashinsky (KS) equation
\begin{equation}\label{1EX}
\displaystyle\frac{\partial x(z,t)}{\partial t}=-\displaystyle\frac{\partial^{4} x(z,t)}{\partial z^{4}}-\varrho \displaystyle\frac{\partial^{2} x(z,t)}{\partial z^{2}}-\mu(t)x(z,t)+\frac{x(z,t)|\sin(t)|}{1+e^{-z t}x^{2}(z,t)}u(z,t),\quad z\in(0,1),\ t\geq t_0>0,\\
\end{equation}
where $t_0$ is the initial time, $\varrho > 0$ is known as the anti-diffusion parameter, $\mu:\mathbb{R}_+\to \mathbb{R}_+$ is a continuously differentiable function.
Consider (\ref{1EX}) with homogenous Dirichlet boundary conditions:
\begin{equation}\label{R1jjd}
x(0,t)=\displaystyle\frac{\partial x}{\partial z}(0,t)=x(1,t)=\displaystyle\frac{\partial x}{\partial z}(1,t)=0,\quad t\geq t_0>0.
\end{equation}
Let $X=L^{2}(0,1)$ and $U=PC(0,1).$

Define 
$$
B:=-\displaystyle\frac{\partial^{4}}{\partial z^{4}}-\varrho \displaystyle\frac{\partial^{2}}{\partial z^{2}},
$$
with domain $D(B)=H^{4}(0,1)\cap H_0^{2}(0,1)$. 

Define the family of operators $\{A(t)\}_{t\geq 0}$ by
$$\displaystyle A(t)f:=Bf-\mu(t)f,$$
with domain $D(A(t))=D(B),\; t\geq0$.

Also, let
$$
\Psi(t,x(z,t),u(z,t)):=\frac{x(z,t)|\sin(t)|}{1+e^{-z t}x^{2}(z,t)}u(z,t).
$$
It is well known from \cite{cerpa2010null,cerpa2011local,liu2001stability} that $B$ is the infinitesimal generator of a strongly
continuous semigroup on $L^{2}(0,1)$ that we denote by $S(t)$. 

Furthermore, one can verify directly that $\{A(t)\}_{ t\geq 0}$ generates a strongly continuous evolution family $\{W(t,s)\}_{t\geq s\geq 0}\subset L(X)$ of the form:
$$
W(t,s)x=S(t-s)\bigg{(}\exp\left(-\int_{s}^{t}\mu(r)dr\right)\bigg{)}x \quad \forall t\geq s\geq 0 \quad \forall x\in X.
$$

It follows from Proposition \ref{propooo1} that the system (\ref{1EX}) has a unique maximal mild solution $\phi(\cdot,0,x_0,u)$ and this mild solution is even a piecewise classical solution satisfying the integral equation (\ref{bvbbLL4}).

For a given $\varrho \in \mathbb{R},$ it can be demonstrated, as shown in \cite{liu2001stability}, that the operator $-B$ with the Dirichlet boundary conditions
(\ref{R1jjd}) has a countable sequence of eigenvalues $(\sigma_n)_{n\in \mathbb{N}},$ such that $\displaystyle\lim_{n\rightarrow\infty}\sigma_n=\infty.$
We define
$$
\sigma(\varrho )=\min_{n \in \N}\sigma_n(\varrho).
$$
Take
$$
Z(t,x)=(1+e^{-t}) \int_{0} ^{1 }x^{2}(z)dz,\quad x\in L^{2}(0,1),\ t\geq 0.
$$
Obviously,
$$
\|x\|_{L^{2}(0,1)}^{2}\leq Z(t,x)\leq 2\|x\|_{L^{2}(0,1)}^{2} \quad \forall (t,x)\in \mathbb{R}_{+}\times L^{2}(0,1).$$

For smooth $x\in \mathcal{D}$ and $u\in\Uc$, the Lie derivative of $Z$ along the solution of the system (\ref{1EX}) leads to
\begin{eqnarray*}
\dot{Z}_{u}(t,x)&=&
-e^{-t}\|x\|_{L^{2}(0,1)}^{2}\\
   &&+2(1+e^{-t}) \int _{0}^{1 }x(z)\left(-\displaystyle\frac{\partial^{4} x(z)}{\partial z^{4}}-\varrho \displaystyle\frac{\partial^{2} x(z)}{\partial z^{2}}-\mu(t)x(z)+\frac{x(z)|\sin(t)|}{1+e^{-zt}x^{2}(z)}u(z,t)\right)dz.
\end{eqnarray*}
Since $1+e^{-z t}x^{2}(z)\geq 1,$ we have
\begin{eqnarray}\label{1nnnX}
\dot{Z}_{u}(t,x) &\leq&-2(1+e^{-t})\int _{0}^{1 }x(z)\frac{\partial^{4} x(z)}{\partial z^{4}}dz-2\varrho (1+e^{-t})\int _{0}^{1 }x(z) \displaystyle\frac{\partial^{2} x(z)}{\partial z^{2}}dz \nonumber\\
&&\ \ \ \ \ \ \ \ \ \ \ \ \ + 2(1+e^{-t})\int _{0}^{1 }x^{2}(z) u(z,t)dz.
\end{eqnarray}
Partial integration of the first and the second terms of (\ref{1nnnX}) together with the Dirichlet boundary condition (\ref{R1jjd}) leads to
\begin{equation}\label{R1jjjd}
\int _{0}^{1 }x(z)\frac{\partial^{4} x(z)}{\partial z^{4}}dz=\int _{0}^{1 }\left(\displaystyle\frac{\partial^{2} x(z)}{\partial z^{2}}\right)^{2}dz,
\end{equation}
and
\begin{equation}\label{R2jjjd}
 \int _{0}^{1 }x(z)\displaystyle\frac{\partial^{2} x(z)}{\partial z^{2}}dz=-\int _{0}^{1 }\left(\displaystyle\frac{\partial x(z)}{\partial z}\right)^{2}dz.
\end{equation}
Combining (\ref{R1jjjd}) and (\ref{R2jjjd}), we find
$$\dot{Z}_{u}(t,x)\leq -2(1+e^{-t})\left(\int _{0}^{1 }\left(\displaystyle\frac{\partial^{2} x(z)}{\partial z^{2}}\right)^{2}dz-\varrho \int _{0}^{1}\left(\displaystyle\frac{\partial x(z)}{\partial z}\right)^{2}dz\right)+2Z(t,x)\|u(\cdot,t)\|_{U}.$$
By Lemma 3.1 in \cite{liu2001stability}, we obtain
$$
\dot{Z}_{u}(t,x)\leq-2\sigma(\varrho )Z(t,x)+2Z(t,x)\|u(\cdot,t)\|_{U}.
$$
Let $\varrho <4\pi^{2}.$ Then, as stated in \cite[Lemma 2.1]{liu2001stability}, $\sigma(\varrho )>0.$
Consider the following coercive candidate iISS Lyapunov function:
$$
V(t,x):=\ln(1+Z(t,x)),\quad x \in X, \ t\geq 0.
$$
Then,
\begin{eqnarray*}
\dot{V}_{u}(t,x)&\le& \frac{-2\sigma(\varrho )Z(t,x)}{1+Z(t,x)}+2\frac{Z(t,x)}{1+Z(t,x)}\|u\|_{\mathcal{U}}\\
&\le&\frac{-2\sigma(\varrho )\|x\|_{L^{2}(0,1)}^{2}}{1+2\|x\|_{L^{2}(0,1)}^{2}}+2\|u\|_{\mathcal{U}}\\
&=&-\vartheta(\|x\|_{L^{2}(0,1)})+ \chi(\|u\|_{\mathcal{U}}),
\end{eqnarray*}
where $\vartheta(s)=\frac{\sigma(\varrho )s^{2}}{1+2s^{2}}$ and $\chi(s)=2s.$
Theorem \ref{hjljg} now shows that (\ref{1EX}), (\ref{R1jjd}) is iISS when $$\varrho <4\pi^{2}.$$
\end{exm}

\begin{exm}\label{expl2}
Let $\nu>0,$ $\ell>0,$ $X=L^{2}(0,\ell)$ and as the input space we take $\mathcal{U}=PC(\mathbb{R}_+, U).$ We consider the controlled heat equation
\begin{equation}\label{22EX}
\left\lbrace
\begin{array}{l}
\displaystyle\frac{\partial x(z,t)}{\partial t}=\nu\frac{\partial^{2} x(z,t)}{\partial z^{2}}+R(t)x(z,t)+\omega\sin(tz)x(z,t)+u(z,t),\quad z\in(0,\ell),\quad t\geq t_0>0,\\
\\
\displaystyle x(0,t)=0=x(\ell,t),
\end{array}\right.
\end{equation}
where $\omega\in\mathbb{R},$ $t_0$ is the initial time and $\{R(t)\}_{t\geq 0}$ is a family of linear bounded operators on $X,$ satisfying $\sup_{t \in \mathbb{R}_{+} }\|R(t)\|=r<\infty,$ and such that $t \mapsto R(t)$ is a continuously differentiable map.

Denote
$$
B:=\nu\displaystyle\frac{\partial^{2}}{\partial z},
$$ 
with domain 
$D(B)=H_{0}^{1}(0,\ell)\cap H^{2}(0,\ell)$.

Furthermore, define $\{A(t)\}_{t\geq 0}$ as
$$
\displaystyle A(t):=B+R(t),
$$ 
with domain $D(A(t))=D(B),\; t\geq0$. 

Finally, define the following map:
$$
\Psi: (t,x(z,t),u(z,t))\mapsto \omega\sin(tz)x(z,t)+u(z,t).
$$
By a classical bounded perturbation result \cite[Theorem 3.2.7 p. 119]{curtain2012introduction}, the family $\{A(t)\}_{t\geq 0}$ generates an evolution family
$\{W(t,s)\}_{t\geq s\geq 0}\subset L(X)$ of the form:
$$
W(t,s)x=S(t-s)x+\int_{s}^{t}S(t-\tau)R(\tau)W(\tau,s)xd\tau\quad \forall t\geq s\geq0\quad \forall x\in X,
$$
where $S(t)$ represents the $C_0$-semigroup on $X$ generated by the operator $B.$

It is clear that the nonlinearity $\Psi(t,x,u)$ is continuous in $t$ and $u$ and locally Lipschitz continuous in $x,$ uniformly in $t$
and $u$ on bounded sets. 

By Proposition \ref{propooo1}, the system (\ref{22EX}) is a well-posed control system satisfying the BIC property. 

Consider the following coercive ISS Lyapunov function candidate:
$$V(t,x)=\|x\|_{L_{2}(0,\ell)}^{2}=\int_{0}^{\ell }x^{2}(z)dz.$$
For smooth $u\in\Uc$ and for $(t,x)\in\mathbb{R}_{+}\times D(B),$ 
the Lie derivative of $V$ with respect to the system (\ref{22EX}) and using the integration by parts, we have
\begin{eqnarray*}
\dot{V}_{u}(t,x)&=&
2\int _{0}^{\ell}x(z)\left(\nu\frac{\partial^{2} x(z)}{\partial z^{2}}+R(t)x(z)+\omega\sin(t z)x(z)+u(z,t)\right)dz\\
&\leq&-2\nu\int _{0}^{\ell}\left(\frac{\partial x(z)}{\partial z}\right)^{2}dz+(2r+\omega)\int _{0}^{\ell}x^{2}(z)dz+2\int _{0}^{\ell}x(z)u(z,t)dz.
\end{eqnarray*}
Utilizing Friedrich's inequality (see \cite{mitrinovic1991inequalities}) in the first term, we continue estimates:
$$\dot{V}_{u}(t,x)\leq-\frac{2\nu\pi^{2}}{\ell^{2}}V(t,x)+(2r+\omega)V(t,x)+2\int _{0}^{\ell}x(z)u(z,t)dz.$$
Using Young's inequality, we have for any $\epsilon>0,$
\begin{eqnarray*}
\dot{V}_{u}(t,x)&\leq&
\left(-\frac{2\nu\pi^{2}}{\ell^{2}}+2(r+\omega)+\epsilon\right)V(t,x)+\frac{1}{\epsilon}\int _{0}^{\ell}u^{2}(z,t)dz\\
&\leq&\left(-\frac{2\nu\pi^{2}}{\ell^{2}}+2(r+\omega)+\epsilon\right)V(t,x)+\frac{\ell}{\epsilon}\|u(\cdot,t)\|_{U}^{2}.
\end{eqnarray*}
To show that $V$ is an ISS Lyapunov function for the equation (\ref{22EX}), we assume that
$-\frac{2\nu\pi^{2}}{\ell^{2}}+2(r+\omega)+\epsilon<0.$ As $\epsilon>0$ can be chosen arbitrarily small, we obtain the following sufficient condition for ISS of the system (\ref{22EX}):
$$r+\omega<\frac{\nu\pi^{2}}{\ell^{2}}.$$
\end{exm}

\chapter{ISpS Criteria For Time-Varying Non-Linear Infinite-Dimensional Systems}
\label{ch:chapter3}
\markboth{ISPS CRITERIA For TIME-VARYING NON-LINEAR}{ }

In this Chapter, we consider the ISpS for time-varying nonlinear infinite-dimensional systems and then
obtain some novel criteria using Lyapunov functions techniques and a nonlinear inequality. We 
show if a control system possesses the CpUAG property, then it has an ISpS property. We show that the existence of an ISpS-Lyapunov function implies the ISpS of a system and the CpUAG characterization. Then, we develop a method for constructing ISpS-Lyapunov functions for a certain class of time-varying semilinear evolution equations. Also, we study ISpS of interconnected systems in the case when explicit Lyapunov functions are known and provide a generalization of the Lyapunov small-gain theorem for the case of finite-dimensional systems. 
\section{\sectiontitle{Lyapunov Theorems for ISpS and iISpS}}
In this section, sufficient conditions are established for the ISpS and iISpS of the system $\Sigma$ by using the Lyapunov theory with piecewise-right continuous inputs. \\
The following theorem provides a characterization of ISpS.
\begin{thm}\label{tmddff1}
If the control system $\Sigma=(X,\mathcal{U},\phi)$ is completely practical uniform asymptotic gain property (CpUAG), then $\Sigma$ is ISpS.
\end{thm}
\begin{proof}
Let $\Sigma=(X,\mathcal{U},\phi)$ be a control system which is CpUAG. Then, there are $\beta\in\mathcal{KL},$ $\gamma\in \mathcal{K},$ and $c,\varsigma>0,$ such that for all $x_0\in X,$ all $u\in \mathcal{U},$ all $t_0\geq 0$ and all $t\geq t_0,$ we have
\begin{eqnarray*}
\|\phi(t,t_0,x_0,u)\|_{X}&\leq &
\beta(\|x_0\|_{X}+c,t-t_0)+\gamma(\|u\|_{\mathcal{U}})+\varsigma\\
&\leq&\beta(2\|x_0\|_{X},t-t_0)+\beta(2c,t-t_0)+\gamma(\|u\|_{\mathcal{U}})+\varsigma\\
&\leq& \beta(2\|x_0\|_{X},t-t_0)+\gamma(\|u\|_{\mathcal{U}})+\beta(2c,0)+\varsigma,
\end{eqnarray*}
which shows ISpS of $\Sigma.$
\end{proof}
\begin{thm}\label{dfffdddf1}
Consider the control system $\Sigma=(X,PC(\mathbb{R}_+ ,U),\phi).$ Assume that there exist continuously differentiable function $V:\mathbb{R}_+\times X\to \mathbb{R}_+,$ functions $\psi_1,\psi_2,\delta\in\mathcal{K}_{\infty},$ a function $\kappa\in\mathcal{K}$ and
a continuous integrable function $\ell:\mathbb{R}_+ \to \mathbb{R}_+,$ such that, for all $(t,x)\in \mathbb{R}_+\times X$ and all $u\in PC(\mathbb{R}_+,U),$
\begin{enumerate}
\item[$(i)$] \begin{equation}
\psi_1(\|x\|_{X})\le V(t,x)\le \psi_2(\|x\|_{X}).
\label{llMMMM}
\end{equation}
\item[$(ii)$] There holds
\begin{equation}
\dot{V}_{u}(t,x)\leq -\delta(V(t,x))+\ell(t),\; \hbox{whenever}\; V(t,x)\geq \kappa(\|u(t)\|_{U}).
\label{rama1bbbbb}
\end{equation}
\end{enumerate}
Then, system $\Sigma$ is ISpS.\\
Moreover, if $\displaystyle \lim_{t\to \infty} \ell(t)=0,$ then
the solution $\phi(t,t_0,x_0,u)$ of the system $\Sigma$ satisfies
\begin{equation}
\|\phi(t,t_0,x_0,u)\|_{X}\leq\beta( \psi_2(\|x_0\|_{X})+\int_{ 0} ^{\infty}\ell(s)ds,t-t_0)+\gamma(\|u\|_{[t_0,t]})+\varsigma,
\label{lllll}
\end{equation}
where $\beta$ is a class $\mathcal{KL}$ function on $\mathbb{R}_+\times \mathbb{R}_+,$ $\gamma$ is a class $\mathcal{K}$ function on $\mathbb{R}_+$ and $\varsigma$ is a positive constant.\\
The estimate (\ref{lllll}) shows CpUAG which implies that the system $\Sigma$ is ISpS.
\end{thm}
\begin{proof}
For given initial time $t_0,$ initial condition $x_0\in X,$ and an input $u\in PC(\mathbb{R}_+ ,U),$ let $y(t)=V(t,x(t)),$ where $x(t)=\phi(t,t_0,x_0,u).$ Condition $(ii)$ shows that for all $t_0\geq 0$ and all $t\geq t_0,$
\begin{center}
$\dot{y}(t)\leq -\delta( y(t)) +\ell(t),$ when $y(t)\geq \kappa(\|u(t)\|_{U}).$
\end{center}
For any $t\in[t_0,\infty),$ we consider the following inequality
\begin{equation}
y(s)\geq \kappa(\|u(s)\|_{U}) \quad \forall s\in[t_0,t].
\label{vnnR1CCC}
\end{equation}
Suppose that the inequality (\ref{vnnR1CCC}) is true, then it follows from Proposition \ref{propbbvcc} that there exists $\sigma_1\in\mathcal{KL},$ such that for all $x_0\in X,$ all $t_0\geq 0$ and all $t\geq t_0,$
\begin{equation}
y(t)\leq \sigma_1(y(t_0),t-t_0)+2\int_{0 }^{\infty}\ell(s)ds.
\label{bbbbR1CCC}
\end{equation}
Now, we assume that the inequality (\ref{vnnR1CCC}) does not hold. Let us consider the following set $\sup\{s\in [t_0,t]: y(s)\leq\kappa(\|u(s)\|_{U})\}$ which is non-empty. \\
Denote
$$t^{\ast}=\sup\{s\in [t_0,t]: y(s)\leq\kappa(\|u(s)\|_{U})\}.$$
Thus, we get either $t^{\ast}=t$ or $t^{\ast}< t.$ \\
If $t^{\ast}=t,$ then from the definition of $t^{\ast}$ we get
\begin{equation}
y(t)=y(t^{\ast})\leq \kappa(\|u(t^{\ast})\|_{U})\leq\kappa(\|u\|_{[t_0,t]}).
\label{jjjjnhyy}
\end{equation}
If $t^{\ast}< t,$ then $y(s)\geq \kappa(\|u(s)\|_{U}),\; s\in[t^{\ast},t],$ which by (\ref{rama1bbbbb}) implies
$$\dot{y}(s)\leq -\delta( y(s)) +\ell(s) \quad \forall s\in [t^{\ast},t].$$
Using Proposition \ref{propbbvcc}, there exists $\sigma_2\in\mathcal{KL},$ such that
$$y(t)\leq \sigma_2(y(t^{\ast}),t-t^{\ast})+2\int_{0 }^{\infty}\ell(s)ds.$$
Then, the following holds:
\begin{equation}
y(t)\leq \sigma_2(\kappa(\|u\|_{[t_0,t]}),0)+2\int_{0 }^{\infty}\ell(s)ds.
\label{nnnnnhyy}
\end{equation}
Hence, we get from (\ref{bbbbR1CCC}), (\ref{jjjjnhyy}) and (\ref{nnnnnhyy}) that
$$y(t)\leq \sigma_1(y(t_0),t-t_0)+\sigma(\|u\|_{[t_0,t]})+2\int_{0 }^{\infty}\ell(s)ds,$$
where $\sigma(s)=\max(\kappa(s),\sigma_2(\kappa(s),0)).$\\
Since $\psi_1$ is of class $\mathcal{K}_{\infty}$ and satisfies $\psi_1^{-1}(a+b+c)\leq\psi_1^{-1}(2a)+\psi_1^{-1}(4b)+\psi_1^{-1}(4c)$ for any $a,b,c\in\mathbb{R}_{+},$ we have
$$\|\phi(t,t_0,x_0,u)\|_{X}\leq \beta(\|x_0\|_{X},t-t_0)+\gamma(\|u\|_{[t_0,t]})+\eta,$$
where $$\beta(r,s)=\psi_1^{-1}(2\sigma_1(\psi_2(r),s)), \quad \gamma(r)=\psi_1^{-1}(\sigma(4r))$$ and $$\eta=\psi_1^{-1}(8\int_{0 }^{\infty}\ell(s)ds).$$ Thus, system $\Sigma$ is ISpS.\\
Now, we consider the case when $\ell(t)$ is a continuous integrable function with $\displaystyle \lim_{t\to \infty} \ell(t)=0.$
If the inequality (\ref{vnnR1CCC}) is true, we obtain from $(ii)$ and Proposition \ref{propbbvcc} that there exists $\sigma_1\in\mathcal{KL},$ such that for all $t_0\geq 0$ and all $t\geq t_0,$
\begin{equation}\label{bMLOPPP}
y(t)\leq \sigma_1(y(t_0)+\int_{0 }^{\infty}\ell(s)ds,t-t_0).
\end{equation}
If the inequality (\ref{vnnR1CCC}) is not true, then it can be obtained by repeating the above process that there exists $\sigma_2\in\mathcal{KL},$ such that
\begin{equation}\label{KJYFFy}
y(t)\leq \sigma_2(2\kappa(\|u\|_{[t_0,t]}),0)+\sigma_2(2\int_{0 }^{\infty}\ell(s)ds,0).
\end{equation}
Then, via (\ref{jjjjnhyy}), (\ref{bMLOPPP}) and (\ref{KJYFFy}), we further have
$$y(t)\leq \sigma_1(y(t_0)+\int_{0 }^{\infty}\ell(s)ds,t-t_0)+\sigma(\|u\|_{[t_0,t]})+\sigma_2(2\int_{0 }^{\infty}\ell(s)ds,0),$$
where $\sigma(s)=\max(\kappa(s),\sigma_2(2\kappa(s),0)).$\\
Therefore,
\begin{equation}\label{KJYbbbbFFy}
\|\phi(t,t_0,x_0,u)\|_{X}\leq\beta( \psi_2(\|x_0\|_{X}+\int_{0}^{\infty}\ell(s)ds,t-t_0)+\gamma(\|u\|_{[t_0,t]})+\varsigma,
\end{equation}
where $\beta=\psi_1^{-1}\circ2\sigma_1,$ $\gamma=\psi_1^{-1}\circ4\sigma$ and $\varsigma=\psi_1^{-1}(4\sigma_2(2\int_{0 }^{\infty}\ell(s)ds,0)).$
The inequality (\ref{KJYbbbbFFy}) shows that $\Sigma$ satisfies the CpUAG property and hence $\Sigma$ is ISpS.
\end{proof}
\begin{rem}
The function $\kappa$ is called ISpS Lyapunov gain for $\Sigma$ and function $V$ is called ISpS-Lyapunov function for the system $\Sigma$ in an implicative form. On the other hand, if we choose $\ell(t)=0$ then it is easy to verify that the system $\Sigma$
is ISS.
\end{rem}
\begin{thm}\label{dfffdddf1}
Consider the control system $\Sigma=(X,PC(\mathbb{R}_+,U),\phi),$ with the BIC property. Assume that there exists a continuous Lyapunov function $V:\mathbb{R}_+\times X\to \mathbb{R}_+,$ functions $\alpha_1,\alpha_2\in\mathcal{K}_{\infty},$ a constant $c\geq0,$ a function $\kappa\in\mathcal{K}$ and
functions $\nu,\psi\in C(\mathbb{R}_+, \mathbb{R}),$ such that, for all $(t,x)\in \mathbb{R}_+\times X$ and all $u\in PC(\mathbb{R}_+,U),$
\begin{enumerate}
\item[$(i)$]
\begin{equation}\label{rama1}
\alpha_1(\|x\|_{X})\le V(t,x)\le \alpha_2(\|x\|_{X})+c,
\end{equation}
\item[$(ii)$] the Lie derivative of $V$ along the trajectories of system $\Sigma$ satisfies
\begin{equation}\label{rama1bbb}
\dot{V}(t,x)\leq \nu(t) V(t,x)+\psi(t),\; \hbox{whenever}\; V(t,x)\geq \kappa(\|u(t)\|_{U}).
\end{equation}
\item[$(iii)$] There exist $\eta>0,$ $\xi\geq 0$ and $\rho>0,$ such that, for all $t\geq t_0\geq 0,$
$$\displaystyle \int _{t_0}^{t} \nu(\tau) d\tau\leq -\eta(t-t_0)+\xi,$$
and
$$\int _{t_0}^{t}\vert \psi(\tau)\vert\Theta(t,\tau)d\tau\leq\rho,$$ where $\Theta(t,\tau)=\exp(\displaystyle\int _{\tau}^{t}\nu(s)ds).$
\end{enumerate}
Then, $\Sigma$ is ISpS.
\end{thm}
\begin{proof}
Let $d\in(t_0,\infty)$ be such that $[t_0,d)$ is the maximal time interval over which a system $\Sigma$ admits a solution. For given initial time $t_0,$ initial condition $x_0\in X,$ and an input $u\in\mathcal{U},$ let $y(t)=V(t,x(t)),$ where $x(t)=\phi(t,t_0,x_0,u).$ Condition $(ii)$ shows that
\begin{center}
$D^{+}y(t)\leq  \nu(t)  y(t)+\psi(t),$ when $y(t)\geq \kappa(\|u(t)\|_{U}),$
\end{center}
for all $t\in [t_0,d),$ where $t_0< d\leq \infty.$\\
For any $t\in[t_0,d),$ we consider the following inequality
\begin{equation}\label{vnnR1CCC}
y(s)\geq \kappa(\|u(s)\|_{U}),\quad \forall s\in[t_0,t].
\end{equation}
Suppose the inequality (\ref{vnnR1CCC}) is true, then it follows from Lemma \ref{prop3vvv} that
$$y(t)\leq y(t_0)e^{\int_{t_0}^{ t} \nu(s)ds}+\int _{t_0}^{t} \psi(s)e^{\int_{s}^{ t} \nu(\tau)d\tau}ds.$$
Thus, from Condition $(iii)$ and (\ref{rama1}), we can get
\begin{equation}\label{bbbbR1CCC}
y(t)\leq \alpha_2(\|x_0\|_{X})e^{\xi}e^{-\eta(t-t_0)}+\rho+ce^{\xi}.
\end{equation}
Now, we assume that the inequality (\ref{vnnR1CCC}) does not hold. Let us consider the following set $\sup\{s\in [t_0,t]: y(s)\leq\kappa(\|u(s)\|_{U})\}$ which is non-empty. \\ Denote
$$\hat{t}=\sup\{s\in [t_0,t]: y(s)\leq\kappa(\|u(s)\|_{U})\}.$$
Then, we get either $\hat{t}=t$ or $\hat{t}< t.$
If $\hat{t}=t,$ it follows from the definition of $\hat{t}$ that
$$y(t)=y(\hat{t})\leq \kappa(\sup_{\nu\in[t_0,t]} \|u(\nu)\|_{U}).$$
If $\hat{t}< t,$ then $y(s)\geq \kappa(\|u(s)\|_{U}),\quad s\in[\hat{t},t],$ which, by (\ref{rama1bbb}), implies
$$D^{+}y(s)\leq  \nu(s)  y(s)+\psi(s) \quad \forall s\in[\hat{t},t].$$
Then, Lemma \ref{prop3vvv} yields
$$y(t)\leq y(\hat{t})e^{\int_{\hat{t}} ^{t} \nu(s)ds}+\int _{\hat{t}}^{t} \psi(s)e^{\int_{s}^{ t} \nu(\tau)d\tau}ds\leq y(\hat{t})e^{\int_{\hat{t}} ^{t} \nu(s)ds}+\rho.$$
Thus,
\begin{equation}\label{nnnnnhyy}
y(t)\leq \kappa(\sup _{t_0\leq s\leq t}\|u(s)\|_{U})e^{\int_{\hat{t}} ^{t} \nu(s)ds}+\rho\leq \kappa(\sup _{t_0\leq s\leq t}\|u(s)\|_{U})e^{\xi}+\rho.
\end{equation}
Hence, we get from (\ref{bbbbR1CCC}) and (\ref{nnnnnhyy}) that
$$y(t)\leq \alpha_2(\|x_0\|_{X})e^{\xi}e^{-\eta(t-t_0)}+\kappa(\sup_{t_0\leq s\leq t}(\|u(s)\|_{U})e^{\xi}+2\rho+ce^{\xi}.$$
Since $\alpha_1$ is of class $\mathcal{K}_{\infty}$ and satisfies $\alpha_1^{-1}(a+b+c)\leq\alpha_1^{-1}(2a)+\alpha_1^{-1}(4b)+\alpha_1^{-1}(4c)$ for any $a,b,c\in\mathbb{R}_{+},$ we have
\begin{equation}\label{nyasmine}
\|x(t)\|_{X}\leq \beta(\|x_0\|_{X},t-t_0)+\gamma(\sup _{t_0\leq s\leq t}\|u(s)\|_{U})+r,
\end{equation}
where $\beta(r,s)=\alpha_1^{-1}(2\alpha_2(r)e^{\xi}e^{-\eta s}),$ $\gamma(r)=\alpha_1^{-1}(4\kappa(r)e^{\xi})$ and $r=\alpha_1^{-1}(8\rho+4ce^{\xi}).$ Therefore, $x(t)$ is uniformly bounded at the maximal interval of existence, and BIC property implies that $d=\infty,$ and the bound (\ref{nyasmine}) holds for all $t\in[t_0,\infty), x_0\in X$ and $u\in PC(\mathbb{R}_+, U).$ Thus, system $\Sigma$ is ISpS.
\end{proof}
\begin{rem}
The function $\kappa$ is called ISpS Lyapunov gain for $\Sigma$ and function $V$ is called ISpS-Lyapunov function for the system $\Sigma$ in an implicative form. The biggest advantage of this approach is that the upper right-hand derivative of the Lyapunov function is neither required to be negative definite nor required to be
negative semi-definite. 
Moreover, if we choose $c=0$ and $\psi(t)=0,$ then it is easy to verify that the system $\Sigma$ is ISS. 
\end{rem}
\begin{cor}\label{MARIAM}
Consider the control system $\Sigma=(X,PC(\mathbb{R}_+,U),\phi),$ with the BIC property. Assume that there exists a continuous Lyapunov function $V:\mathbb{R}_+\times X\to \mathbb{R}_+,$ constants $b_1>0,$ $b_2\geq0,$ $c\geq 0$ and $m\geq1,$ a function $\kappa\in\mathcal{K}$ and functions $\nu,\psi\in C(\mathbb{R}_+, \mathbb{R}),$ such that, for all $(t,x)\in \mathbb{R}_+\times X$ and all $u\in PC(\mathbb{R}_+,U),$
\begin{enumerate}
\item[$(i)$]
\begin{equation}\label{rama2}
b_1\|x\|_{X}^{m}\le V(t,x)\le b_2\|x\|_{X}^{m}+c,
\end{equation}
\item[$(ii)$] the Lie derivative of $V$ along the trajectories of $\Sigma$ satisfies
$$\dot{V}(t,x)\leq \nu(t) V(t,x)+\psi(t),\; \hbox{whenever}\; V(t,x)\geq \kappa(\|u(t)\|_{U}).$$
\item[$(iii)$] Condition $(iii)$ in Theorem \ref{dfffdddf1} holds.
\end{enumerate}
Then, $\Sigma$ is exponentially ISpS (eISpS).
\end{cor}
Next, we present an ISpS result for the time-varying infinite-dimensional systems where the ISpS-Lyapunov function is introduced in dissipative form.
\begin{thm}\label{dffff1}
Consider the control system $\Sigma=(X,PC(\mathbb{R}_+,U),\phi).$ Assume that there exists a continuously differentiable function $V:\mathbb{R}_+\times X\to \mathbb{R}_+,$ functions $\psi_1,\psi_2,\delta\in\mathcal{K}_{\infty},$ a function $\theta\in\mathcal{K}$ and
a continuous integrable function $\ell:\mathbb{R}_+ \to \mathbb{R}_+,$ such that, for all $(t,x)\in \mathbb{R}_+\times X,$ and all $u\in PC( \mathbb{R}_+ ,U),$
\begin{enumerate}
\item[$(i)$] $$\psi_1(\|x\|_{X})\le V(t,x)\le \psi_2(\|x\|_{X}).$$
\item[$(ii)$] The following holds:
\begin{equation}
\dot{V}_{u}(t,x)\leq -\delta(V(t,x))+\ell(t)+\theta(\|u(t)\|_{U}).
\label{yesmine}
\end{equation}
\end{enumerate}
Then, system $\Sigma$ is ISpS. Moreover, if $\displaystyle \lim_{t\to \infty} \ell(t)=0,$ then the system $\Sigma$ has the CpUAG property.
\end{thm}
\begin{proof}
Let $V$ be a Lyapunov function in a dissipative form so that (\ref{yesmine}) is satisfied. Whenever $\|u(t)\|_{U}\leq\theta^{-1}(\frac{1}{2}\delta(V(t,x)))$ holds, which is the same as $V(t,x)\geq\delta^{-1}(2\theta(\|u(t)\|_{U})),$ the inequality
$$\dot{V}_{u}(t,x)\leq -\frac{1}{2}\delta(V(t,x))+\ell(t),$$
is verified, which means that $V$ is an ISpS-Lyapunov function in an implicative form with a Lyapunov gain $\kappa=\delta^{-1}\circ2\theta.$
\end{proof}
In the following, the sufficient conditions of uniform global practical asymptotic stability for
the system $\Sigma$ at the zero input are presented as by-products of Theorem \ref{dffff1}.
\begin{cor}\label{COR1}
Consider the control system $\Sigma=(X,PC(\mathbb{R}_+,U),\phi).$ Assume that there exists a continuously differentiable function $V:\mathbb{R}_+\times X\to \mathbb{R}_+,$ functions $\psi_1,\psi_2,\delta\in\mathcal{K}_{\infty}$ and
a continuous integrable function $\ell:\mathbb{R}_+ \to \mathbb{R}_+,$ such that, for all $(t,x)\in \mathbb{R}_+\times X,$
\begin{enumerate}
\item[$(i)$] $$\psi_1(\|x\|_{X})\le V(t,x)\le \psi_2(\|x\|_{X}).$$
\item[$(ii)$] There holds
\begin{equation}
\dot{V}_{u}(t,x)\leq -\delta(V(t,x))+\ell(t).
\label{MARIAM1}
\end{equation}
\end{enumerate}
Then, system $\Sigma$ is 0-UGpAS.\\
Moreover, if $\displaystyle \lim_{t\to \infty} \ell(t)=0,$ then
the solution $\phi(t,t_0,x_0,u)$ of the system $\Sigma$ satisfies
\begin{equation}
\|\phi(t,t_0,x_0,u)\|_{X}\leq\beta( \psi_2(\|x_0\|_{X})+\int_{ 0} ^{\infty}\ell(s)ds,t-t_0),
\label{llyasminKJe}
\end{equation}
where $\beta$ is a class $\mathcal{KL}$ function on $\mathbb{R}_+\times \mathbb{R}_+.$\\
The estimate $(\ref{llyasminKJe})$ implies that the system $\Sigma$ at the zero input is uniformly globally attractive and uniformly globally
practically stable.
\end{cor}
In the following, we present an iISpS result for the time-varying infinite-dimensional systems by employing an indefinite Lyapunov function.
\begin{thm}\label{dffff1}
Consider the control system $\Sigma=(X,C(\mathbb{R}_+,U),\phi),$ with the BIC property. Assume that there exists a continuous Lyapunov function $V:\mathbb{R}_+\times X\to \mathbb{R}_+,$ functions $\alpha_1,\alpha_2\in\mathcal{K}_{\infty},$ a function $\theta\in\mathcal{K},$ a constant $c\geq0,$ and functions $\nu,\psi\in C(\mathbb{R}_+, \mathbb{R}),$ such that, for all
$(t,x)\in \mathbb{R}_+\times X$ and all $u\in C(\mathbb{R}_+,U),$
\begin{enumerate}
\item[$(i)$] \begin{equation}\label{rama3}
\alpha_1(\|x\|_{X})\le V(t,x)\le \alpha_2(\|x\|_{X})+c,
\end{equation}
\item[$(ii)$] the Lie derivative of $V$ along the trajectories of system $\Sigma$ satisfies
\begin{equation}\label{ramffff}
\dot{V}(t,x)\leq \nu(t) V(t,x)+\psi(t)+\theta(\|u(t)\|_{U}).
\end{equation}
\item[$(iii)$] Condition $(iii)$ in Theorem \ref{dfffdddf1} holds.
\end{enumerate}
Then, $\Sigma$ is iISpS.
\end{thm}
\begin{proof}
Let $d\in(t_0,\infty)$ be such that $[t_0,d)$ is the maximal time interval over which a system $\Sigma$ admits a solution. \\ 
Set $y(t)=V(t,x(t)),$
where $x(t)=\phi(t,t_0,x_0,u).$ By (\ref{ramffff}), we can get
$$D^{+}y(t)\leq \nu(t) y(t)+\psi(t)+\theta(\|u(t)\|_{U}),$$
for all $t\in [t_0,d).$ \\ 
By Lemma \ref{prop3vvv}, we have
\begin{equation}\label{FFFKLLL}
y(t)\leq y(t_0)e^{\int_{t_0}^{t} \nu(s)ds}+\int _{t_0} ^{t}e^{\int _{s}^{t} \nu(\tau)d\tau}[\theta(\|u(s)\|_{U})+\psi(s)]ds.
\end{equation}
Then, via condition $(iii),$ inequality (\ref{FFFKLLL}) implies that
\begin{equation}\label{nyasminewwwwww}
y(t)\leq \alpha_2(\|x_0\|_{X})e^{-\eta(t-t_0)}e^{\xi}+\int _{t_0} ^{t}e^{\xi}\theta(\|u(s)\|_{U})ds+\rho+ce^{\xi}.
\end{equation}
Combining (\ref{nyasminewwwwww}) with (\ref{rama3}), we obtain
\begin{equation}\label{vRCCN1}
\|x(t)\|_{X}\leq \beta(\|x_0\|_{X},t-t_0)+\alpha\left(\int _{t_0} ^{t}\theta(\|u(s)\|_{U})ds\right)+r,
\end{equation}

where $\beta(\varsigma,s)=\alpha_1^{-1}(2e^{\xi}\alpha_2(\varsigma)e^{-\eta s})\in \mathcal{KL},$ $\alpha(\varsigma)=\alpha_1^{-1}(4e^{\xi}\varsigma)$ and $r=\alpha_1^{-1}(4\rho+4ce^{\xi}).$
Using (\ref{vRCCN1}) and the BIC property, we can establish that $d=\infty.$ Therefore, system $\Sigma$ is iISpS. 
\end{proof}
\begin{rem}
Condition $(ii)$ in Theorem \ref{dfffdddf1} implies condition $(ii)$ in Theorem \ref{dffff1}, but the converse may not be true. As a special case, if let $c=0$ and $\psi(t)=0$ in Theorem \ref{dffff1}, then the system  $\Sigma$ is iISS. 
\end{rem}
Next, we give sufficient conditions to guarantee the ISpS of $\Sigma$ where a particular ISpS Lyapunov function is introduced in Banach spaces.
\begin{thm}\label{MARIAM11}
Consider the control system $\Sigma=(X,PC(\mathbb{R}_+,U),\phi),$ with the BIC property. Assume that there exists a continuous Lyapunov function $V:\mathbb{R}_+\times X\to \mathbb{R}_+,$ functions $\alpha_1,\alpha_2\in\mathcal{K}_{\infty},$ constants $c\geq 0,$ $\varrho >0,$ $\epsilon\geq0$ and a function $\kappa\in\mathcal{K},$ such that, for all $(t,x)\in \mathbb{R}_+\times X$ and all $u\in PC(\mathbb{R}_+,U),$
\begin{enumerate}
\item[$(i)$]
\begin{equation}\label{YDFDFD}
\alpha_1(\|x\|_{X})\le V(t,x)\le \alpha_2(\|x\|_{X})+c,
\end{equation}
\item[$(ii)$] the Lie derivative of $V$ along the trajectories of system $\Sigma$ satisfies
\begin{equation}\label{YJJasmine}
\dot{V}(t,x)\leq -\varrho V(t,x)+\kappa(\|u(t)\|_{U})+\epsilon.
\end{equation}
\end{enumerate}
Then, $\Sigma$ is ISpS.
\end{thm}
\begin{proof}
Pick any $\lambda\in(0,\varrho)$ and define $\omega(s)=\frac{1}{\varrho-\lambda}\kappa(s).$ For all $(t,x)\in \mathbb{R}_+\times X$ and all $u\in PC(\mathbb{R}_+,U),$ we have the following implication
$$V(t,x)\geq\omega(\|u(t)\|_{U})\Longrightarrow \dot{V}(t,x)\leq -\lambda V(t,x)+\epsilon.$$
Then, $V$ satisfies the inequality (\ref{rama1bbb}) with $\nu(t)=-\lambda$ and $\psi(t)=\epsilon.$ Thus, all conditions of Theorem \ref{dfffdddf1} are satisfied. Therefore, $\Sigma$ is ISpS.
\end{proof}
\begin{rem}
For finite-dimensional systems existence of an ISpS-Lyapunov function (\ref{YJJasmine}) implies the ISpS of systems. 
\end{rem}
\section{\sectiontitle{ISpS of time-varying infinite-dimensional systems with integrable inputs}}

In this section, we extend the ISpS Lyapunov methodology to make it suitable for the analysis of ISpS w.r.t. inputs from $L_p$-spaces. We show that the existence of the so-called $L_p$-ISpS Lyapunov function implies the $L_p$-ISpS of a system. 
\begin{dfn}
Let $\Sigma=(X, \mathcal{U},\phi)$ be a control system. We say that $\phi$ depends continuously on inputs with respect to $\mathcal{U}$-norm, if for all $x\in X,$ all $u\in\mathcal{U},$ all $T\in[t_0,d)$ and all $\varepsilon>0,$ there is $\delta>0,$ such that for all $\tilde{u}\in\mathcal{U}:$ $\|u-\tilde{u}\|_{\mathcal{U}}<\delta$ it holds that $d(x,\tilde{u})\geq T$ and
$$\|\phi(t,x,u)-\phi(t,x,\tilde{u})\|_{X}<\varepsilon\quad \forall t\in[t_0,T].$$
\end{dfn}
\begin{thm}\label{yasmp}
Let $p\in [1,\infty).$ Consider a control system $\Sigma=(X, \mathcal{U},\phi)$ with $\mathcal{U}=L_{p}(\mathbb{R}_+,U).$ Suppose that $\Sigma$ has BIC property and that $\phi$ is continuous on inputs with respect to $L_p-$norm. Assume that there exist a continuous Lyapunov function $V:\mathbb{R}_+\times X \to \mathbb{R}_+,$ functions $\alpha_1,\alpha_2\in \mathcal{K}_{\infty},$ constants $\ell >0,$ $c\geq 0$ and functions $\nu,\psi\in C(\mathbb{R}_+, \mathbb{R}),$ such that, for all $(t,x)\in \mathbb{R}_+\times X$ and all $u\in \mathcal{U},$
\begin{enumerate}
\item[$(i)$]
\begin{equation}\label{ramkklj}
\alpha_1(\|x\|_{X}) \leq V(t,x) \leq  \alpha_2 (\|x\|_{X})+c,
\end{equation}
\item[$(ii)$] the Lie derivative of $V$ along the trajectories of system $\Sigma$ satisfies
$$\dot{V}(t,x) \leq \nu(t)V(t,x)+\psi(t)+\ell \|u(t)\|_{U}.$$
\item[$(iii)$] Condition $(iii)$ in Theorem \ref{dfffdddf1} holds.
\end{enumerate}
Then, $\Sigma$ is ISpS w.r.t. inputs in the space $L_p(\mathbb{R}_{+},U).$
\end{thm}
\begin{proof}
Pick any initial time $t_0\in\mathbb{R}_+,$ any initial condition $x_0 \in X$ and any admissible control $u \in C(\mathbb{R}_+,U)\subset L_{p}(\mathbb{R}_+,U).$ As $\Sigma$ is a control system, there is $d \in (t_0,\infty),$ such that $A_{\phi} \cap ([t_0,\infty)\times \{(t_0,x_0, u)\})= [t_0,d)\times \{(t_0,x_0, u)\}.$ Define $y(t)=V(t,x(t)),\; t \in [t_0,d).$ We have
\begin{equation}\label{R2nYASMINEbbbbb}
D^{+}y(t) \leq \nu(t)y(t)+\psi(t)+\ell \|u(t)\|_{U}.
\end{equation}
By Condition $(iii)$ and Lemma \ref{prop3vvv} one gets from (\ref{R2nYASMINEbbbbb}) that
\begin{equation}\label{R2nYASMINE}
y(t)\leq y(t_0) e^{-\eta(t-t_0)} e^ \xi + \rho + \ell e^ \xi \int_{t_0}^{t}e^{-\eta(t-s)} \|u(s)\|_{U}ds.
\end{equation}
The following two cases are considered respectively.
\begin{enumerate}
\item[\bf{Case 1.}] $p=1.$ By using $e^{-\eta(t-s)} \leq 1,\quad t\geq s,$ we have from (\ref{ramkklj}) and (\ref{R2nYASMINE}),
$$\alpha_1(\| x(t)\|_{X}) \leq \alpha_2(\|x_0\|_{X}) e^{-\eta(t-t_0)}e^ \xi +\rho +c+ \ell e^ \xi \int_{t_0}^{t} \|u (s)\|_{U }ds.$$
Then, by using $\alpha_1^{-1}(a+b+c)\leq\alpha_1^{-1}(2a)+\alpha_1^{-1}(4b)+\alpha_1^{-1}(4c),\; \alpha_1^{-1}\in\mathcal{K}_{\infty},\; a,b,c\in\mathbb{R}_{+},$ we get
$$\|x(t)\|_{X}\leq \beta (\|x_0\|_{X},t-t_0)+ \gamma(\|u\|_{L_1(\mathbb{R}_1,U)})+r,$$
where $\beta(\upsilon,s)=\alpha_{1}^{-1} (2\alpha_{2}(\upsilon)e^{-\eta s}e^ {\xi}),$ $\gamma(s)=\alpha_{1}^{-1}(4 \ell e^ \xi s)$ and $r=\alpha_{1}^{-1} (4\rho+4c).$ The above estimate and the BIC property show that the solution $x(t)$ exists in $[t_0,\infty).$ Therefore, the system $\Sigma $ is ISpS w.r.t. inputs in the space $ L_1(\mathbb{R}_+,U).$
\item[\bf{Case 2.}] $p>1.$ Pick any $q\geq 1,$ so that $\frac{1}{p}+\frac{1}{q}=1.$ We continue the estimates (\ref{R2nYASMINE}), using the H\"older Inequality
\begin{eqnarray*}
y(t)&\leq & y(t_0) e^{-\eta(t-t_0)} e^ \xi + \rho+ \ell e^{\xi} \bigg{(} \int_{t_0}^{t} e^{-\frac{\eta}{2}q(t-s)}\bigg{)}^{\frac{1}{q}} \bigg{(}\int_{t_0}^{t} e^{-\frac{p\eta}{2}}\|u(s)\|_{U}^{p} \bigg{)}^{\frac{1}{p}}\\
&\le& y(t_0) e^{-\eta(t-t_0)}e^\xi+\rho+\ell e^\xi \left(\frac{2}{q\eta}\right)^\frac{1}{q} \bigg{(}\int_{t_0}^{t}\|u(s)\|_{U}^p ds\bigg{)}^\frac{1}{p} \\
&\le& y(t_0) e^{-\eta(t-t_0)}e^\xi+\rho+\ell e^\xi \left(\frac{2}{q\eta}\right)^\frac{1}{q}\|u\|_{L_p(\mathbb{R}_+,U)}.
\end{eqnarray*}
Hence, it can be deduced that
$$\|x(t)\|_{X}\leq \beta (\|x_0\|_{X},t-t_0)+ \gamma(\|u\|_{L_p(\mathbb{R}_+,U)})+r,$$
where $\beta(\upsilon,s)=\alpha_{1}^{-1} (2\alpha_{2}(\upsilon)e^{-\eta s}e^ {\xi}),$ $\gamma(s)=\alpha_{1}^{-1}(4\ell e^ \xi \left(\frac{2}{q\eta}\right)^\frac{1}{q} s)$ and $r=\alpha_{1}^{-1}(4\rho+4c).$
\end{enumerate}
In particular, the solution $x(\cdot)$ stays bounded on $[t_0,d),$ if $d < \infty,$ then by the BIC property it can be prolonged to a larger interval which contradiction to the maximality of $d.$ Hence, $d=\infty $ and the solution exists on $[t_0,\infty)$ for all $x_0 \in X $ and all $u \in C(\mathbb{R}_{+},U).$ Now, pick any $x_0\in X$ and any input $u\in L_{p}(\mathbb{R}_+,U).$ As $\Sigma$ is a control system, the corresponding solution $x(\cdot)$ exists on a certain maximal interval $[t_0,d).$ It is well know that $C([0,d),U)$ is dense in $\mathcal{U}=L_{p}([0,d), U).$ Since $\phi$ is continuous on inputs with respect to $L_p-$norm and the BIC property, we have $\Sigma$ is ISpS w.r.t. inputs in the space $L_p(\mathbb{R}_{+},U).$
\end{proof}
\begin{rem}
Function $V$ is called $L_p$-ISpS Lyapunov function for the system $\Sigma.$ On the other hand, if we take $c=0$ and $\psi(t)=0$ in Theorem \ref{yasmp}, we can derive that the system $\Sigma$ is ISS w.r.t. inputs in the space $L_p(\mathbb{R}_{+},U).$
\end{rem}

\section{\sectiontitle{ISpS of semi-linear evolution equations}}

The following results show how to construct a suitable ISpS-Lyapunov function for the time-varying infinite-dimensional system (\ref{R1}).
\begin{thm}\label{th0b1}
Let Assumption (H) hold and let $A(t)$ be a uniformly bounded linear operator for all $t\geq 0$ in the sense that $\sup _{t\geq 0 } \|A(t)\|<\infty.$ Assume that the evolution operator generated by $\{A(t)\}_{t\geq0}$ is uniformly asymptotically stable and there exists a continuously differentiable, uniformly bounded, coercive positive operator-valued function $P:\mathbb{R}_{+}\to L(X)$, satisfying (\ref{R2}). If there are integrable continuous functions $a,b:\mathbb{R}_{+} \to \mathbb{R}_{+},$ a constant $\varsigma\geq0,$ and a function $\zeta\in \mathcal{K},$ such that for all $x\in X,$ all $u\in \mathcal{U},$ all $t_0\geq 0$ and all $t\geq t_0,$
\begin{equation}\label{R2nnkkkk}
\|\Psi(t,x,u)\|_{X}\leq a(t)\|x\|_{X}+b(t)\zeta(\|u(t)\|_{U})+\varsigma,
\end{equation}
then the system \eqref{R1} is ISpS and its ISpS-Lyapunov function can be constructed as
$$V(t,x)=\langle P(t)x,x\rangle.$$
\end{thm}
\begin{proof}
Since $P$ is uniformly bounded and coercive, for some $\mu_1,\mu_2> 0,$ il holds that
$$\mu_1\|x\|_{X}^{2}\leq V(t,x)\leq \mu_2\|x\|_{X}^{2}\quad \forall x \in X,$$
where $\mu_2=\sup_{t\in \mathbb{R}_{+}}\|P(t)\|$ and property (\ref{R2}) is verified.\\
 The  Lie derivative of $V$ along the trajectories of system \eqref{R1} is given as follows:
\begin{align*}
\dot{V}(t,x)&=\langle\dot{P}(t)x,x\rangle+\langle P(t)\dot{x},x\rangle +\langle P(t)x,\dot{x}\rangle\\
&=\langle \dot{P}(t)x,x\rangle+ \langle P(t)[A(t)x+\Psi(t,x,u)],x\rangle +\langle P(t)x, A(t)x+\Psi(t,x,u)\rangle.
\end{align*}
It follows by (\ref{R2nnkkkk}) with the help of Cauchy-Schwartz inequality that
\begin{align*}
\dot{V}(t,x)&\leq-\|x\|_{X}^{2}+2a(t)\|P(t)\|\|x\|_{X}^{2}+2\|P(t)\|b(t) \|x\|_{X} \zeta(\|u(t)\|_{U})+2\|P(t)\| \|x\|_{X}\varsigma\\
&\leq-\|x\|_{X}^{2}+\|P(t)\|(2a(t)+b(t))\|x\|_{X}^{2}+\|P(t)\|b(t) (\zeta(\|u(t)\|_{U}))^{2}+2\|P(t)\| \|x\|_{X}\varsigma.
\end{align*}
For any $\epsilon>0,$ by applying Young's inequality
$$2\|P(t)\|\|x\|_{X}\varsigma \leq \epsilon\|x\|^{2}+ \frac{\mu_2^{2}\varsigma^{2}}{\epsilon},$$ we have
$$\dot{V}(t,x)\leq \left(-\frac{1}{\mu_2 } +\frac{\mu_2}{\mu_1}(2a(t)+b(t))+\frac{\epsilon}{\mu_1}\right)V(t,x)
+ \mu_2 b(t) (\zeta(\|u(t)\|_{U}))^{2} +\frac{\mu_2^{2}\varsigma^{2}}{\epsilon}.$$
If $V(t,x)\geq(\zeta(\|u\|_{U}))^{2},$ then
$$\dot{V}(t,x)\leq \left(-\frac{1}{\mu_2} +\frac{2\mu_2}{\mu_1}(a(t)+b(t))+\frac{\epsilon}{\mu_1}\right)V(t,x)+ \frac{\mu_2^{2}\varsigma^{2}}{\epsilon}.$$
Let $\epsilon=\frac{\mu_1}{2\mu_2}\cdot$ Thus,
$$\dot{V}(t,x)\leq \nu(t)V(t,x)+\psi(t),$$
where
$\nu(t)=-\frac{1}{2\mu_2} +\frac{2\mu_2}{\mu_1}(a(t)+b(t))$ and $\psi(t)=\frac{2\mu_2^{3}\varsigma^{2}}{\mu_1}.$
Now, all conditions of Theorem \ref{dfffdddf1} are satisfied, and then the system (\ref{R1}) is ISpS.
\end{proof}
The following results show how to construct a suitable ISpS-Lyapunov function for the time-varying semilinear evolution equation (\ref{psiunb}).
\begin{thm}\label{th0b1} Let Assumption $(\mathcal{H}_{1})$ hold and let $A(t)$ be a bounded linear operator for all $t\geq 0$ in the sense that $\sup _{t\geq 0 } \|A(t)\|<\infty.$
Assume that the system (\ref{unb}) is ISS and there exists a continuously differentiable, bounded, coercive positive operator-valued function $P:\mathbb{R}_{+}\to L(X)$, satisfying (\ref{R2}).
If there exists a continuous function $\omega:\mathbb{R}_{+} \to \mathbb{R}_{+},$ such that for all $x\in X,$ all $u\in \mathcal{U},$ all $t_0\geq 0$ and all $t\geq t_0,$
$$\|\Psi(t,x,u)\|_{X}\leq \omega(t),$$
with $\displaystyle\int _{0} ^{\infty}\omega^{2}(s)<\infty,$ then the system (\ref{psiunb}) is ISpS and its ISpS-Lyapunov function can be constructed as
\begin{equation}
V(t,x)=\langle P(t)x,x\rangle.
\label{R2lynv}
\end{equation}
Moreover, if $\displaystyle \lim_{t\to \infty}\omega^{2}(t)=0,$ then the system (\ref{psiunb}) has the CpUAG property.
\end{thm}
\begin{proof}
Consider $V:\mathbb{R}_{+}\times X \to \mathbb{R}_{+}$ as defined in (\ref{R2lynv}). Since $P$ is bounded and coercive, for some $\mu_1,\mu_2> 0$, it holds that
$$\mu_1\|x\|_{X}^{2}\leq V(t,x)\leq \mu_2\|x\|_{X}^{2}\quad \forall x \in X \quad \forall t\geq0,$$
where $\mu_2=\displaystyle\sup_{t\in \mathbb{R}_{+}}\|P(t)\|.$ Thus, (\ref{llMMMM}) holds. The Lie derivative of $V$ along the trajectories of system \eqref{psiunb} is given as follows:
\begin{align*}
\dot{V}_{u}(t,x)&=\langle\dot{P}(t)x,x\rangle+\langle P(t)\dot{x},x\rangle +\langle P(t)x,\dot{x}\rangle\\
&=\langle \dot{P}(t)x,x(t)\rangle+ \langle P(t)[A(t)x+B(t)u+\Psi(t,x,u)],x\rangle\\ &+\langle P(t)x, A(t)x+B(t)u+\Psi(t,x,u)\rangle.
\end{align*}
Since $\langle P(t)x,A(t)x\rangle=\langle A(t)^*P(t)x,x\rangle,$ by applying the Lyapunov equation (\ref{R2}) in the form
$$\langle \dot{P}(t)x,x\rangle+\langle P(t)A(t)x,x\rangle+\langle A(t)^*P(t)x,x\rangle=-\|x\|^{2}_{X},$$
and using Cauchy-Schwartz inequality, we obtain
\begin{align*}
\dot{V}_{u}(t,x)&\leq -\|x\|_{X}^{2}+2 \mu_2\|x\|_{X}(\|B(t)\|\|u\|_{U}+\Psi(t,x,u))\\
&\leq -\|x\|_{X}^{2}+2 \mu_2\|x\|_{X}(\upsilon\|u\|_{U}+\omega(t)),
\end{align*}
where $\upsilon=\displaystyle\sup_{t\in \mathbb{R}_{+}}\|B(t)\|.$ For any $\eta>0,$ by applying Young's inequality
$$2\|x\|_{X}(\upsilon\|u\|_{U}+\omega(t))\le\eta\|x\|_{X}^{2}+\frac{1}{\eta}(\upsilon\|u\|_{U}+\omega(t))^{2},$$
we have
$$\dot{V}_{u}(t,x)\leq -(1-\eta  \mu_2 )\|x\|_{X}^{2}+\frac{2 \mu_2 \upsilon^{2} }{\eta}\|u\|_{U}^{2}+\frac{2\mu_2}{\eta}\omega^{2}(t).$$
Let, $\eta=\displaystyle\frac{1}{2\mu_2}\cdot$ It yields,
$$\dot{V}_{u}(t,x)\leq -\frac{1}{2}\|x\|_{X}^{2}+4\mu_2^{2} \upsilon^{2}\|u\|_{U}^{2}+4\mu_2^{2}\omega^{2}(t).$$
Now, take $a\in(0,\frac{1}{2\mu_2})$ and let
$$V(t,x)\geq \frac{8\mu_2^{3}\upsilon^{2}\|u\|_{U}^{2}}{1-2\mu_2a}\cdot$$
Thus, we have
$$\dot{V}_{u}(t,x)\leq -a V(t,x)+4\mu_2^{2}\omega^{2}(t).$$
From Theorem \ref{dfffdddf1}, one can conclude that the system (\ref{psiunb}) is ISpS. If $\displaystyle \lim_{t\to \infty} \omega^{2}(t)=0,$ then the system (\ref{psiunb}) has the CpUAG property.
\end{proof}

\section{\sectiontitle{ISpS of interconnected systems}}
In this section, we have shown, how an ISpS-Lyapunov function can be explicitly constructed for a given interconnection if the small-gain condition is satisfied. \\ 
Let $X_i$ and $U_i$ be Banach spaces and as the input space, we take $\mathcal{U}_i=PC(\mathbb{R}_{+}, U_i),$ for $i=1,...,n.$
We consider a time-varying semi-linear interconnected system:
\begin{equation}
\left\lbrace
\begin{array}{l}
\dot{x_i}(t)=A_i(t)x_i(t)+\Psi_i(t,x_1(t),...,x_n(t),u_i(t)),\quad i=1,...,n,\\
\\
x_i(t)\in X_i,\quad  u_i(t)\in U_i ,\quad t\geq t_0\geq 0,
\end{array}\right.
\label{GGGSnnS}
\end{equation}
where $X_i$ is a state space and $U_i$ is an input space of the $i$-th subsystem, $t_0\geq 0$ is the initial time, $\{A_i(t):D(A_i(t)) \subset X_i \rightarrow X_i\}_{t \geq t_0\geq 0},$ is a family of linear operators depending on time where the domain $D(A_i(t))$ of the operator $A_i(t)$ is assumed to be dense in $X_i$ for all $t\geq 0$ and generates a strongly continuous evolution family $\{W_i(t,s)\}_{t\geq s\geq 0},$ $i=1,...,n,$ the function $\Psi_i:\mathbb{R}_{+}\times X_i\times U_i \to X_i$ is a nonlinear operator and we assume that the solution of (\ref{GGGSnnS}) exists, is unique, and is forward complete.\\
We denote the state space of the system (\ref{GGGSnnS}) by $X=X_1\times ...\times X_n$ which is Banach with the norm $\|\cdot\|_X = \|\cdot\|_{X_{1}}+...+\|\cdot\|_{X_{n}}.$\\
To present a small-gain criterion for the interconnected system (\ref{GGGSnnS}), let us consider the following assumption:
\begin{description}
\item[$(\mathcal{H}_{4})$] For each $i=1,...,n,$ there exists a continuously differentiable $V_i :\mathbb{R}_{+}\times X_i \to \mathbb{R}_{+}$ satisfying the following properties:
\begin{itemize}
  \item There are functions $\psi_{i_{1}},\psi_{i_{2}} \in \mathcal{K}_{\infty}$ so that for all $(t,x_i)\in \mathbb{R}_{+}\times X_i,$
$$ \psi_{i_{1}}(\|x_i\|_{X_i})\leq V_i(t,x_i)\leq \psi_{i_{2}}(\|x_i\|_{X_i}).$$
  \item  There are functions $\delta_{i}\in\mathcal{K}_{\infty},$ $\sigma_i,\xi_i\in \mathcal{K}_{\infty}\cup\{0\}$ and continuous integrable functions $\ell_{i}:\mathbb{R}_+ \to \mathbb{R}_+$ so that the following hold: for all $t_0\geq 0,$ all $t\geq t_0,$ all $x_i\in  X_i,$ all $x_{n+1-i}\in  X_{n+1-i}$ and all $u_i\in  \mathcal{U}_i=PC(\mathbb{R}_{+},U_i)$ one has
$$\dot{V}_{iu}(t,x_{i})\leq -\delta_{i}(V_{i}(t,x_i))+\sigma_i(V_{n+1-i}(t,x_{n+1-i}))+\ell_{i}(t)+\xi_i(\|u_i(t)\|_{U_i}).$$
\end{itemize}
\end{description}
\noindent Now, we can present our small-gain theorem which provides an ISpS-Lyapunov function for the interconnected system (\ref{GGGSnnS}).
\begin{thm}\label{ml}
Let Assumption $(\mathcal{H}_{4})$ hold.
If there exist a parameter $0<\zeta<1$ and a function $\delta\in \mathcal{K}_{\infty},$ such that the small-gain condition
\begin{equation}
\frac{1}{\zeta}\sigma\circ\delta^{-1}(s)<s \quad \forall s>0,
\label{RKKKKn}
\end{equation}
holds with $\sigma(s)=\displaystyle \max_{i=1,...,n}(\sigma_i(s))$ and $\delta\left(\displaystyle\sum_{ i=1}^{n}s_i\right) \leq \displaystyle\sum_{ i=1}^{n}\delta_{i}(s_{i}),$ then the interconnected system (\ref{GGGSnnS}) is ISpS. \\ Moreover, if $\displaystyle \lim_{t\to \infty} \ell_i(t)=0$ for $i=1,...,n,$ then the system (\ref{GGGSnnS}) satisfies the CpUAG property.
\end{thm}
\begin{proof} Consider the following candidate for an ISpS-Lyapunov function
$$V(t,x)=\sum_{ i=1}^{n}V_i(t,x_i).$$
The Lie derivative of $V$ along the trajectories of system (\ref{GGGSnnS}) is given by:
$$\dot{V}_{u}(t,x)\leq\sum_{i=1}^{n}(-\delta_{i}(V_{i}(t,x_i))+\sigma_i(V_{n+1-i}(t,x_{n+1-i}))+\ell_{i}(t)+\xi_i(\|u_i(t)\|_{U_i})).$$
Then, there exists $\delta\in \mathcal{K}_{\infty},$ such that
$$\dot{V}_{u}(t,x)\leq -\delta(V(t,x))+\sigma(V(t,x))+\ell(t)+\sum_{ i=1}^{n}\xi_i(\|u_i(t)\|_{U_i}),$$
where $\delta\left(\displaystyle\sum_{ i=1}^{n}s_i\right) \leq \displaystyle\sum_{ i=1}^{n}\delta_{i}(s_{i}),$ $\sigma(s)=\displaystyle \max_{i=1,...,n}(\sigma_i(s))$ and $\ell(s)=\displaystyle\sum_{ i=1}^{n}\ell_{i}(s).$\\
By using the fact $u=(u_1^{T},...,u_n^{T})^{T},$ we have
$$\dot{V}_{u}(t,x)\leq -\delta(V(t,x))+\sigma(V(t,x))+\ell(t)+\xi(\|u(t)\|_{U}),$$
where $\xi(s)=\displaystyle\sum_{ i=1}^{n}\xi_i(s).$
Thus, there exists $0<\zeta<1,$ such that
$$\dot{V}_{u}(t,x)\leq-(1-\zeta)\delta(V(t,x))-\zeta\delta(V(t,x))+\sigma(V(t,x))+\ell(t)+\xi(\|u(t)\|_{U}).$$
Which implies that
$$\dot{V}_{u}(t,x)\leq-(1-\zeta)\delta(V(t,x))+\ell(t),\; \hbox{if} \; [\sigma-\zeta\delta]V(t,x)+\xi(\|u(t)\|_{U})\leq0.$$
Therefore, using (\ref{RKKKKn}), there exists $\kappa\in \mathcal{K}_{\infty},$ defined by
$$\kappa(s)=(\zeta\delta)^{-1}\circ\left(I-\frac{1}{\zeta}\sigma\circ\delta^{-1}\right)^{-1}\circ\xi(s),$$
such that, we have the following implication
$$V(t,x)\geq \kappa(\|u(t)\|_{U})\Longrightarrow \dot{V}_{u}(t,x)\leq-\delta(V(t,x))+\ell(t).$$
Now, all conditions of Theorem \ref{dfffdddf1} are satisfied, and hence the interconnected system (\ref{GGGSnnS}) is ISpS. If $\displaystyle \lim_{t\to
\infty} \ell_i(t)=0$ for $i=1,...,n,$ then the system (\ref{GGGSnnS}) has the CpUAG property.
\end{proof}
\begin{rem} In a recent paper \cite{Jiang1996}, a nonlinear small-gain theorem was proved for autonomous interconnected systems of ODEs. The authors \cite{Jiang1994,Jiang1996} proved that the interconnection of two ISpS systems remains an
ISpS system where they established a Lyapunov-type nonlinear small-gain
theorem based on the construction of an appropriate Lyapunov function. To overcome this challenge, we have considered a class of time-varying $n$ systems where we have imposed some restrictions through the
Assumption $(\mathcal{H}_{4})$ and the small-gain condition (\ref{RKKKKn}). A suitable Lyapunov function can therefore be designed to ensure the ISpS of the interconnected
system (\ref{GGGSnnS}). However, if $\sigma_i=0$ for all $i=1,...,n,$ in Assumption $(\mathcal{H}_{4}),$ then the $n$ subsystems are ISpS.
\end{rem}
\section{\sectiontitle{Examples}}
\begin{exm}
  Consider the system
\begin{equation}\label{1EX}
\dot{x}(\zeta,t)=\left(\frac{1}{1+t+x^{2}(\zeta,t)}-t|\cos(t)|\right)x(\zeta,t)+\frac{2t|\cos(t)|u(\zeta,t)}{\pi(1+x^{2}(\zeta,t))}+\frac{|\cos(t)|(\tan \zeta)^{\frac{1}{2}}}{1+x^{2}(\zeta,t)},
\end{equation}
where $t\geq0$ and $\zeta\in \left(0,\frac{\pi}{2}\right).$ The functions $x(\zeta,t)$ and $u(\zeta,t)$ are scalar-valued. Let $X=L_{2}(0,\frac{\pi}{2})$ and $U=C([0,\frac{\pi}{2}]).$ Choose the Lyapunov function as
$$V(t,x)=\int_{0}^{\frac{\pi}{2}}x^{2}(\zeta,t)d\zeta+e^{-t}.$$
The function $V$ satisfies condition (\ref{rama2}) with $b_1=b_2=1,$ $m=2$ and $c=1.$ The derivative of $V$ in $t$ along the solutions $x(\cdot,t)$ of system (\ref{1EX}) satisfies
\begin{equation}\label{2EXyasmine}
\dot{V}(t,x)\leq 2\left(\frac{1}{1+t^{2}}-t |\cos(t)|\right)V(t,x)+(2te^{-t}+M)|\cos(t)|+t|\cos(t)|\|u(\cdot,t)\|_{C([0,\frac{\pi}{2}])},
\end{equation}
where $M=\int _{0}^{\frac{\pi}{2}}(\tan \zeta)^{\frac{1}{2}}d\zeta<\infty.$\\
If $\|u(\cdot,t)\|_{C([0,\frac{\pi}{2}])}\leq V(t,x),$ we can obtain from $(\ref{2EXyasmine})$ that
$$\dot{V}(t,x)\leq \nu(t) V(t,x)+\psi(t),$$
where $\nu(t)= \frac{2}{1+t^{2}}-t |\cos(t)|$ and $\psi(t)=(2te^{-t}+M)|\cos(t)|.$\\
Let $\Theta(t,\tau)=\exp(\displaystyle\int _{\tau}^{t}\nu(s)ds).$ \\
Since $\displaystyle\int _{t_0} ^{t} 2\tau e^{-\tau}\Theta(t,\tau)d\tau \to 0$ as $t \to +\infty,$ then
$$ \displaystyle \int _{t_0}^{t} \nu(\tau) d\tau\leq -\eta(t-t_0)+\xi,$$
see Appendix 3 in \cite{zhou2017stability}
and
$$\int _{t_0}^{t}\vert \psi(\tau)\vert\Theta(t,\tau)d\tau\leq\rho,$$
where $\eta=\frac{4\pi}{3},$ $\xi=2\ln(1+\frac{3\pi}{2})+2$ and $\rho=\displaystyle\sup _{t\geq t_0}(\int _{t_0} ^{t} 2\tau e^{-\tau}\Theta(t,\tau)d\tau)+\frac{Me^{-\eta}}{\xi}.$
Thus, all hypotheses of Corollary \ref{MARIAM} are satisfied, and hence system (\ref{1EX}) is exponentially ISpS (eISpS).
\end{exm}
\subsection{\subsectiontitle{Semi-linear reaction-diffusion equation}}
\begin{exm}
Consider the following semi-linear reaction-diffusion equation:
\begin{equation}\label{2EX}
\left\lbrace
\begin{array}{l}
\displaystyle\frac{\partial x(l,t)}{\partial t}=\frac{\partial^{2} x(l,t)}{\partial ^{2} l}+\frac{\Phi(t)(u(l,t)+e^{-t})}{(1+x^{2}(l,t))(1+(\pi-l)x^{2}(l,t))},\\
\\
\displaystyle x(0,t)=0=x(\pi,t),\; x(l,0)=x_{0}(l),
\end{array}\right.
\end{equation}
on the region $(l,t)\in(0,\pi)\times(0,\infty)$ of the valued functions $x(l,t)$ and $u(l,t)$ and $\Phi:\mathbb{R}_+\to \mathbb{R}$ is a continuous and  integrable function. We denote $X=L_{2}(0,\pi),$ $U=C([0,\pi])$ and $\mathcal{U}=C(\mathbb{R}_+,U).$ Thus, system (\ref{2EX}) can be rewritten as
\begin{equation}\label{Rexm1nb}
\left\lbrace
\begin{array}{l}
\dot{x}(t)=Ax(t)+\chi(t,x(t),u(t)),\quad  t> 0,\\
\\
x(0)=x_{0}(l),
\end{array}\right.
\end{equation}
where $A=\displaystyle\frac{\partial^{2}}{\partial^{2} l},$ is defined on $D(A)=H_{0}^{1}(0,\pi)\cap H^{2}(0,\pi)$ and
$$\displaystyle \chi(t,x,u)=\frac{\Phi(t)(u(l,t)+e^{-t})}{(1+x^{2}(l,t))(1+(\pi-l)x^{2}(l,t))}.$$ Operator $A$ generates an analytic semigroup $(S(t))_{t\geq0}$ on $X.$
For every $(x_{0}(l),u)\in X\times \mathcal{U},$ the system (\ref{2EX}) has a unique mild solution $x(\cdot)\in C([0,\infty),X)$ and 
for $(x_{0}(l),u)\in D(A)\times \mathcal{U}$ this mild solution is even a classical solution which satisfies the integral equation
$$x(t)=S(t)x_{0}(l)+\int_{0}^{t}S(t-s)\chi(t,x(s),u(s))ds.$$
Consider the following ISpS-Lyapunov function candidate:
$$V(t,x)=\|x\|_{L_{2}(0,\pi)}^{2}+e^{-t}=\int_{0}^{\pi }x^{2}(l,t)dl+e^{-t}$$ which satisfies the inequalities $(\ref{rama1})$ and $(\ref{rama3})$ given in Theorems $\ref{dfffdddf1}$ and $\ref{dffff1},$ with $\alpha_1(s)=\alpha_2(s)=s^{2}$ and $c=1.$ Then,
$$\dot{V}(t,x)=2\int_{0}^{\pi}x(l,t)\left(\frac{\partial^{2} x(l,t)}{\partial ^{2} l}+\frac{\Phi(t)(u(l,t)+e^{-t})}{(1+x^{2}(l,t))(1+(\pi-l)x^{2}(l,t))}\right)dl\\-e^{-t}.$$
Since $1+(\pi-l)x^{2}(l,t)\geq 1,$ we get
$$\dot{V}(t,x)\leq -2\int _{0}^{\pi}\left(\frac{\partial x(l,t)}{\partial l}\right)^{2}dl+\pi|\Phi(t)|\|u(\cdot,t)\|_{C([0,\pi])}\\
+ \pi |\Phi(t)|e^{-t}.$$
Applying Friedrich's inequality (\cite{Pecaric}) in the first term, we continue estimates:
\begin{equation}\label{LLLMM}
\dot{V}(t,x)\leq -2V(t,x)+(2+\pi|\Phi(t)|)e^{-t}+\pi|\Phi(t)|\|u(\cdot,t)\|_{C([0,\pi])}.
\end{equation}
Let $V(t,x)\geq \|u(\cdot,t)\|_{C([0,\pi])}.$
Then,
$$\dot{V}(t,x)\leq \nu(t)V(t,x)+\psi(t),$$
where $\nu(t)=-2+\pi|\Phi(t)|$ and $\psi(t)=(2+\pi|\Phi(t)|)e^{-t}.$ \\
Let $\Theta(t,\tau)=\exp(\displaystyle\int _{\tau}^{t}\nu(s)ds).$ Therefore, for all $t\geq t_0,$
\begin{equation}\label{kjjkjk}
\displaystyle \int _{t_0}^{t} \nu(\tau) d\tau\leq -\eta(t-t_0)+\xi,
\end{equation}
and
\begin{equation}\label{kjjkbbbjk}
\int _{t_0}^{t}\vert \psi(\tau)\vert\Theta(t,\tau)d\tau\leq\rho,
\end{equation}
where $\eta=-2,$ $\xi=\pi\lambda$ and $\rho=e^{\pi\lambda}(2+\pi\lambda)$ with $\lambda=\displaystyle\int _{0 }^{+\infty} |\Phi(s)|ds.$
Thus, all conditions of Theorem $\ref{dfffdddf1}$ are satisfied, and hence the system (\ref{2EX}) is ISpS.\\
Now, if the function $\Phi$ is bounded and integrable on $\mathbb{R}_+.$ We can also obtain from $(\ref{LLLMM})$ that
$$\dot{V}(t,x)\leq \nu(t)V(t,x)+\psi(t)+\pi\displaystyle\sup_{s\in\mathbb{R}_+}|\Phi(s)| \|u(\cdot,t)\|_{C([0,\pi])},$$
where $\nu(t)=-2,$ $\psi(t)=(2+\pi|\Phi(t)|)e^{-t}$ and $\theta( \|u(\cdot,t)\|_{C([0,\pi])})=\pi\displaystyle\sup_{s\in\mathbb{R}_+}|\Phi(s) | \|u(\cdot,t)\|_{C([0,\pi])}.$
By repeating the same calculation as before, we obtain that (\ref{kjjkjk}) and (\ref{kjjkbbbjk}) are verified with
$\eta=-2,$ $\xi=0$ and $\rho=2+\pi\lambda ,$ which, by Theorem $\ref{dffff1}$ implies that the system (\ref{2EX}) is iISpS.
\end{exm}

\subsection{\subsectiontitle{Time-varying interconnected reaction-diffusion PDEs}}
\begin{exm}
We consider the following time-varying system of interconnected nonlinear reaction-diffusion PDEs:
\begin{equation}
\left\lbrace
\begin{array}{l}
\frac{\partial x_1(\zeta,t)}{\partial t}=c_1\frac{\partial^{2} x_1(\zeta,t)}{\partial^{2} \zeta}-\upsilon  (t)x_1(\zeta,t)+e^{-t}x_2(\zeta,t)+\frac{t^{2}+1}{1+(1+t^{2})^{2}x_1^{2}(\zeta,t)}+u_1^{2}(\zeta,t)e^{-x_1^{2}(\zeta,t)},\\
x_1(0,t)=0=x_1(L,t),\\
\frac{\partial x_2(\zeta,t)}{\partial t}=c_2\frac{\partial^{2} x_2(\zeta,t)}{\partial^{2} \zeta}-\upsilon (t)x_2(\zeta,t)+\sin(u_2(\zeta,t))x_1(\zeta,t)+\frac{e^{-2t}+u_2(\zeta,t)}{1+x_2^{2}(\zeta,t)},\\
x_2(0,t)=0=x_2(L,t),
\end{array}\right.
\label{1EX}
\end{equation}
on the region $(\zeta,t)\in(0,L)\times(0,\infty)$ where $L\in (0,\infty),$ $\upsilon:\mathbb{R}_+\to \mathbb{R}_+$ is a continuously differentiable function and $c_i>0, i=1,2,$ are the diffusion coefficients.\\
The state spaces of subsystems we choose as $X_1=L^{2}(0,L)$ for $x_1(\cdot,t)$ and $X_2=L^{2}(0,L)$ for $x_2(\cdot,t).$ We take the spaces of input values for the subsystems as $U_1=U_2=C([0,L]).$ The state of the whole system (\ref{1EX}) is denoted by $X=X_1\times X_2.$ For both subsystems, take the set of input functions as $\mathcal{U}_i=C(\mathbb{R}_{+}, U_i).$\\
A family $\{A_i(t):D(A_i(t)) \subset X_i \rightarrow X_i\}_{t \geq t_0\geq 0},$ $i=1,2$ of linear operator can be represented as follows:
$\displaystyle A_i(t)f=A_if-\upsilon(t)f,$ with the domain $D(A_i(t))=D(A_i),\; for \ all \ t \geq t_0\geq 0$ where $A_i$ is defined by $$c_i\displaystyle\frac{\partial^{2}}{\partial^{2} \zeta},$$ with the domain $$D(A_i)=\{f\in H^{2}(0,L)/f(0)=f(L)=0\}.$$
Here $H^{2}(0,L)$ denotes the Sobolev space of functions $f\in L^{2}(0,L),$ which have weak derivatives of order $\leq 2,$ all of
which belong to $L^{2}(0,L).$ Setting
$$\Psi_1(t,x_1,x_2,u_1)=e^{-t}x_2(\zeta,t)+\frac{t^{2}+1}{1+(1+t^{2})^{2}x_1^{2}(\zeta,t)}+u_1^{2}(\zeta,t)e^{-x_1^{2}(\zeta,t)}$$
and
$$\Psi_2(t,x_1,x_2,u_2)=\sin(u_2(\zeta,t))x_1(\zeta,t)+\frac{e^{-2t}+u_2(\zeta,t)}{1+x_2^{2}(\zeta,t)}\cdot$$
We can verify that $\{A_i(t)\}_{ t\geq 0}$ generates a strongly continuous evolution operator $\{W_i(t,s)\}_{t\geq s\geq 0}$ of the form:
$$W_i(t,s)=e^{(t-s)A_i}\exp\left(\int_{s}^{t} -\upsilon(r)dr\right) \quad \forall t\geq s\geq 0.$$
We choose $V_i, i=1,2$ defined by
$$V_1(t,x_1)=\int_{0} ^{L }x_1^{2}(\zeta,t)d\zeta,\quad  V_2(t,x_2)=\int_{0} ^{L }x_2^{2}(\zeta,t)d\zeta.$$
Consider the derivative of $V_1:$
\begin{eqnarray*}
\dot{V}_1(t,x_1)&=&\int _{0}^{L}2x_1(\zeta,t)(c_1\frac{\partial^{2} x_1(\zeta,t)}{\partial^{2} \zeta}-\upsilon (t)x_1(\zeta,t)+e^{-t}x_2(\zeta,t)\\&+&\frac{t^{2}+1}{1+(1+t^{2})^{2}x_1^{2}(\zeta,t)}+u_1^{2}(\zeta,t)e^{-x_1^{2}(\zeta,t)})d\zeta\\&\leq&
-2\int _{0} ^{L}\left(\frac{\partial x_1(\zeta,t)}{\partial \zeta}\right)^{2}+2\int _{0} ^{L}x_1(\zeta,t) x_2(\zeta,t)d\zeta+\frac{L}{1+t^{2}}+2L\|u_1(\cdot,t)\|^{2}_{C([0,L])}\cdot
\end{eqnarray*}
By the Friedrich's inequality \cite{Pecaric} and Young's inequality, we obtain for any $\varepsilon>0:$
$$\dot{V}_1(t,x_1)\le -2c_1\left(\frac{\pi}{L}\right)^{2}V_1(t,x_1)+\varepsilon V_1(t,x_1) +\frac{1}{\varepsilon}V_2(t,x_2)+\frac{L}{1+t^{2}}+2L\|u_1(\cdot,t)\|^{2}_{C([0,L])}\cdot$$
Choosing $\varepsilon=c_1\left(\frac{\pi}{L}\right)^{2},$ we get
\begin{equation}
\dot{V}_1(t,x_1)\le-c_1\left(\frac{\pi}{L}\right)^{2}V_1(t,x_1)+\frac{1}{c_1}\left(\frac{L}{\pi}\right)^{2}V_2(t,x_2)+\frac{L}{1+t^{2}}+2L\|u_1(\cdot,t)\|^{2}_{C([0,L])}\cdot
\label{R2nnnnn}
\end{equation}
Consider the derivative of $V_2:$
\begin{eqnarray*}
\dot{V}_2(t,x_2)&=& \int _{0}^{L}2x_2(\zeta,t)(c_2\frac{\partial^{2} x_2(\zeta,t)}{\partial^{2} \zeta}-\upsilon (t)x_2(\zeta,t)+\sin(u_2(\zeta,t))x_1(\zeta,t)\\
&+&\frac{e^{-2t}+u_2(\zeta,t)}{1+x_2^{2}(\zeta,t)})d\zeta.
\end{eqnarray*}
Using Friedrich's inequality \cite{Pecaric} in the first term and Young's inequality, we get for any $\xi>0:$
\begin{equation}
\dot{V}_2(t,x_2)\le -2c_2\left(\frac{\pi}{L}\right)^{2}V_2(t,x_2)+\xi V_2(t,x_2)+\frac{1}{\xi}V_1(t,x_1)+Le^{-2t}+L\|u_2(\cdot,t)\|_{C([0,L])}.
\label{R2nnn}
\end{equation}
Choosing $\xi=c_2\left(\frac{\pi}{L}\right)^{2},$ it follows from (\ref{R2nnn}) that
\begin{equation}
\dot{V}_2(t,x_2)\le-c_2\left(\frac{\pi}{L}\right)^{2}V_2(t,x_2)+\frac{1}{c_2}\left(\frac{L}{\pi}\right)^{2}V_1(t,x_1)+Le^{-2t}+L\|u_2(\cdot,t)\|_{C([0,L])}.
\label{R2nnbvccnnn}
\end{equation}
Due to (\ref{R2nnnnn}) and (\ref{R2nnbvccnnn}), we have $(\mathcal{H}_{4})$ for
$$\psi_{1_{1}}=\psi_{1_{2}}=\psi_{2_{1}}= \psi_{2_{2}}:s \mapsto s^{2},$$ $$\delta_{1}(s)=c_1\left(\frac{\pi}{L}\right)^{2}s,\;\sigma_1(s)=\frac{1}{c_1}\left(\frac{L}{\pi}\right)^{2}s,\; \ell_{1}(t)=\frac{L}{1+t^{2}},\;\xi_1(s)=2Ls^{2},$$
$$\delta_{2}(s)=c_2\left(\frac{\pi}{L}\right)^{2}s,\;\sigma_2(s)=\frac{1}{c_2}\left(\frac{L}{\pi}\right)^{2}s,\;\ell_{2}(t)=L e^{-2t},\;\xi_2(s)=Ls.$$
Then, Theorem \ref{ml} establishes ISpS and the CpUAG property for the system (\ref{1EX}) when
$$\frac{\max\left(\frac{1}{c_1}\left(\frac{L}{\pi}\right)^{2},\frac{1}{c_2}\left(\frac{L}{\pi}\right)^{2}\right)}{\min\left(c_1\left(\frac{\pi}{L}\right)^{2},c_2\left(\frac{\pi}{L}\right)^{2}\right)}<\zeta,\quad 0<\zeta<1.$$
\end{exm}

\chapter{Conclusion}
\label{ch:chapter4}
\markboth{Conclusion}{Conclusion }
\textbf{Control systems}. Definition \ref{csyol} of a control system is frequently used within the system-theoretic community at least since the 1960s. Not all important systems are covered by this definition. In particular, the input space $C(\mathbb{R}_{+}, U)$ of continuous $U-$valued functions does not satisfy the axiom of concatenation. This should not be a big restriction since already piecewise continuous and $L^{p}$ inputs, which are used in control theory much more frequently than continuous ones, satisfy the axiom of concatenation. In a similar spirit, even more general system classes can be considered, containing output maps, time-variant dynamics, the possibility for a solution to jump at certain time instants, systems that fail to satisfy the cocycle property, etc..

We have investigated the input-to-state practical stability problem of time-varying nonlinear infinite-dimensional systems, integral input-to-state practical stability, and global uniform practical stability problems of time-varying infinite-dimensional systems by employing indefinite Lyapunov functions in Banach spaces. Based on Lyapunov theory and a nonlinear inequality, we have proposed an input-to-state practical stability theorem and established criteria for ISpS in terms of the uniform practical asymptotic gain property. Moreover, the global asymptotic practical stability of time-varying nonlinear infinite-dimensional systems with zero input is studied. These results have been applied to study ISpS for time-varying nonlinear evolution equations in Hilbert spaces, demonstrating the explicit construction of ISpS-Lyapunov functions for interconnected systems satisfying the small-gain condition.\\

In addition, we have explored the input-to-state stability of time-varying evolution equations in terms of coercive and non-coercive ISS Lyapunov functions with piecewise-right continuous inputs. Indeed, we have seen in this paper some classes of time-varying linear systems with
bounded input operators in the construction of Lyapunov functions. In addition, new methods for the construction of non-coercive LISS/iISS Lyapunov functions for a certain class of time-varying semi-linear evolution equations with unbounded linear operator $\{A(t)\}_{t\geq t_{0}}$ have been proposed as well. The application to some specific time-varying partial differential equations shows the practical significance of the theoretical results. An interesting problem for future research is the development of Lyapunov methods for ISS and iISS of the time-varying nonlinear parabolic PDEs with boundary disturbances.\\

A promising avenue for future research involves developing Lyapunov methods for ISS and iISS of time-varying nonlinear parabolic PDEs under boundary disturbances. \\
\appendix
\chapter*{Appendix A: Comparison Functions and Elementary Inequalities}
\label{appendix:A}
\markboth{Appendix A: Comparison Functions}{Appendix A: Elementary Inequalities}
\addcontentsline{toc}{chapter}{Appendix A: Comparison Functions and Elementary Inequalities}
In this appendix, I have included some results from Andrii Mironchenko's habilitation \cite{mironchenko2023input} to provide context and save you from referring to multiple sources. These results are not my own but are necessary for understanding my thesis. They are placed here to make the thesis easier to read and comprehend.
\subsection*{A.1 Comparison functions}
\renewcommand{\thelem}{A.\arabic{lem}}
\renewcommand{\thecor}{A.\arabic{cor}}
\renewcommand{\thedfn}{A.\arabic{dfn}}

In this part, we explore the properties of the following classes of comparison functions
\begin{itemize}
\item $\mathcal{P}=\{\gamma:\mathbb{R}_+\to \mathbb{R}_+:\gamma \;\;\hbox{is continuous,} \gamma(0)=0\; \hbox{and}\;\gamma(r)>0 \;\hbox{for} \;r>0 \}.$
\item $\mathcal{K}=\{\gamma\in \mathcal{P}:\gamma \;\;\hbox{is strictly increasing}\}.$
\item $\mathcal{K}_{\infty}=\{\gamma\in \mathcal{K}:\gamma \;\;\hbox{is unbounded}\}.$
\item $\mathcal{L}=\{\gamma:\mathbb{R}_+\to \mathbb{R}_+:\gamma \;\;\hbox{is continuous and strictly decreasing with} \displaystyle\lim_{t\to \infty}\gamma(t)=0\}.$
\item $\mathcal{KL}=\{\beta\in C(\mathbb{R}_+\times\mathbb{R}_+,\mathbb{R}_+):\beta(\cdot,t)\in \mathcal{K}, \forall t\geq0, \beta(r,\cdot)\in \mathcal{L},\forall r> 0\}.$
\end{itemize}
Functions in class $\mathcal{P}$ are also referred to as positive definite functions. Functions in class $\mathcal{K}$ are often called class $\mathcal{K}$ functions or simply $\mathcal{K}$-functions. We apply this same naming convention to other classes of comparison functions.

\par Additionally, we use $\circ$ to denote the composition of functions, meaning $f \circ g(\cdot) := f(g(\cdot)).$
\renewcommand{\theprop}{A.\arabic{prop}}
\renewcommand{\theequation}{A.\arabic{equation}} 

\setcounter{equation}{0} 
\begin{prop}(Elementary properties of $\mathcal{K}-$ and $\mathcal{L}-$functions) The following
properties hold:
\begin{enumerate}
\item[$(i)$] For all $f,g\in \mathcal{K}$ it follows that $f \circ g\in\mathcal{K}.$
\item[$(ii)$]  For any $f\in\mathcal{K}_{\infty}$ there exists $f^{-1}$ that also belongs to $\mathcal{K}_{\infty}.$
\item[$(iii)$]  For any $f\in\mathcal{K}, g\in \mathcal{L}$ it holds that $f\circ g\in \mathcal{L}$ and $g\circ f\in \mathcal{L}.$
\item[$(v)$]  For any $\gamma_{1}, \gamma_{2}\in\mathcal{K}$ it follows that $\gamma_{+} : s \mapsto max\{\gamma_{1}(s), \gamma_{2}(s)\}$ and $\gamma_{-} : s \mapsto
min\{\gamma_{1}(s), \gamma_{2}(s)\}$ also belong to class $\mathcal{K}.$
\item[$(vi)$]  For any $\beta_{1}, \beta_{2} \in \mathcal{KL}$ it follows that $\beta_{+} : (s,t) \mapsto max\{\beta_{1}(s, t), \beta_{2}(s, t)\}$ and
$\beta_{-} : (s, t) \mapsto min\{\beta_{1}(s, t), \beta_{2}(s, t)\}$ also belong to class $\mathcal{KL}.$
\end{enumerate}
\end{prop}
\begin{proof}
  See \cite[Proposition A.1]{Mir23}.
\end{proof}
\renewcommand{\therem}{A.\arabic{rem}}

\begin{lem}\label{lemmss1}(see \cite{sontag1989smooth}):
If $\alpha \in \mathcal{K}_{\infty},$ then for all $a,b\geq 0,$ we have
$$\alpha(a+b)\leq \alpha(2a)+\alpha(2b).$$
\end{lem}
\begin{lem}\cite{jameson2014some}\label{lema5}
Let $a,\ b\ge0$ and $p\ge1,$ then
\begin{itemize}
  \item $(a+b)^p\le 2^{p-1}(a^p+b^p).$
  \item $(a+b)^{1\over p}\leq a^{1\over p}+b^{1\over p}.$
\end{itemize}
\end{lem}

A fundamental yet useful result regarding $\mathcal{K}-$functions is the following:

\begin{prop}(Weak triangle inequality) For any $\gamma\in\mathcal{K}$ and any $\sigma \in \mathcal{K}_{\infty}$, the following holds for all $a, b \geq 0$:
\renewcommand{\theequation}{A.\arabic{equation}} 
\begin{equation}
    \gamma(a + b) \leq \max\{\gamma(a + \sigma(a)), \gamma(b + \sigma^{-1}(b))\}. \label{eq:A1}
\end{equation}
\end{prop}

\begin{proof}
  Choose any $\sigma \in \mathcal{K}_{\infty}.$ Then, either $b \leq \sigma(a)$ or $a \leq \sigma^{-1}(b)$ holds. This directly implies (\ref{eq:A1}).
\end{proof}

In particular, by setting $\sigma:= id$ (identity operator on $\mathbb{R}_+$) in (\ref{eq:A1}), we obtain a straightforward inequality for any $\gamma\in\mathcal{K}$:
$$\gamma(a + b) \leq max \{\gamma (2a), \gamma(2b)\}.$$

\subsection*{A.1.1 Dini derivatives}
It is well-known that continuous functions are differentiable almost everywhere. However, there are instances when it is necessary to "differentiate" only continuous functions. In such cases, we utilize Dini derivatives. For a continuous function $y: \mathbb{R} \rightarrow \mathbb{R},$ the right upper Dini derivative and the right lower Dini derivative are defined as follows:
$$ D^{+}y(t) := \limsup_{h\rightarrow+0} \frac{y(t + h) - y(t)}{h}\quad \ and \ \ D_{+}y(t) := \liminf_{h\rightarrow+0}\frac{y(t + h) -y(t)}{h},$$
respectively. \\

We begin with some basic properties of Dini derivatives:
\begin{lem}
  The following properties of Dini derivatives hold:
  \begin{enumerate}
\item[$(i)$] For all $f , g\in C(\mathbb{R})$ it holds that $ D^{+}(f + g) \leq D^{+}(f ) + D^{+}(g).$
\item[$(ii)$] For any $f \in C(\mathbb{R})$ it holds that $D^{+}(-f ) = -D_{+}(f ).$
\end{enumerate}
\end{lem}

We present next a lemma on the derivatives of monotone functions
\begin{lem}
  Let $b : \mathbb{R}_{+}\rightarrow \mathbb{R}_{+}$ be a nonincreasing function. Then for each $t \in\mathbb{R}_{+}$ it holds that
  \renewcommand{\theequation}{A.\arabic{equation}} 
  \begin{equation}
b(t) \geq D^{+}\int_{0}^{t} b(s)ds \geq D_{+} \int_{0}^{t}b(s)ds \geq \displaystyle\lim_{h\rightarrow+0}b(t + h).   \label{eq:27}
\end{equation}
\end{lem}
\begin{proof}
  See \cite[Lemma A.32]{Mir23}.
\end{proof}

\subsection*{A.1.2 Comparison principles}
Comparison principles establish relationships between the properties of differential equations or inequalities and auxiliary equations or inequalities. \\
We start with the following comparison principle.
\begin{lem}\label{lemmm1}
For any $\theta\in \mathcal{P}$ there exists $\beta\in \mathcal{KL}$ so
that for any continuous function $\omega:\mathbb{R}_+\to \mathbb{R}_+$
satisfying, for a given $t_0\geq 0$ and all $t\geq t_0,$ the differential inequality 
\begin{equation}\label{bbb2}
D^{+}\omega(t)\leq-\theta(\omega(t)),\quad \omega(t_0)=\omega_0,
\end{equation}
it holds that
$$\omega(t)\leq \beta(\omega_0,t-t_0)\quad 
\forall t\geq t_0.$$
\end{lem}

\begin{proof}
From condition (\ref{bbb2}), we obtain
\begin{equation}\label{mlbbb2}
\frac{D^{+}\omega(t)}{\theta(\omega(t))}\leq-1,
\end{equation}
for all $t_0\geq0$ and all $t\geq t_0$ so that $\omega(t)\neq0.$ \\ 
Define
\begin{equation}\label{sxsxxb2}
\pi(s):=\int_{ 1 }^{s }\frac{d\tau}{\omega(\tau)}\cdot
\end{equation}
Since $\omega$ is continuous, then $\pi$  is continuously differentiable and its derivative is positive on its domain of definition.
Then, by \cite[Lemma 2]{MiI16}, we have
$$\frac{D^{+}\omega(t)}{\theta(\omega(t))}=D^{+}\pi(\omega(t)),$$
and (\ref{mlbbb2}) can be rewritten as
\begin{equation}\label{nvcxsxb2w}
D^{+}\pi(\omega(t))\leq -1.
\end{equation}
Since $\theta\in \mathcal{P}$ and $\omega$ is strictly increasing, then the function $\pi(\omega(t))$ is a decreasing function.
Then, via \cite[Proposition 5]{MiI16}, we further have
\begin{equation}\label{HENI1}
\pi(\omega(t))-\pi(\omega_0)\leq\int_{ t_0 }^{t }D^{+}\pi(\omega(\tau))d\tau\leq -(t-t_0).
\end{equation}
Since $\pi$ is strictly increasing, $\pi$ is invertible and  its inverse $\pi^{-1}$ is a strictly increasing function on $[-\infty,\infty).$
Hence, we get from (\ref{HENI1}) that
\begin{equation}\label{HENI2}
\omega(t)\leq \pi^{-1}(\pi(\omega_0)-(t-t_0)),
\end{equation}
for all $t_0\geq0$ and all $t\in [t_0,T),$ where $T=\min\{t\in[t_0,\infty): \omega(t)=0\}.$
Define a function a function $\beta:\mathbb{R}_{+}\times\mathbb{R}_{+}\to \mathbb{R}_{+}$ by

$$\beta(r,s):=\left\lbrace
\begin{array}{l}
\pi^{-1}(\pi(r)-s) \; if\; \;r>0, \\
0\;  \;if \; \;r=0.
\end{array}\right.$$
As a consequence of the strict increasing property of $\pi$ and $\lim_{ s \to 0^{+}}\pi(s)=-\infty,$ we obtain $\beta$ is a class $\mathcal{KL}$ function.
\end{proof}

The following is an easy consequence of Lemma \ref{lemmm1}.
\begin{cor}\label{cor1} 
For any $\theta\in\mathcal{P},$ there is $\beta\in\mathcal{KL}$ such that for any $t_0\ge 0$, any $\tau\in(0,\infty],$ any $\omega\in C([t_0,t_0+\tau), \mathbb{R}_+)$, and any
$\mu\in PC ([t_0,t_0+\tau),\mathbb{R}_+)$ if
\begin{equation}\label{b..20rh}
D^{+}\omega(t)\leq-\theta(\omega(t))+\mu(t)
\end{equation}
holds for all $t\in [t_0,t_0+\tau),$ then for all $t\in[t_0,t_0+\tau)$ it holds that
$$\omega(t)\leq\beta(\omega(t_0),t-t_0)+2\int _{t_0}^{t} \mu(s)ds.$$
\end{cor}
\begin{proof}
  
We assume that (\ref{b..20rh}) holds for all $t \in [t_0, t_0 + \tau)$.

By \cite[Lemma B.1.18. p.236]{mironchenko2023input}, there is a globally Lipschitz continuous function $\varphi \in \mathcal{P}$ such that $\varphi \leq \theta$ pointwise. Consider the following initial value problem:
\begin{equation}
\dot{w}(t) = -\varphi(w(t)) + \mu(t), \quad w(t_0) = \omega(t_0)
\end{equation}

As $\varphi$ is globally Lipschitz and $\mu$ is piecewise continuous, there is a unique piecewise continuously differentiable solution $w$ of this initial value problem.

Take some $\hat{t} \in (t_0, t_0 + \tau)$, such that $w$ is continuously differentiable on $[t_0, \hat{t})$. We have:
\begin{equation}
D^{+}\omega(t)-\dot{w}(t) \leq -(\varphi(\omega(t)) - \varphi(w(t))), \quad for \ all \ t \in [t_0, \hat{t}).
\end{equation}

Assume that there is some $\tau \in [t_0, \hat{t})$ such that $\omega(\tau) > w(\tau)$. Pick the maximal in $[t_0, \hat{t})$ interval $[s, s+\varepsilon)$ containing $\tau$ such that $\omega(t) > w(t)$ for all $t \in (s, s+\varepsilon)$. By the maximality of the interval $(s, s+\varepsilon)$, continuity of $\omega,$ $w$, as $\omega(t_0) = w(t_0)$, we have $\omega(s) = w(s)$.

As $\varphi$ is globally Lipschitz, there is some $L > 0$ such that for all $t \in (s, s+\varepsilon)$ we have:
\begin{equation}
D^{+}(\omega(t) - w(t)) \leq |\varphi(\omega(t)) - \varphi(w(t))| \leq L|\omega(t) - w(t)|= L(\omega(t) - w(t)).
\end{equation}

Define $z(t) := \omega(t) - w(t)$ on $(s, s+\varepsilon)$. By \cite[Proposition B.3.3. p.236]{mironchenko2023input}, this implies that $z(t) \leq z(s)e^{L(t-s)}$, which contradicts $z(t) > 0$ for $t \in (s, s+\varepsilon)$.

This shows that $0 \leq \omega(t) \leq w(t)$ for all $t \in [t_0, t_0 + \tau)$.

Now, define $v_1(t) := \int_{t_0}^t \mu(s) ds$ and $w_1(t) := w(t) - v_1(t)$. Then, for almost all $t \in [t_0, t_0 + \tau)$, we have:
\begin{equation}
w_1'(t) = w'(t) - \mu(t) = -\varphi(w(t)) \leq -\varphi(w_1(t) + v_1(t))
\end{equation}

In \cite{mironchenko2023input}, Proposition B.4.2 ensures that there is a function $\beta \in \mathcal{KL}$, depending on $\varphi$ solely, such that for all $t \in [t_0, t_0 + \tau)$:
\begin{equation}
w_1(t) \leq \max\{\beta(w_1(t_0), t-t_0), \|v_1\|_\infty\}
\end{equation}

As $w_1(t_0) = w(t_0) = \omega(t_0)$, we conclude that for all $t \in [t_0, t_0 + \tau)$:
\begin{equation}
\omega(t) \leq w(t) = w_1(t) + v_1(t) \leq \beta(\omega(t_0), t-t_0) + 2\int_{t_0}^t \mu(s) ds
\end{equation}

This completes the proof.
\end{proof}

A fundamental result is:
\begin{prop}\label{a.3333}
Let $f$ be Lipschitz continuous on bounded subsets of $\mathbb{R}^{n}$ and continuous on $\mathbb{R}^{n}\times \mathbb{R}^{m}.$ Additionally, let $u \in \mathcal{U}, T > 0$ and
$x, y\in AC([0, T),\mathbb{R})$ such that $x(0) \leq y(0)$ and
\renewcommand{\theequation}{A.\arabic{equation}} 

\begin{equation}
\dot{x}(t)-f (x(t), u(t)) \leq \dot{y}(t)-f (y(t), u(t)),\quad for \ a.e. \ t \in [0, T).  \label{eq:28}
\end{equation}
Then $x(t)\leq y(t)$ for every $t\in [0, T).$
\end{prop}
\begin{proof}
  See \cite[Proposition A.33]{Mir23}.
\end{proof}
A variation of the above principle for the Dini derivatives is as follows:
\begin{prop}
  Let $f$ is Lipschitz continuous on bounded subsets of $\mathbb{R}^{n},$ continuous on $\mathbb{R}^{n}\times \mathbb{R}^{m}$ and $u\in C(\mathbb{R}_{+},\mathbb{R}^{m}), T > 0,$ and let
$x, y \in C([0, T),\mathbb{R})$ be such that $x(0) \leq y(0)$ and
\renewcommand{\theequation}{A.\arabic{equation}} 
\begin{equation}
D^{+}x(t)-f (x(t), u(t))\leq D_{+}y(t)- f (y(t), u(t)), \quad for \ all \ t \in [0, T).   \label{eq:30}
\end{equation}
Then $x(t) \leq y(t)$ for every $t\in [0, T).$
\end{prop}
\begin{proof}
  See \cite[Proposition A.34]{Mir23}.
\end{proof}

\section*{A.2 Elementary inequalities}
In this section, we collect several inequalities, which are instrumental in the stability analysis of 
infinite and finite-dimensional control systems and for verifying the important properties of function 
spaces.
\subsection*{A.2.1 Inequalities in $\mathbb{R}^{n}$}
\begin{dfn}
A map $ f: X \rightarrow \mathbb{R} $ is called convex provided that 
   $$f(ax+ (1-a)y)\leq af(x) + (1-a)f(y),$$
  for all $x, y\in X,$ and all $a\in[0,1].$
\end{dfn}

\begin{prop}(Young’s inequality). Let $p,q>1$ be such that $\frac{1}{p}+\frac{1}{q}= 1.$ Then for all $a, 
b\geq0$ holds 
\renewcommand{\theequation}{B.\arabic{equation}} 
\begin{equation}\label{D2}
ab\leq \frac{a^{p}}{p}+\frac{b^{q}}{q}.
\end{equation}
\end{prop}

\begin{proof}
The function $\ln(x)$ is concave since $\ln'(x) = \frac{1}{x}$ and $\ln''(x) = -\frac{1}{x^2}$. Thus, we have
\[
\ln(ab) = \ln\left((a^p)^{\frac{1}{p}} (b^q)^{\frac{1}{q}}\right) = \frac{1}{p} \ln(a^p) + \frac{1}{q} \ln(b^q).
\]
By the concavity of $\ln(x)$, it follows that
\[
\frac{1}{p} \ln(a^p) + \frac{1}{q} \ln(b^q) \leq \ln\left(\frac{a^p}{p} + \frac{b^q}{q}\right).
\]
Exponentiating both sides, we get
\[
ab \leq \frac{a^p}{p} + \frac{b^q}{q}.
\]
\end{proof}

\begin{prop}\label{generalyi}(Young’s inequality, general form). For all $a, b\geq0$ and all $\omega, p 
>0$ it holds
\renewcommand{\theequation}{B.\arabic{equation}} 
\begin{equation}\label{D3}
ab\leq\frac{\omega}{p}a^{p}+\frac{1}{\omega^{\frac{1}{p-1}}}\frac{p-1}{p}b^{\frac{p}{p-1}}.
\end{equation}
\end{prop}

\begin{proof}
Write $ab=\varepsilon^{\frac{1}{p}}a\cdot \frac{b}{\varepsilon^{\frac{1}{p}}},$ and apply for 
$q=\frac{p}{p-1}$ Young’s inequality (\ref{D2}).
\end{proof}
Often a special case of Proposition \ref{generalyi} (also called Cauchy’s inequality with $\varepsilon$) 
is used:
\begin{prop} (Young’s inequality, special case). For all $a, b\in\mathbb{R}$ and all $\varepsilon >0$ it 
holds
\renewcommand{\theequation}{B.\arabic{equation}} 
\begin{equation}\label{d4}
  ab\leq\frac{\varepsilon}{2}a^{2}+\frac{1}{2\varepsilon}b^{2}.
\end{equation}
\end{prop}
We adopt the next formulation from \cite[Lemma 2.7, p.42]{holden2013spectral}):

\begin{lem} (Generalized Gronwall’s inequality). Let $\psi,\alpha,\beta\in C([0,T],\mathbb{R})$ and let 
the
inequality holds
\renewcommand{\theequation}{B.\arabic{equation}} 
\begin{equation}\label{D5}
 \psi(t)\leq\alpha(t) +\int_{0}^{t}\beta(s)\psi(s)ds,\quad t\in [0,T],
\end{equation}
with $\alpha(t)\in \mathbb{R}$ and $\beta(t)\geq0,t\in [0,T].$ Then:
\begin{itemize}
\item[$(i)$] $\psi(t)\leq\alpha(t) +\int_{0}^{t} \alpha(s)\beta(s)\exp\{\int_{s}^{t}\beta(r)dr\}ds,t\in 
    [0,T].$
\item[$(ii)$] If further $\alpha(s)\leq\alpha(t)$ for $s\leq t,$ then it holds that
 $$\psi(t)\leq\alpha(t) \exp \{(\int_{0}^{t}\beta(s)ds)\}, t\in[0,T].$$
\end{itemize}
\end{lem}
\subsection*{A.2.2 Basic integral inequalities}
We collect here several basic inequalities in $L^{p}$ spaces.
\begin{prop} (Hölder's inequality). Let $G$ be an open, bounded region in $\mathbb{R}^{n}$. As some 
$p,q\in [1,\infty]$ and $\frac{1}{p}+\frac{1}{q}= 1.$ Then if $x\in L^{p}(G), \nu\in L^{q}(G),$ then
\renewcommand{\theequation}{B.\arabic{equation}} 
\begin{equation}\label{d6}
  \int_{G} |x(z)\nu(z)|dz\leq \|x\|_{L^{p}(G)}\|\nu\|_{L^{q}(G)}.
\end{equation}
\end{prop}
\begin{proof}
Assume first that $\|x\|_{L^{p}(G)}= 1$ and $\|\nu\|_{L^{q}(G)}=1.$ By Young’s inequality (\ref{D2}) for 
$p, q$ as in the statement of the proposition, we have for a.e. $z\in G$
\renewcommand{\theequation}{B.\arabic{equation}} 
\begin{equation}\label{d7}
  |x(z)\nu(z)| \leq\frac{1}{p}|x(z)|^{p}+\frac{1}{q}|\nu(z)|^{q}.
\end{equation}
The function $x\nu$ is measurable and (\ref{d7}) shows that its absolute value is upper-bounded by an 
integrable function, thus $x\nu\in L^{1}(G)$ and
 $$ \int_{G} 
 |x(z)\nu(z)|dz\leq\frac{1}{p}\int_{G}|x(z)|^{p}dz+\frac{1}{q}\int_{G}|\nu(z)|^{q}dz=\frac{1}{p}+\frac{1}{q}= 
 1 =\|x\|_{L^{p}(G)} \|\nu\|_{L^{q}(G)}.$$
For general $x\in L^{p}(G), \nu \in L^{q}(G),$ the result follows by homogeneity.
\end{proof}

\begin{prop}(Cauchy-Bunyakovsky-Schwarz inequality in $L^{2}(G)$).
For all $x\in L^{2}(G),\nu\in L^{2}(G)$ it holds that 
\renewcommand{\theequation}{B.\arabic{equation}} 
\begin{equation}\label{d9}
\int_{G}|x(z)\nu(z)|dz\leq\|x\|_{L^{2}(G)}\|\nu\|_{L^{2}(G)}.
\end{equation}
\end{prop}
\begin{proof}
The result is a special case of Hölder's inequality (\ref{d6}) for $p=q=2,$ as well as a special case of the general Cauchy-Bunyakovsky-Schwarz’ inequality in Hilbert spaces.
\end{proof}
The following result constitutes the triangle inequality in the $L^{p}(G)-$space
\begin{prop} (Minkowski inequality). Assume that $p, q \in [1,\infty]$ and $\frac{1}{p}+\frac{1}{q}=1.$
Then if $x,\nu\in L^{p}(G),$ then
$$\|x+\nu\|_{L^{p}(G)}\leq \|x\|_{L^{p}(G)}+\|\nu\|_{L^{p}(G)}.$$
\end{prop}
\begin{proof}
See \cite[p. 707]{evans2022partial}.
\end{proof}
An important integral inequality for convex functions is
\begin{prop}\label{d.2.10} (Jensen’s inequality). Let $G$ be open and bounded, and
$f :\mathbb{R}^{m}\rightarrow\mathbb{R}$ be convex. Let also $x :G \rightarrow\mathbb{R}^{m}$ be 
summable. Then
\renewcommand{\theequation}{B.\arabic{equation}} 
\begin{equation}\label{d.11}
  \frac{1}{|G|}\int_{G}f(x(z))dz\geq f(\frac{1}{|G|}\int_{G}x(z)dz).
\end{equation}
\end{prop}
\begin{proof}
See \cite[p. 705]{evans2022partial}.
\end{proof}
In the scalar case, Jensen’s inequality simplifies to
\begin{prop} (Scalar Jensen’s inequality).
For any convex $f :\mathbb{R} \rightarrow \mathbb{R}$ and any summable
$x :[0,L]\rightarrow\mathbb{R},$ it holds that
\renewcommand{\theequation}{B.\arabic{equation}} 
\begin{equation}\label{d.12}
 \int_{0}^{L}f(x(z))dz\geq Lf(\frac{1}{L}\int_{0}^{L}x(z)dz).
\end{equation}
Similarly, for any concave $f :\mathbb{R} \rightarrow \mathbb{R}$ and any summable $x 
:[0,L]\rightarrow\mathbb{R},$ Jensen’s
inequality holds with the reverse inequality sign:
\renewcommand{\theequation}{B.\arabic{equation}} 
\begin{equation}\label{d.13}
 \int_{0}^{L}f(x(z))dz\leq Lf(\frac{1}{L}\int_{0}^{L}x(z)dz).
\end{equation}
\end{prop}
\begin{proof}
For convex $f,$ the claim follows from the general Jensen’s inequality. To see the second part, note that 
$f$ is concave if $-f$ is convex. Applying Jensen’s inequality to $-f,$ we obtain the result for concave 
$f.$ 
\end{proof}

\section*{A.2.3 Integral inequalities for 1-dimensional domains}
This section presents some several integral inequalities on $(0, L)$.\\ 
We begin with the well-known Friedrichs' inequality (see, e.g., \cite[Appendix A]{murray2003spatial}).

\begin{prop}[Friedrichs' inequality]\label{d.6.1}
For all $x \in H^1_0(0, L)$ it holds that
\renewcommand{\theequation}{B.\arabic{equation}}  
\begin{equation}\label{d29}
\frac{L^{2}}{\pi^{2}} \int_{0}^{L} \left( \frac{d x(z)}{dz} \right)^2 dz \geq \int_{0}^{L} x^2(z) dz.
\end{equation}
\end{prop}

\begin{proof}
See \cite[Proposition D.6.1]{mironchenko2023input}.
\end{proof}

\begin{prop} For every $x \in H^1(0, L)$ with either $x(0) = 0$ or $x(L) = 0$, it holds that
\renewcommand{\theequation}{B.\arabic{equation}}  
\begin{equation}\label{d30}
\frac{4L^2}{\pi^2} \int_{0}^{L} \left(\frac{dx(z)}{dz}\right)^2 dz \geq \int_{0}^{L} x^2(z) dz.
\end{equation}
\end{prop}
\begin{proof}
  See \cite[Section 7.6, Nr 256]{hardy1952inequalities}.
\end{proof}
A notable inequality relating the $L^\infty$-norm to the $H^1$-norm, is:
\begin{prop}(Agmon's inequality). For all $x \in H^1(0,L)$, it holds that
\renewcommand{\theequation}{B.\arabic{equation}}  
\begin{equation}\label{d31}
\|x\|^2_{L^\infty(0,L)} \leq |x(0)|^2 + 2\|x\|_{L^2(0,L)}\big{\|}\frac{dx}{dz}\big{\|}_{L^2(0,L)}.
\end{equation}
\end{prop}
\begin{proof} 
See \cite[Lemma 2.4., p. 20]{krstic2008boundary}.
\end{proof}
We finish this Appendix by presenting a useful inequality from \cite[Lemma 1]{ZhZ18}:
\begin{prop} Let $x \in C^1([a,b],\mathbb{R})$. Then for each $c \in [a, b]$, it holds that
\renewcommand{\theequation}{B.\arabic{equation}}  
\begin{equation}\label{d32}
x^2(c) \leq \frac{2}{b - a} \|x\|^2_{L^2(0,1)} + (b - a) \|x_{z}\|^2_{L^2(0,1)}.
\end{equation}
\end{prop}

\section*{A.2.4 Gronwall's inequality}

To prove the various theorems presented in this paper, we need the following lemmas:
\begin{lem} \label{propbbvcc}
Let's consider the following nonlinear inequality:
\begin{equation}
\dot{y}\leq -\delta(y)+\ell(t) \quad \forall t\geq t_0,\quad y \in C^{1}([t_0,\infty),\mathbb{R}_+),
\label{h2ni2}
\end{equation}
where $\ell: \mathbb{R}_+ \rightarrow \mathbb{R}_+$ be a continuous integrable function and $\delta $ be of class $\mathcal{K}_{\infty}.$ Then, there exists a $\mathcal{KL} $ function $\sigma : \mathbb{R}_+\times \mathbb{R}_+\rightarrow \mathbb{R}_+,$
such that
\begin{equation}
y(t)\leq \sigma(y(t_0),t-t_0)+2\int_{0}^{\infty}\ell(s)ds\quad \forall t\geq t_0.
\label{h24n2}
\end{equation}
Moreover, if $\displaystyle\lim_{t\to \infty} \ell(t)=0,$ there exists a function $\beta \in \mathcal{KL}$ depending only on $\delta$ and $\ell,$ such that
\begin{equation}
y(t)\leq\beta(y(t_0)+ \int_{0}^{\infty}\ell(s)ds,t-t_0) \quad \forall t\geq t_0.
\label{hjj102}
\end{equation}
In particular, if $\ell$ is a uniformly continuous function, then (\ref{hjj102}) holds.
\end{lem}
\begin{proof}
It follows from \cite{ASW00} that when $\ell$ is a continuous integrable function, there exists a $\mathcal{KL} $ function $\sigma$ such that (\ref{h24n2}) holds.

Now, we consider the case when $\ell$ is a nonzero function and $\displaystyle\lim_{t \to \infty} \ell(t) = 0$.

Let
\[ H(t_0) = \{ y : [t_0,\infty) \to \mathbb{R} \text{ continuously differentiable function satisfying (\ref{h2ni2})} \} \]
Note that
\begin{equation}
y(t) \leq y(t_0) + M \quad \forall t \geq t_0, \, y \in H(t_0), 
\label{254mnmnj}
\end{equation}
where
\[ M = \int_0^\infty \ell(s) \, ds \]
and this shows that $y$ is bounded $\forall y \in H(t_0)$. We start by showing the property:
\begin{description}
\item[$(\mathcal{P})$] For all $(\epsilon, R, t_0) \in \mathbb{R}^{*}_{+} \times \mathbb{R}_{+} \times \mathbb{R}_{+}, \, y \in H(t_0)$, there exists $\widetilde{T}(\epsilon, R+M)$ such that if $y(t_0) \leq R$ then $0 \leq y(t) \leq \epsilon, \, \forall t \geq t_0 + \widetilde{T}(\epsilon, R+M)$.
\end{description}

Indeed, let $R > 0$, and $0 < \epsilon \leq R + M$. Because $\displaystyle\lim_{t \to \infty} \ell(t) = 0$, there exists a time $T := T(\epsilon)$, such that
\begin{equation}
 t \geq T \implies \ell(t) \leq \frac{1}{2} \delta \left( \frac{\epsilon}{2} \right). 
\label{hnai22}
\end{equation}
We consider the region
\[ L_\epsilon := \{ (t, y) \in \mathbb{R}_{+} \times \mathbb{R}_{+} : y \leq \epsilon, \, t \geq T(\epsilon) \}, \]
which is positively invariant. In fact, when
\[ \frac{\epsilon}{2} \leq y(t) \quad \text{and} \quad t \geq T(\epsilon) \]
for some $y \in H(t_0)$, then by (\ref{h2ni2}) and (\ref{hnai22}), we have
\[ \dot{y}(t) \leq -\frac{1}{2} \delta (y) < 0. \]

We next establish that if we define
\begin{equation}
 \widetilde{T}(\epsilon, r) = T(\epsilon) + \frac{2r}{\delta(\epsilon)}, 
 \label{hlokl25}
\end{equation}
then the following is fulfilled:
\begin{equation}
 \text{For all } t \geq t_0 + \widetilde{T}(\epsilon, R+M), \, y \in H(t_0) \text{ and } y(t_0) \leq R \implies (t, y(t)) \in L_\epsilon. 
\label{hklokl25}
\end{equation}
Indeed otherwise, by positive invariance of $L_\epsilon$, there would exist $y \in H(t_0)$ and $y(t_0) \leq R$, such that
\[ (t, y(t)) \notin L_\epsilon \quad \forall t \in [t_0 + T(\epsilon), t_0 + \widetilde{T}(\epsilon, R+M)], \]
and because $t \geq T(\epsilon)$, we would have
\begin{equation}
 y(t) > \epsilon \quad \forall t \in [t_0 + T(\epsilon), t_0 + \widetilde{T}(\epsilon, R+M)]. 
 \label{hn924}
\end{equation}
On the other hand, by (\ref{h2ni2}), (\ref{hnai22}) and (\ref{hn924}), it follows that
\begin{equation}
 \dot{y}(t) \leq -\frac{1}{2} \delta(\epsilon) \quad \forall t \in [t_0 + T(\epsilon), t_0 + \widetilde{T}(\epsilon, R+M)]. 
  \label{23mama48}
\end{equation}

It turns out from (\ref{254mnmnj}), (\ref{hn924}) and (\ref{23mama48}) that
\begin{equation}
 \epsilon < y(t) \leq R + M - \frac{1}{2} (t - t_0 - T(\epsilon)) \delta(\epsilon) \quad \forall t \in [t_0 + T(\epsilon), t_0 + \widetilde{T}(\epsilon, R+M)].
  \label{215nmnai}
\end{equation}

Using (\ref{215nmnai}) and taking into account the definition (\ref{hlokl25}) of $\widetilde{T}(\epsilon, R+M)$, we obtain $\epsilon < y(t_0 + \widetilde{T}) = 0$, which is a contradiction. This establishes (\ref{hklokl25}). From (\ref{hklokl25}) and the positive invariance of $L_\epsilon$, we deduce $(\mathcal{P})$.

In order to establish inequality (\ref{h2ni2}), we exploit $(\mathcal{P})$. We proceed as follows. Define
\begin{equation}
 h(r) = \sqrt{r} \exp(r), 
\label{12mkhmli}
\end{equation}
\begin{equation} 
v(s) = \sup \left\{ \frac{y(t_0 + s)}{h(y(t_0) + M)}, \, y \in H(t_0), \, t_0 \geq 0 \right\}. 
\label{213mkh}
\end{equation}

Because $M > 0$ ($\ell$ is a continuous nonzero function), the denominator in (\ref{213mkh}) is strictly positive and (\ref{254mnmnj}) and (\ref{12mkhmli}) imply that $v(t) \leq 1$ for all $t \geq 0$.

Now, we prove that $\displaystyle\lim_{t \to \infty} v(t) = 0$.

Indeed, let $\epsilon > 0$ and let $\delta_1 := \delta_1(\epsilon), \, \delta_2 := \delta_2(\epsilon)$ be a pair of constants with $0 < \delta_1 < \delta_2$ and being defined in such a way that
\[ s \notin [\delta_1, \delta_2] \implies \frac{s}{h(s)} < \epsilon. \]
Then, by (\ref{254mnmnj}), it follows that
\begin{equation} 
 \frac{y(t)}{h(y(t_0) + M)} < \epsilon \quad \forall t \geq t_0, \, y \in H(t_0) 
\label{15rjnemnm}
\end{equation}
provided that either $y(t_0) + M < \delta_1$ or $y(t_0) + M > \delta_2$.

It remains to consider the case $\delta_1 \leq y(t_0) + M \leq \delta_2$.\\
 By $(\mathcal{P})$, we obtain
\begin{equation} 
 \frac{y(t)}{h(y(t_0) + M)} \leq \frac{y(t)}{h(\delta_1)} \leq \epsilon \quad \forall t \geq t_0 + \widetilde{T}(\epsilon h(\delta_1), \delta_2 + M), \, y \in H(t_0). 
 \label{rhnim13}
\end{equation}

It turns out from (\ref{213mkh}), (\ref{15rjnemnm}) and (\ref{rhnim13}) that
\begin{equation} 
 v(t) \leq \epsilon \quad \forall t \geq \widetilde{T}(\epsilon h(\delta_1), \delta_2 + M). 
 \label{rlom13}
\end{equation}
Because $\epsilon > 0$ is arbitrary, (\ref{rlom13}) asserts that $\displaystyle\lim_{t \to \infty} v(t) = 0$. Therefore, if we define
\[ \widetilde{v}: \mathbb{R}_{+} \to \mathbb{R}_{+} \]
by
\[ \widetilde{v}(t) = \int_{t-1}^t \widetilde{v}(s) \, ds, \quad \text{for } t \geq 0, \]
where
\[ \widetilde{v}(t) := \sup_{s \geq t} v(s), \quad \forall t \geq 0 \]
and
\[ \widetilde{v}(t) := \widetilde{v}(0) \quad \forall t < 0 \]
which is a decreasing function and thus measurable. So, $\widetilde{v}$ is a continuous decreasing function with $v(t) \leq \widetilde{v}(t)$ for all $t \geq 0$ and
\[ \displaystyle\lim_{t \to \infty} \widetilde{v}(t) = 0. \]
Hence, by considering $\beta(s, t) = h(s) \widetilde{v}(t)$, (\ref{hjj102}) is fulfilled for all $y \in H(t_0)$ with $t_0 \geq 0$.

Finally, it is obvious from Barbalat's lemma (\cite{hammami2001stability}) that if $\ell$ is uniformly continuous then $\displaystyle\lim_{t \to \infty} \ell(t) = 0$ which makes us able to deduce (\ref{hjj102}) in this particular case.

\begin{rem}
If $\ell = 0$, we can take $\beta = \sigma$ and the proof is finished.
\end{rem}
\end{proof}

\begin{lem}\label{prop3vvv} (\cite[Theorem 9]{hagood2006recovering})
Let $\tilde{t}\in(t_0,\infty)$ with $t_0\geq 0$ and let $y:[t_0,\tilde{t})\to \mathbb{R}_+$ be a continuous function satisfying the following scalar differential inequality:
$$D^{+}y(t)\leq \mu(t)y(t)+v(t)\quad \forall t\in (t_0,\tilde{t}),$$
where $\mu,v\in C(\mathbb{R}_+,\mathbb{R}).$ Then, for all $t\in [t_0,\tilde{t}),$ we have
$$y(t)\leq y(t_0)e^{\int_{t_0}^{t}\mu(s)ds}+\int_{t_0}^{t}v(s)e^{\int_{s}^{t}\mu(\tau)d\tau}ds.$$
\end{lem}

\chapter*{Appendix B: Semigroups and Their Stability}
\label{appendix:B}
\markboth{Appendix B: Semigroups}{Appendix B: Stability}
\addcontentsline{toc}{chapter}{Appendix B: Semigroups and Their Stability}
\section*{B.1 Semigroup theory}
\renewcommand{\thedfn}{B.\arabic{dfn}}
\setcounter{dfn}{0}
\renewcommand{\thelem}{B.\arabic{lem}}
\setcounter{lem}{0}
\renewcommand{\thethm}{B.\arabic{thm}}
\setcounter{thm}{0}
\renewcommand{\theprop}{B.\arabic{prop}}
\setcounter{prop}{0}
In this section, we introduce basic definitions and state known results from semigroup theory.
Let $X$ denote a real or complex Banach space with the norm $\|\cdot\|_{X}.$
\begin{dfn}(Semi-group)
The operator value function $T(t):\mathbb{R}_+ \to L(X)$ is called uniformly continuous semigroup if satisfies:
\begin{enumerate}
\item[$(i)$] $T(0)=I.$
\item[$(ii)$] $T(t+s)=T(t)T(s) \;\forall s,t\geq0.$
\item[$(iii)$]  $\displaystyle\lim_{t\longrightarrow 0}\|T(t)-I\|=0.$
\end{enumerate}
\end{dfn}
The linear operator $A$ is defined by
$$D(A)=\{x\in X: \displaystyle\lim_{t\longrightarrow 0}\frac{T(t)x-x}{t}\; \;exists \; \;in \; \;X\}.$$
and $$Ax=\displaystyle\lim_{t\longrightarrow 0}\frac{T(t)x-x}{t} \;\;\hbox{for} \; \;x\in \; \;D(A)$$
is the infinitesimal generator of the semigroup $(T(t))_{t\geq0}.$
\begin{thm}(\cite{Paz83})
A linear operator $A$ is the infinitesimal generator of a uniformly continuous semigroup if and only if $A$ is a bounded linear operator.
\end{thm}
\begin{thm}(\cite{Paz83})
Let $T(t)$ and $S(t)$ be $C_0-$semigroups of bounded linear operators with infinitesimal generators $A$ and $B$ respectively. If $A=B,$ then $T(t)=S(t)$ for $t\geq 0.$
\end{thm}
\renewcommand{\thecor}{B.\arabic{cor}}
\setcounter{cor}{0}
\begin{cor}(\cite{Paz83})
Let $(T(t))_{t\geq0}$ be a uniformly continuous semigroup of bounded linear operators. Then,
\begin{enumerate}
\item[$(a)$] There exists a constant $\omega\geq0,$ such that $\|T(t)\|\leq e^{\omega t},$ for all $t\geq0.$
\item[$(b)$] There exists a unique bounded linear operator $A,$ such that $$T(t)=e^{tA}=\displaystyle\sum_{n=0}^{\infty} \frac{(tA)^{n}}{n!}, \ for \ all \ t\geq0.$$
\item[$(c)$] The operator $A$ in part $(b)$ is the infinitesimal generator of $(T(t))_{t\geq0}.$
\item[$(d)$] $t \to T(t)$ is differentiable in norm and $$\frac{dT(t)}{dt}=AT(t)=T(t)A.$$
\end{enumerate}
\end{cor}
\begin{dfn}
A semigroup $(T(t))_{t\geq0},$ of bounded linear operators on $X$ is a strongly continuous semigroup of bounded linear operators if for all $x\in X$
$$\displaystyle\lim_{t\longrightarrow 0}\|T(t)x-x\|=0.$$
\end{dfn}
A strongly continuous semigroup of bounded linear operators on $X$ will be called a semigroup of class $C_0$ or simply a $C_0-$semigroup.
\begin{thm}(\cite{Paz83})\label{th1}
Let $T(t)$ be a $C_0$-semigroup. $(T(t))_{t\geq 0}$ is uniformly exponentially stable if there exist constants $\omega\geq0$ and $M\geq 1,$ such that
$$\|T(t)\|\leq Me^{\omega t} \;\;\hbox{for} \; \; 0\leq t< \infty.$$
\end{thm}
\begin{dfn}
$(T(t))_{t\geq 0}$ is uniformly asymptotically stable if $\|T(t)\| \to 0,\quad t\to \infty.$
\end{dfn}
\begin{dfn}
  A strongly continuous semigroup $(T(t))_{t\geq0}$ is called
\begin{enumerate}
    \item Exponentially stable, if there exists $\lambda > 0$ such that $\displaystyle\lim_{t \to \infty} e^{\lambda t} \|T(t)\| = 0$.
    \item Uniformly stable, if $\displaystyle\lim_{t \to \infty} \|T(t)\| = 0$.
    \item Strongly stable, if $\displaystyle\lim_{t \to \infty} \|T(t)x\|_X = 0 \ \forall \ x \in X$.
\end{enumerate}
\end{dfn}
Note that strong stability of a semigroup $(T(t))_{t\geq0}$ is what we call attractivity of a corresponding dynamical system $\Sigma := (X, \{0\}, \phi)$.
\begin{lem} (Proposition 1.2, p. 296 in \cite{engel2000one}). A $C_0$-semigroup is uniformly stable if and only if it is exponentially stable.
\end{lem}
Uniform stability implies strong stability, but the converse is not generally true.\\
We need methods for checking the exponential stability of $C_0$-semigroups.

\begin{lem} Let $\omega_0 := \inf_{t>0} \left( \frac{1}{t} \log \|T(t)\| \right)$ be well-defined. Then for all $\omega > \omega_0$, there exists $M_\omega$ such that $\|T(t)\| \leq M_\omega e^{\omega t}$.
\end{lem}
\begin{dfn} 
The constant $\omega_0$ from the previous lemma is called the growth bound of a $C_0$-semigroup.
\end{dfn}

Denote by $\Re(\lambda)$ the real part of a complex number $\lambda$.

\begin{dfn}  Let $(T(t))_{t\geq0}$ be a $C_0$-semigroup and $A$ be its generator. If $\omega_0 = \sup_{\lambda \in \text{Spec}(A)} \Re(\lambda)$, then we say that $(T(t))_{t\geq0}$ satisfies the spectrum-determined growth assumption.
\end{dfn}
\begin{cor}(\cite{Paz83})
If $(T(t))_{t\geq0}$ is a $C_0$-semigroup, then for every $x\in X,\; t \to T(t)x$ is a continuous function $\R_+$ into $X.$
\end{cor}
\begin{cor}(\cite{Paz83})
If $A$ is the infinitesimal generator of a $C_0-$semigroup $(T(t))_{t\geq0},$ then $D(A)$ is dense in $X$ and $A$ is a closed linear operator.
\end{cor}

\subsection*{B.1.1 Linear systems}
We consider now a linear system of the form:
\renewcommand{\theequation}{B.\arabic{equation}}
\setcounter{equation}{0}
\begin{equation}\label{2ff}
\dot{x}=Ax,\qquad x(0)=x_0\in X,\qquad\ t\ge0,
\end{equation}
where $A$ is a linear operator generating a strongly continuous semigroup $(T(t))_{t\geq0}.$\\
The general solution $x(t,x_0)$ of system \eqref{2ff} is given by
\renewcommand{\theequation}{B.\arabic{equation}}
\begin{equation}\label{3}
x(t,x_0)=T(t)x_0,\qquad t\ge0,\quad x_0\in D(A),
\end{equation}
\begin{dfn}
The operator $A:D(A)\subset X \to X$ is a closed linear operator with $cl(D(A))=X,$ generates a strongly continuous semigroup on $X,$ $(T(t))_{t\geq0},$ exponentially stable, that is, there exist $M\geq 1$ and $\alpha >0,$ such that $$\|T(t)\|\leq M e^{-\alpha t} \quad \forall t\in \mathbb{R}^{+}.$$
\end{dfn}

\begin{dfn}
System \eqref{2ff} is exponentially stable if for every $x_0\in D(A),$ there are
numbers $M>0$ and $\alpha> 0,$ such that
$$\|T(t)x_0\|\leq Me^{-\alpha t}\|x_0\| \qquad \forall t\ge0.$$
\end{dfn}

\begin{prop}(\cite{Paz83})
Let $X$ be a Banach space. $A$ is the generator of the semigroup $(T(t))_{t\geq0}.$ The following conditions are equivalent:
\begin{enumerate}
\item[$(i)$] System \eqref{2ff} is exponentially stable.
\item[$(ii)$] $A$ is exponentially stable.
\item[$(iii)$]  For all $x_0\in D(A):$ $\displaystyle\int _{0}^{+\infty}\|x(t,x_0)\|^{2}dt<+\infty.$
\end{enumerate}
\end{prop}

\subsection*{B.1.2 Existence of solution of semi-linear system}
We consider the semilinear initial value problem:
\begin{equation}\label{R1dhanen}
\left\lbrace
\begin{array}{l}
\dot{x}(t)=Ax(t)+f(t,x(t)),\qquad t\ge t_0\ge0,\\
\\
x(t_0)=x_0,
\end{array}\right.
\end{equation}
where $x(t)\in X$ and $A$ is the infinitesimal generator of a $C_0-$semigroup $(S(t))_{t\geq0}$ on a Banach space $X,$ $t_0\geq0$ is the initial time, $x_0\in X$ is the initial condition and the function $f:\mathbb{R}_{+}\times X\times U \to X$ is a nonlinear operator.\\
\begin{thm}\label{teoyasmine}(\cite{Paz83})
Let $f:[t_0,T]\times X \to X$ be continuous in $t$ on $[t_0,T]$ and uniformly Lipschitz continuous on $X.$ If $A$ is the infinitesimal generator of a $C_0-$semigroup $(S(t))_{t\geq0},$ on $X,$ then for every $x_0\in X$ the initial value problem (\ref{R1dhanen}) has a unique mild solution $x\in C([t_0,T],X).$
\end{thm}

\begin{proof} For a given $u_{0}\in X$ we define a mapping
$$F : C([t_{0},T]: X)\longrightarrow C([t_{0},T]: X),$$
by
\begin{eqnarray}\label{eq3}
(Fu)(t)= S(t-t_{0})u_{0}+\int_{t_0}^{t}S(t-s)f(s,x(s))ds,\quad  t_{0} \leq t\leq T.
\end{eqnarray}
Denoting by $\Vert u\Vert_{\infty}$ the norm of $u$ as an element of $C([t_{0},T]) : X)$ it follows readily from the definition of $F$ that
 \begin{eqnarray}\label{eq4}
 \Vert (Fu)(t)-(F v)(t)\Vert \leq M L(t-t_{0})\Vert u-v\Vert_{\infty},
 \end{eqnarray}
where $M$ is a bound of $\Vert S(t)\Vert$ on $[t_{0},T].$ Using $\eqref{eq3}$, $\eqref{eq4}$ and induction on $n$ it follows easily that
\begin{eqnarray}
\Vert (F^{n}u)(t)-(F^{n}v)(t)\Vert \leq\frac{(M L(t-t_{0}))^{n}}{n!}\Vert u-v\Vert_{\infty}.
 \end{eqnarray}
Hence,
\begin{eqnarray}
\Vert (F^{n}u - F^{n}v)\Vert \leq \frac{(M L T)^{n}}{n!}\Vert u-v\Vert_{\infty}
 \end{eqnarray}
For $n$ large enough $(MLT)^{n}/n! < 1$ and by a well-known extension of the contraction principle $F$ has a unique fixed point is the desired solution of the integral equation (\ref{eq3}).
The uniqueness of $u$ and the Lipschitz continuity of the map $u_{0}\longrightarrow u$ are consequences of the following argument. Let $v$ be a mild solution of \eqref{R1dhanen} on $[t_{0},T]$ with the initial value $v_{_{0}}$. Then,
\begin{align}
 \Vert u(t)- v(t)\Vert & \leq \Vert S(t-t_{0})u_{0}-S(t-t_{0}v_{0})\Vert + \int_{t_0}^{t}\Vert S(t-s)(f(s,u(s))-f(s,v(s))ds   \nonumber \\
 &\leq  M \Vert u_{0}-v_{0}\Vert + ML \int_{t_0}^{t}\Vert u(s)- v(s)\Vert ds.
\end{align}
This implies, by Gronwall's inequality, that
$$\Vert u(t)-v(t)\Vert \leq Me^{ML(T-t_{0})}\Vert u_{0}-v_{0}\Vert,$$
and therefore
$$\Vert u - v\Vert_{\infty} \leq Me^{ML(T-t_{0})}\Vert u_{0}-v_{0}\Vert.$$
Which yields both the uniqueness of $u$ and the Lipschitz continuity of the map $u_{0}\longrightarrow u$.
It is not difficult to see that if $g\in C([t_{0},t]:X)$ and in the proof of theorem \ref{teoyasmine} modify the definition of $F$ to
$$(Fu)(t)= g(t)+\int_{t_0}^{t}S(t-s)f(s,u(s))ds,$$
\end{proof}
we obtain the following slightly more general result.
\begin{cor}\cite[corollary 1.3]{Paz83}
If $A$ and $f$ satisfy the conditions of theorem \ref{teoyasmine}, then for every $g\in C([t_{0},T]:X)$ the integral equation
\begin{eqnarray}
w(t)= g(t)+\int_{t_0}^{t}S(t-s)f(s,w(s))ds
\end{eqnarray}
has a unique solution $w\in C([t_{0},T]:X)$.
\end{cor}

The uniform Lipschitz condition of the function $f$ in Theorem \ref{teoyasmine} assures the existence of a global (i.e. defined on all of $[t_0, T]$) mild solution of (\ref{R1dhanen}). If we assume that $f$ satisfies only a local Lipschitz condition in $x,$ uniformly in $t$ on bounded intervals, that is, for every $\tilde{t}\geq0$ and constant $c\geq 0,$ there is a constant $M(c,\tilde{t}),$ such that
\begin{eqnarray}\label{eq7}
\|f(t,x)-f(t,y)\|\leq  M(c,\tilde{t}) \|x-y\|
\end{eqnarray}
holds for all $x,y\in X$ with $\|x\|\leq c,$ and $\|y\|\leq c$ and $t\in [0,\tilde{t}],$ then we have the following local version of Theorem \ref{teoyasmine}.

\addcontentsline{toc}{chapter}{Bibliography}
\bibliographystyle{abbrv}
\bibliography{Mir_LitList_NoMir,MyPublications}
\end{document}

%% file: Commands_and_Macro.tex
\makeatletter
\DeclareOldFontCommand{\rm}{\normalfont\rmfamily}{\mathrm}
\DeclareOldFontCommand{\sf}{\normalfont\sffamily}{\mathsf}
\DeclareOldFontCommand{\tt}{\normalfont\ttfamily}{\mathtt}
\DeclareOldFontCommand{\bf}{\normalfont\bfseries}{\mathbf}
\DeclareOldFontCommand{\it}{\normalfont\itshape}{\mathit}
\DeclareOldFontCommand{\sl}{\normalfont\slshape}{\@nomath\sl}
\DeclareOldFontCommand{\sc}{\normalfont\scshape}{\@nomath\sc}
\makeatother

\usepackage{csquotes}  

\usepackage{physics}   

\newcommand \N   {\mathbb{N}}
\newcommand \R   {\mathbb{R}}

\newcommand{\Uc}{\ensuremath{\mathcal{U}}}

\newcommand{\vertiii}[1]{{\left\vert\kern-0.25ex\left\vert\kern-0.25ex\left\vert #1 
    \right\vert\kern-0.25ex\right\vert\kern-0.25ex\right\vert}}

\newcommand \qrq   {\quad\Rightarrow\quad}

\newcommand{\normt}[1]{{\left\vert\kern-0.25ex\left\vert\kern-0.25ex\left\vert #1 
		\right\vert\kern-0.25ex\right\vert\kern-0.25ex\right\vert}}

\newcommand{\intt}{{\rm int}\,}

\newif\ifMath					
\newif\ifEngi					

\newif\ifAndo              
													
\newif\ifExercises					
\newif\ifSolutions          
\newif\ifGerman							
\newif\ifEnglish						

\newif\ifnothabil						

\newif\ifFuture							

\newif\ifConf                    
\newif\ifJournal								 

\newif\ifNOTFORBOOK
\newif\ifFullVersion
\newif\ifExludedDueToSpaceReasons

\usepackage{xifthen}

\newcommand{\einsnorm}[2]{\ensuremath{
    \!\!\;\!\!\!\;
    \left\bracevert\!\!\!\!\!\left\bracevert
    \!
		\ifthenelse{\isempty{#2}}{#1}{#1(#2)}
    \!
      \right\bracevert\!\!\!\!\!\right\bracevert
    \!\!\;\!\!\!\;
  }}

\usepackage{xcolor}
\definecolor{blond}{rgb}{0.98, 0.94, 0.75}
	
\newlength\mytemplen
\newsavebox\mytempbox

\makeatletter
\newcommand\mybluebox{%
    \@ifnextchar[
       {\@mybluebox}%
       {\@mybluebox[0pt]}}

\def\@mybluebox[#1]{%
    \@ifnextchar[
       {\@@mybluebox[#1]}%
       {\@@mybluebox[#1][0pt]}}

\def\@@mybluebox[#1][#2]#3{
    \sbox\mytempbox{#3}%
    \mytemplen\ht\mytempbox
    \advance\mytemplen #1\relax
    \ht\mytempbox\mytemplen
    \mytemplen\dp\mytempbox
    \advance\mytemplen #2\relax
    \dp\mytempbox\mytemplen
    \colorbox{blond}{\hspace{1em}\usebox{\mytempbox}\hspace{1em}}}

\makeatother

\makeatletter
\let\origd=\d
\renewcommand*\d{
  \relax\ifmmode
    \mathrm{d}%
  \else
    \expandafter\origd
  \fi
}\makeatother

\usepackage{mathtools}

\makeatletter 
\newcommand{\pushright}[1]{\ifmeasuring@#1\else\omit\hfill$\displaystyle#1$\fi\ignorespaces}
\newcommand{\pushleft}[1]{\ifmeasuring@#1\else\omit$\displaystyle#1$\hfill\fi\ignorespaces}
\makeatother 

\newcounter{syscounter}

\newcounter{WPcounter}